\documentclass[11pt,a4paper]{amsart}
\usepackage{amssymb,amsmath,amsthm}

\usepackage[colorlinks=true, pdfstartview=FitV, linkcolor=blue, citecolor=blue, urlcolor=blue,pagebackref=false]{hyperref}
\usepackage{esint}
\usepackage{mathrsfs}

\usepackage{times,latexsym,color}

\newcommand{\BOX}{\ensuremath\Box}

\newtheorem{theorem}{Theorem}
\newtheorem{proposition}{Proposition}
\newtheorem{lemma}[proposition]{Lemma}

\theoremstyle{remark}
\newtheorem{remark}[proposition]{Remark}

\theoremstyle{definition}

\DeclareMathOperator{\supp}{supp}

\newcommand{\N}{\mathbb{N}}

\newcommand{\R}{\mathbb{R}}
\newcommand{\Z}{\mathbb{Z}}
\newcommand{\dd}{\,{\rm d}}

\newcommand{\ep}{\varepsilon}

\def\Xint#1{\mathchoice
{\XXint\displaystyle\textstyle{#1}}%
{\XXint\textstyle\scriptstyle{#1}}%
{\XXint\scriptstyle\scriptscriptstyle{#1}}%
{\XXint\scriptscriptstyle\scriptscriptstyle{#1}}%
\!\int}
\def\XXint#1#2#3{{\setbox0=\hbox{$#1{#2#3}{\int}$}
\vcenter{\hbox{$#2#3$}}\kern-.5\wd0}}

\def\dashint{\Xint-}

\definecolor{darkgreen}{rgb}{0,0.5,0}
\definecolor{darkblue}{rgb}{0,0,0.7}
\definecolor{darkred}{rgb}{0.9,0.1,0.1}
\definecolor{lightblue}{rgb}{0,0.51,1}

\newenvironment{proofx}[1]%
{\vskip\baselineskip\noindent\textbf{Proof of {#1}:}}%
{\hspace*{.1pt}\hspace*{\fill}\BOX\vskip\baselineskip}
{\vskip\baselineskip\noindent\textbf{Proof of Theorem \protect\ref{#1}:}}%
{\hspace*{.1pt}\hspace*{\fill}\BOX\vskip\baselineskip}
{\vskip\baselineskip\noindent\textbf{Proof of Lemma~\protect\ref{#1}:}}%
{\hspace*{.1pt}\hspace*{\fill}\BOX\vskip\baselineskip}
{\vskip\baselineskip\noindent\textbf{Proof of Proposition~\protect\ref{#1}:}}%
{\hspace*{.1pt}\hspace*{\fill}\BOX\vskip\baselineskip}
{\vskip\baselineskip\noindent\textbf{Proof of Theorems \protect\ref{#1} --
\protect\ref{#2}:}}%
{\hspace*{.1pt}\hspace*{\fill}\BOX\vskip\baselineskip}

\begin{document}

\title[Regularity for stationary flows over bumpy boundaries]{Regularity for the stationary Navier-Stokes equations over bumpy boundaries and a local wall law}

\author[M. Higaki]{Mitsuo Higaki$^\ast$}
\address[M. Higaki]{Kobe University, Kobe, Japan}
\email{higaki@math.kobe-u.ac.jp}

\author[C. Prange]{Christophe Prange}
\address[C. Prange]{Universit\'e de Bordeaux, CNRS, UMR [5251], IMB, Bordeaux, France}
\email{christophe.prange@math.u-bordeaux.fr}

\subjclass[2010]{}
\date{\today}

\maketitle

\noindent {\bf Abstract.}\, We investigate regularity estimates for the stationary Navier-Stokes equations above a highly oscillating Lipschitz boundary with the no-slip boundary condition. Our main result is an improved Lipschitz regularity estimate at scales larger than the boundary layer thickness. We also obtain an improved $C^{1,\mu}$ estimate and identify the building blocks of the regularity theory, dubbed `Navier polynomials'. In the case when some structure is assumed on the oscillations of the  boundary, for instance periodicity, these estimates can be seen as local error estimates. Although we handle the regularity of the nonlinear stationary Navier-Stokes equations, our results do not require any smallness assumption on the solutions.

\vspace{0.3cm}


\vspace{0.3cm}


\footnote[0]{$^\ast$Part of this work was done while the first author was a postdoctoral researcher at Universit\'e de Bordeaux.}

\small

\section{Introduction}

This paper is concerned with the local regularity of viscous incompressible fluid flows above rough bumpy boundaries $x_3>\ep\gamma(x'/\ep)$ with $\gamma$ Lipschitz and the no-slip boundary condition. Although bumpy boundaries  have a complicated geometry and low regularity, the flow may paradoxically be better behaved than for smooth or flat boundaries. It is well documented in the physical \cite{Jim04,Sch16} and the mathematical \cite{JM2,M1,GM,H} literature that roughness favors slip of the fluid on the boundary in certain regimes. In the striking paper \cite{FTBC06} it is even showed experimentally that roughness may delay the transition to turbulence. This also supports the idea that the vanishing viscosity limit from Navier-Stokes to Euler may be less singular above highly oscillating boundaries than above flat ones \cite{IS11,GLNR18,Ngu19}. 

Our goal is to investigate these effects, such as the enhanced slip, or the delay of the transition to turbulence, from the point of view of the regularity theory. Due in particular to vorticity creation at the boundary, the boundary regularity of fluid flows with the no-slip boundary conditions is delicate. In the nonstationary case, it is for instance not known whether there is an analogue of Constantin and Fefferman's \cite{CF} celebrated geometric regularity criteria for supercritical blow-up scenarios. For perfect slip or Navier-slip boundary conditions on the contrary, the situation is brighter. In particular an extension of the criteria of \cite{CF} is known in this case; see the work \cite{BDVB09} by Beir{\~{a}}o da Veiga and Berselli and \cite{Li19} by Li. We expect that fluids over bumpy boundaries have an intermediate behavior between these two extreme no-slip and (full-)slip situations, especially as far as the mesoscopic regularity properties are concerned.

Our approach grounds on the use of asymptotic analysis to prove regularity estimates. The success of such methods to prove the regularity to certain Partial Differential Equations is spectacular. One of the striking examples is that of homogenization. The basic idea is that the large-scale regularity is determined by the macroscopic properties of the systems, i.e. in the homogenization limit, while the small-scale regularity is determined by the regularity of the data (coefficients, boundary). Two approaches were developed: (a) blow-up and compactness arguments in periodic homo\-genization in the wake of the pioneering works \cite{AL1,AL2}, (b) quantitative arguments based on suboptimal local error estimates as developed for periodic homogenization \cite{Shen15,GS19,NSX19}, almost periodic homogenization \cite{AS16}, and stochastic homogenization \cite{GNO14,AKM19}.

In this work, we focus on the regularity for stationary problems. We consider the three-dimensional stationary Navier-Stokes equations
\begin{equation}\tag{NS$^\ep$}\label{intro.NS.ep}
\left\{
\begin{array}{ll}
-\Delta u^\ep+\nabla p^\ep=-u^\ep\cdot\nabla u^\ep&\mbox{in}\ B^\ep_{1,+}(0)\\
\nabla\cdot u^\ep=0&\mbox{in}\ B^\ep_{1,+}(0)\\
u^\ep=0&\mbox{on}\ \Gamma^\ep_1(0)\,,
\end{array}
\right.
\end{equation}
where the functions $u^\ep=u^\ep(x)=(u^\ep_1(x), u^\ep_2(x), u^\ep_3(x))^\top\in\R^3$ and $p^\ep=p^\ep(x)\in\R$ denote respectively the velocity field and the pressure field of the fluid. We have set for $\ep\in(0,1]$ and $r\in(0,1]$,
\begin{align}\label{intro.def1}
\begin{split}
&B^\ep_{r,+}(0) = \{x\in \R^3~|~x'\in(-r,r)^2\,, \ \ \ep\gamma(\frac{x'}{\ep})<x_3< \ep\gamma(\frac{x'}{\ep}) + r \}\,, \\
&\Gamma^\ep_r(0) = \{x\in \R^3~|~x'\in(-r,r)^2\,, \ \ x_3 = \ep\gamma(\frac{x'}{\ep})\}\,.
\end{split}
\end{align}
The boundary function $\gamma\in W^{1,\infty}(\R^2)$ is assumed to satisfy $\gamma(x')\in(-1,0)$ for all $x'\in\R^2$.

Our use of compactness arguments to tackle the regularity for solutions of \eqref{intro.NS.ep} is reminiscent of the pioneering work of Avellaneda and Lin \cite{AL1,AL2} in homogenization, and of the works by G\'{e}rard-Varet \cite{G1}, Gu and Shen \cite{GS15}, and Kenig and Prange \cite{KP1,KP2}. 
We separate the small-scale regularity, i.e. at scales $\lesssim \ep$, from the mescopic- or large-scale regularity, i.e. at scales $\ep\lesssim r\leq 1$. Concerning the small scales, the classical Schauder regularity theory for the Stokes and the Navier-Stokes equations was started by Lady\v{z}enskaja \cite{L} using potential theo\-ry and by Giaquinta and Modica \cite{GiMo} using Campanato spaces. These classical estimates require some smoothness of the boundary and typically depend on the modulus of continuity of $\nabla \gamma$ when the boundary is given by $x_3=\gamma(x')$. Therefore, these estimates degenerate for highly oscillating boundaries $x_3=\ep\gamma(x'/\ep)$ with sufficiently small $\ep\in(0,1)$. As for the large scales, on the contrary, the regularity is inherited from the limit system when $\ep\rightarrow 0$ posed in a domain with a flat boundary. Here no regularity is needed for the original boundary, beyond the boundedness of $\gamma$ and of its gradient. The mechanism for the regularity at small scales and at large scales is hence completely different.  Moreover, it is possible to prove, at the large scales, improved estimates that are known to be false at the small scales. An example of this is our large-scale Lipschitz estimate of Theorem \ref{theo.lip.nonlinear} below that is known to be false over a Lipschitz graph at the small scales even in the case of a linear elliptic operator \cite{KS11,KS11b,Shen15}.

Beyond improved regularity estimates, our objective is to develop local error estimates for the homogenization of viscous incompressible fluids over bumpy boundaries and derive local wall laws. 
The wall law catches an averaged effect from the $O(\ep)$-scale on large scale flows of order $O(1)$ through homogenization. In the wall law, a rough boundary is modeled as a smooth one and an appropriate condition is imposed on it reflecting the roughness of the original boundary. In typical situations, this process gives a Navier-type condition with slip length of $O(\ep)$, the so-called Navier wall law. This effective boundary condition reads for instance in two dimensions
\begin{equation}\label{e.walllaw}
u_1=\ep\alpha\partial_2u_1\,,\quad u_2=0\quad\mbox{on}\quad \partial\R^2_+
\end{equation}
with a constant $\alpha$ depending only on the boundary function $\gamma$. We now briefly review the literature concerned with the derivation of wall laws such as \eqref{e.walllaw} and the proof of error estimates in the global setting. The literature is vast and it is impossible to be exhaustive here. The wall law for simple stationary shear flows is analyzed in the pioneering work J\"{a}ger and Mikeli\'{c} \cite{JM1} when the boundary is periodic. This result is extended to a random setting by G\'erard-Varet \cite{G1} and to the almost periodic setting by G\'{e}rard-Varet and Masmoudi \cite{GM}. Nonstationary cases are studied in Mikeli\'{c}, Ne\u{c}asov\'{a}, and Neuss-Radu \cite{MNN} under the assumption that the limit flows are space-time $C^2$ functions. The strong regularity condition in \cite{MNN} implies that a careful analysis is needed when we study Initial Boundary Value Problems (IBVPs). Indeed, for these cases, no matter how regular the initial data are, there is the loss of regularity of solutions due to the boundary compatibility condition. Higaki \cite{H} considers an IBVP in a bumpy half-space and verifies the Navier wall law for $C^1$ initial data under natural compatibility conditions. A key ingredient is to make use of the $L^\infty$-regularity theory of the Navier-Stokes equations in the half-spaces; see Abe and Giga \cite{AG} for the analyticity of the Stokes semigroup in the $L^\infty$-type spaces. Theorem \ref{theo.lip.nonlinear.periodic} below provides a local counterpart of these global error estimates in the case of the stationary Navier-Stokes equations.

\subsection*{Outline and novelty of our results}

Our main results are given in the two theorems below. In Theorem \ref{theo.lip.nonlinear} we state a uniform Lipschitz estimate. In Theorem \ref{theo.lip.nonlinear.periodic} we give a local error estimate and identify the building blocks of the regularity theory. Both results hold for weak solutions of the nonlinear equations \eqref{intro.NS.ep} and hold without any smallness assumption on the size of the solutions.
%
\begin{theorem}[mesoscopic Lipschitz estimate]\label{theo.lip.nonlinear}
For all $M\in(0,\infty)$, there exists a constant $\ep^{(1)} \in(0,1)$ depending on $\|\gamma\|_{W^{1,\infty}(\R^2)}$ and $M$ such that the following statement holds. For all $\ep\in(0, \ep^{(1)}]$ and $r\in[\ep/\ep^{(1)}, 1]$, any weak solution $u^\ep \in H^1(B^\ep_{1,+}(0))^3$ to \eqref{intro.NS.ep} with
\begin{align}\label{est1.theo.lip.nonlinear}
\bigg(\dashint_{B^\ep_{1,+}(0)} |u^\ep|^2 \bigg)^\frac12 
\le
M
\end{align}
satisfies
\begin{align}\label{est2.theo.lip.nonlinear}
\bigg(\dashint_{B^\ep_{r,+}(0)} |u^\ep|^2 \bigg)^\frac12 
\le
C^{(1)}_M r\,,
\end{align}
where the constant $C^{(1)}_M$ is independent of $\ep$ and $r$, and depends on $\|\gamma\|_{W^{1,\infty}(\R^2)}$ and $M$. Moreover, $C^{(1)}_M$ is a monotone increasing function of $M$ and converges to zero as $M$ goes to zero.
\end{theorem}
%
\begin{remark}\label{rem.theo.lip.periodic1}
(i) By using the Caccioppoli inequality in Appendix \ref{appendix.Caccioppoli}, one can easily prove
\begin{align*}
\bigg(\dashint_{B^\ep_{r,+}(0)} |\nabla u^\ep|^2 \bigg)^\frac12 
\le
\widetilde{C^{(1)}_M}
\end{align*}
for $r\in[\ep/\ep^{(1)}, \frac12]$. Here the constant $\widetilde{C^{(1)}_M}$ satisfies the same property as $C^{(1)}_M$. \\
(ii) In the paper \cite{G1}, G\'erard-Varet obtains a uniform H\"older estimate 
for weak solutions of the Stokes equations when $\gamma\in C^{1,\omega}(\R^2)$ for a fixed modulus of continuity $\omega$. Let us emphasize that there is a gap in difficulty between the uniform H\"older estimate (right-hand side of \eqref{est2.theo.lip.nonlinear} replaced by $Cr^\mu$ with $\mu\in(0,1)$) and the uniform Lipschitz estimate \eqref{est2.theo.lip.nonlinear}. Indeed the Lipschitz estimate requires the analysis of the boundary layer corrector. Moreover, let us emphasize that the Lipschitz estimate is the best that can be proved for $u^\ep$ uniformly in $\ep$. This comment does not contradict the uniform $C^{1,\mu}$ estimate below. Indeed the estimate in Theorem \ref{theo.lip.nonlinear.periodic} is a measure of the oscillation between $u^\ep$ and affine functions, and is not an estimate for $u^\ep$ directly.\\
(iii) As in the works \cite{AL1,G1,KP1} one can combine the mesoscopic regularity estimate with the classical regularity, provided the boundary is regular enough, i.e. when $\nabla\gamma$ is H\"older continuous. In that case, we can prove the full Lipschitz estimate $\|\nabla u^\ep\|_{L^\infty(B^\ep_{1,+}(0))}$ for \eqref{intro.NS.ep}. However, one cannot expect such an estimate to hold in Lipschitz domains even for the Laplace equation with the Dirichlet boundary condition.\\ 
(iv) There is a version of Theorem \ref{theo.lip.nonlinear} for the linear Stokes equations; see Theorem \ref{theo.lip} in Section \ref{sec.linear} below. An important application of such uniform Lipschitz estimates is for estimating the Green and Poisson kernels associated to the Stokes equations in the Lipschitz half-space $\{y_3>\gamma(y')\}$. Following \cite{AL1,AL3}, such estimates were proved for elliptic systems in bumpy domains in \cite{KP1}. Such estimates play a crucial role for the homogenization of boundary layer correctors, in particular in the works \cite{GM2,AKMP17,SZ18}.
\end{remark}
%

Next let us state the result which gives a local justification of the Navier wall law. 
The following theorem is concerned with the polynomial approximation of weak solutions to \eqref{intro.NS.ep} at mesoscopic scales. Remark \ref{rem.theo.lip.periodic2} below states consequences of the estimates in the theorem and Remark \ref{rem.theo.lip.periodic3} establishes the connection between our theorem and the Navier wall law.
%
\begin{theorem}[polynomial approximation]\label{theo.lip.nonlinear.periodic}
Fix $M\in(0,\infty)$ and $\mu\in(0,1)$. Then there exists a constant $\ep^{(2)}\in(0,1)$ depending on $\|\gamma\|_{W^{1,\infty}(\R^2)}$, $M$, and $\mu$ such that for all weak solutions $u^\ep \in H^1(B^\ep_{1,+}(0))^3$ to \eqref{intro.NS.ep} satisfying the bound \eqref{est1.theo.lip.nonlinear}, the following statements hold. 

\noindent {\rm (i)} 
For all $\ep\in(0,\ep^{(2)}]$ and $r\in[\ep/\ep^{(2)}, 1]$, we have
\begin{align}\label{est1.theo.lip.nonlinear.periodic}
\begin{split}
& \bigg(\dashint_{B^\ep_{r,+}(0)} 
\big|u^\ep(x) - \sum_{j=1}^{2} c^\ep_{r,j} x_3 {\bf e}_j\big|^2 \dd x \bigg)^\frac12 
\le
C^{(2)}_M (r^{1+\mu} + \ep^\frac12 r^\frac12)\,,
\end{split}
\end{align}
where the coefficient $c^\ep_{r,j}$, $j\in\{1,2\}$, is a functional of $u^\ep$ depending on $\ep$, $r$, $\|\gamma\|_{W^{1,\infty}(\R^2)}$, $M$, and $\mu$, while the constant $C^{(2)}_M$ is independent of $\ep$ and $r$, and depends on $\|\gamma\|_{W^{1,\infty}(\R^2)}$, $M$, and $\mu$. 

\noindent {\rm (ii)} We assume in addition that $\gamma\in W^{1,\infty}(\R^2)$ is $2\pi$-periodic in each variable. Then there exists a constant vector field $\alpha^{(j)} = (\alpha^{(j)}_1, \alpha^{(j)}_2, 0)^\top \in \R^3$, $j\in\{1,2\}$, depending only on $\|\gamma\|_{W^{1,\infty}(\R^2)}$ such that for all $\ep\in(0,\ep^{(2)}]$ and $r\in[\ep/\ep^{(2)}, 1]$, we have
\begin{align}\label{est2.theo.lip.nonlinear.periodic}
\begin{split}
& \bigg(\dashint_{B^\ep_{r,+}(0)} 
\big|u^\ep(x) - \sum_{j=1}^{2} c^\ep_{r,j} ( x_3 {\bf e}_j + \ep \alpha^{(j)} )\big|^2 \dd x \bigg)^\frac12 
\le
\widetilde{C^{(2)}_M} (r^{1+\mu} + \ep^\frac32 r^{-\frac12})\,,
\end{split}
\end{align}
where the coefficient $c^\ep_{r,j}$, $j\in\{1,2\}$, is same as in the estimate \eqref{est1.theo.lip.nonlinear.periodic}, while the constant $\widetilde{C^{(2)}_M}$ is independent of $\ep$ and $r$, and depends on $\|\gamma\|_{W^{1,\infty}(\R^2)}$, $M$, and $\mu$.
\end{theorem}
%
\begin{remark}\label{rem.theo.lip.periodic2}
(i) Each of the constants $C^{(2)}_M$ and $\widetilde{C^{(2)}_M}$ satisfies the same property as $C^{(1)}_M$ in Theorem \ref{theo.lip.nonlinear} as functions of $M$. \\
(ii) Note that at the small scale, namely when $r=O(\ep)$, the right-hand side in the estimate \eqref{est1.theo.lip.nonlinear.periodic} is no better than the right-hand side of \eqref{est2.theo.lip.nonlinear} in Theorem \ref{theo.lip.nonlinear}. Hence there is no improvement at this scale. On the other hand, if we consider the case $r\in[(\ep/\ep^{(2)})^\delta, 1]$ with $\delta\in(0,1)$, then we see that
\begin{align*}
r^{1+\mu} + \ep^\frac12 r^{\frac12} 
\le (1 + (\ep^{(2)})^\frac12 r^{\frac{1-\delta}{2\delta}-\mu}) r^{1+\mu}\,.
\end{align*}
Therefore, we call the estimate \eqref{est1.theo.lip.nonlinear.periodic} a mesoscopic $C^{1,\mu}$ estimate at the scales $r\in[(\ep/\ep^{(2)})^\delta, 1]$ with $\delta\in(0,(2\mu+1)^{-1}]$.\\
(iii) A comparison between the estimates \eqref{est1.theo.lip.nonlinear.periodic} and \eqref{est2.theo.lip.nonlinear.periodic} highlights the regularity improvement coming from the boundary periodicity. Indeed, 
the estimate \eqref{est2.theo.lip.nonlinear.periodic} is sharper than \eqref{est1.theo.lip.nonlinear.periodic} at mesoscopic scales because $\ep^\frac32 r^{-\frac12} \le \ep^\frac12 r^\frac12$ holds whenever $r\in[\ep, 1]$. 
\end{remark}
%
%
\begin{remark}[relation with the wall law]\label{rem.theo.lip.periodic3}
(i) Let us denote the polynomial in \eqref{est2.theo.lip.nonlinear.periodic} by $P^\ep_{N,j}$, $j\in\{1,2\}$:
\begin{align}\label{polynomial.intro.}
P^\ep_{N,j}(x)=x_3 {\bf e}_j + \ep \alpha^{(j)}\,.
\end{align}
Then each $P^\ep_{N,j}$ is a shear flow in the half-space $\R^3_+$ and is an explicit solution to the  Navier-Stokes equations with a Navier-slip boundary condition
\begin{equation}\tag{NS$^\ep_N$}\label{Navier-slip.ep}
\left\{
\begin{array}{ll}
-\Delta u^\ep_N + \nabla p^\ep_N=-u^\ep_N \cdot \nabla u^\ep_N &\mbox{in}\ \R^3_+\\
\nabla\cdot u^\ep_N=0 &\mbox{in}\ \R^3_+\\
u^\ep_{N,3}=0 &\mbox{on}\ \partial\R^3_+\\
(u^{\ep}_{N,1}, u^{\ep}_{N,2})^\top
= \ep {\overline M} (\partial_3 u^{\ep}_{N,1}, \partial_3 u^{\ep}_{N,2})^\top &\mbox{on}\ \partial\R^3_+
\end{array}
\right.
\end{equation}
with a trivial pressure $p^\ep_{N}=0$. Here the $2\times2$ matrix ${\overline M}=(\alpha^{(j)}_i)_{1\le i,j\le 2}$ can be proved to be positive definite; see Proposition \ref{prop.per.BL} (ii). Thus the estimate \eqref{est2.theo.lip.nonlinear.periodic} in Theorem \ref{theo.lip.nonlinear.periodic} reads as follows: any weak solution $u^\ep$ to \eqref{intro.NS.ep} can be approximated at any mesoscopic scale by a linear combination of the Navier polynomials $P^\ep_{N,1}$ and $P^\ep_{N,2}$ multiplied by constants depending on $u^\ep$. This is a local version of the Navier wall law at the $O(\ep^\delta)$-scales, which has been widely studied in the global framework. \\
(ii) Our result can be extended to the stationary ergodic or the almost periodic setting. We also note that the wall law breaks down when the boundary does not have any structure at all; see \cite{GM}.
\end{remark}
%

The novelty of our results can be summarized as follows:

(I) Singular boundary: it is just Lipschitz and has no structure (except in Theorem \ref{theo.lip.nonlinear.periodic} {\rm (ii)}). 

(II) No smallness assumption on the size of solutions.

(III) Derivation of a local wall law and local error estimates.

As is stated in (I), one of the originalities of Theorem \ref{theo.lip.nonlinear} is that it does not rely on the smoothness of the boundary such as, the H\"older continuity of $\nabla\gamma$. Moreover, one cannot use any Fourier methods due to the lack of structure of the boundary. In fact, when working with Lipschitz boundaries, the classical Schauder theory is not applicable directly since there is no improvement of flatness coming from zooming on the boundary as is explained in \cite{KP2}. The smoothing happens at scales larger than that of the boundary layer thickness. 

Concerning point (II), 
we are able to remove any smallness assumption on the size of the solutions in Theorem \ref{theo.lip.nonlinear} and Theorem \ref{theo.lip.nonlinear.periodic}. This is in stark contrast with previous works concerned with the regularity of elliptic or Stokes systems \cite{AL1,G1,GS15,KP1,KP2}. Moreover, as far as we know the error estimates in the stationary global setting are all in the perturbative regime; see for instance \cite{GM}.

Point (III) is concerned with Theorem \ref{theo.lip.nonlinear.periodic}. It is important physically as well as mathematically since we are interested in the effects of rough boundaries on viscous fluids. Our result is a first-step toward understanding roughness effects on the Navier-Stokes flows in view of regularity improvement. As far as we know, estimate \eqref{est2.theo.lip.nonlinear.periodic} is the first justification of a local wall law.

These three aspects are further discussed in connection with our strategy in the paragraph below.

\subsection*{Difficulties and strategy}

The proof of Theorem \ref{theo.lip.nonlinear} and Theorem \ref{theo.lip.nonlinear.periodic} is based on a compactness argument as in \cite{KP1,KP2} originating from the works \cite{AL1, AL2} on uniform estimates in homogenization. In principle, we follow the strategy of \cite{KP2} concerned with the regularity theory of elliptic systems in bumpy domains. The main points in \cite{KP2} are: (1) construction of a boundary layer corrector in the Lipschitz half-space, (2) proof of the mesoscopic regularity by compactness and iteration. This strategy entails difficulties related to the lack of structure of the boundary which implies a lack of compactness of the solution to the boundary layer problem, and to the unavailability of Fourier me\-thods up to the boundary. In addition to these difficulties, our proof is more involved due to: (i) the vectoriality of the equations \eqref{intro.NS.ep} and the divergence-free condition, (ii) the nonlocal pressure, (iii) the nonlinearity of the Navier-Stokes equations and the lack of smallness of the solutions.

Concerning the first point, the (vectorial) divergence-free condition $\nabla\cdot u^\ep=0$ causes a difficulty in the compactness argument even for the Stokes equations; see Section \ref{sec.linear}, especially Lemma \ref{lem.cpt} and its proof. A key idea is that no boundary layer is needed on the vertical component of the velocity. Therefore the boundary layer corrector is naturally constructed as a divergence-free function.

Concerning the second point, let us stress a key difference between the stationary Navier-Stokes equations and the nonstationary ones. For the stationary Stokes equations imposed in a ball $B_1(0)$, one can estimate the pressure directly in terms of the velocity as follows:
\begin{equation}\label{e.estpressure}
\begin{split}
\big\|p- (\overline{p})_{B_1(0)}\big\|_{L^2(B_1(0))}
&\leq C\|\nabla p\|_{H^{-1}(B_1(0))}\\
&\leq C\|\Delta u\|_{H^{-1}(B_1(0))}\leq C\|\nabla u\|_{L^2(B_1(0))}.
\end{split}
\end{equation}
Similar estimates in balls intersecting the boundary and for the Navier-Stokes equations are intensively used in our paper. This is in strong contrast with the nonstationary Navier-Stokes equations where the pressure interacts with the time derivative of the velocity. This yields parasitic solutions, which are responsible for a lack of local smoothing in time and also for a more serious lack of local smoothing in space of the gradient of the velocity in the half-space; see Kang \cite{K} and Seregin and \v{S}ver\'{a}k \cite{SS}. 

The third aspect is partly related to (ii). In typical statements of the partial regularity theory for the nonstationary Navier-Stokes equations, one assumes smallness of certain scale-critical quantities in $\ep$ and hence one obtains linear equations in the limit $\ep\rightarrow 0$. Then the regularity theory for the linear equations yields a space-time H\"older regularity improvement for the original solution; see Lin \cite{Li}, Lady\v{z}enskaja and Seregin \cite{LS99}, and Mikhailov \cite{Mi} for example. However, for the stationary Navier-Stokes equations discussed in our paper, we do not need such a smallness condition; see Theorem \ref{theo.lip.nonlinear}. The limit equations when $\ep\rightarrow 0$ are not linear, but we can prove the smoothness of weak solutions because $H^1$ bounds are enough to control both the nonlinear term and pressure term in $L^2$ space (see Appendix \ref{appendix.Regularity} for details). Then bootstrapping using the standard elliptic regularity in a smooth domain leads to the (spatial) $C^\infty$-regularity for the limit equations. Estimate \eqref{e.estpressure} is the reason why one can bootstrap the regularity. Once the regularity is inherited at a fixed scale $\theta\in(0,1)$, a serious difficulty arises in the iteration of such an estimate. At each step in the induction, we need to use the Caccioppoli inequality from Appendix \ref{appendix.Caccioppoli} to control the norm $\|u^\ep\|_{L^2}$. A naive approach yields an estimate that depends algebraically on the size $M$ of $u^{\ep}$ as in \eqref{est1.theo.lip.nonlinear}. Hence the naive estimate becomes unbounded in $M$ as the iteration proceeds. This prevents one from closing the induction due to the lack of uniformity. We overcome this difficulty by choosing the free parameter $\theta$ in the compactness lemma in terms of the data $\gamma$ and $M$. This is done in the spirit of the Newton shooting method. We will make this idea precise in Section \ref{sec.nonlinear}. It should finally be emphasized that the boundary layer corrector, entering the scheme for the nonlinear Navier-Stokes equations \eqref{intro.NS.ep}, solves the linear Stokes equations. This is expected from the following formal heuristics. Indeed, in the boundary layer $u^\ep\simeq \ep v(x/\ep)$, so that $v$ solves $-\frac1\ep\Delta v+\ep v\cdot\nabla v+\nabla q=0$, $\nabla\cdot v=0$.

\subsection*{Outline of the paper} The following two sections are devoted to the analysis of the boundary layer equations. In Section \ref{sec.Preliminaries} we collect preliminary results on the well-posedness for the Stokes problem and on the Dirichlet-to-Neumann operator ${\rm DN}$ in the framework of non-localized Sobolev data. In Section \ref{sec.BL} we study the boundary layer equations by formulating equivalent equations on a strip bounded in the vertical direction and involving the nonlocal operator Dirichlet-to-Neumann ${\rm DN}$. Our goal is to prove the unique existence of solutions of the equivalent equations. We study the asymptotic behavior of the solution away from the boundary when the boundary is periodic in Subsection \ref{subsec.per.BL}. In Section \ref{sec.linear} we prove the linear version of Theorem \ref{theo.lip.nonlinear} in order to show how the compactness method works in the regularity argument. In Section \ref{sec.nonlinear} we prove the main results namely Theorem \ref{theo.lip.nonlinear} and Theorem \ref{theo.lip.nonlinear.periodic}. The regularity theory in a domain with a flat boundary and the Caccioppoli inequality are stated respectively in Appendices \ref{appendix.Regularity} and \ref{appendix.Caccioppoli}.

\subsection*{Notations} Let us summarize the notations in this paper for easy reference. For $x=(x_1, x_2, x_3)^\top\in\R^3$, we denote by $x'$ its tangential part $(x_1, x_2)^\top$. For $d\in\{2,3\}$ and $x, y\in\R^d$, we denote by $x\cdot y$ the inner product of $x$ and $y$. Then $|\cdot|$ denotes the corresponding norm in $\R^d$. For $r\in(0,1]$ and $\ep\in(0,1]$, we define $B^\ep_{r,+}(0)$ and $\Gamma^\ep_r(0)$ as is done in \eqref{intro.def1} and set 
\begin{align*}
&B_{r}(0) = \{x\in \R^3~|~x'\in(-r,r)^2\,, \ \ x_3\in(-r,r) \}=(-r,r)^3\,, \\
&B_{r,+}(0) = \{x\in \R^3~|~x'\in(-r,r)^2\,, \ \ x_3\in(0,r) \} \,, \\
&\Gamma_r(0) = \{x\in \R^3~|~x'\in(-r,r)^2\,, \ \ x_3 = 0\}\,.
\end{align*}
Note that formally we have $B_{r,+}(0)=B^0_{r,+}(0)$ and $\Gamma_r(0)=\Gamma^0_r(0)$. For an open set $\Omega\subset\R^3$ and a Lebesgue measurable function $f$ on $\Omega$, we set
\begin{align}\label{intro.def3}
\dashint_{\Omega} |f| &= \frac{1}{|\Omega|} \int_{\Omega} |f|\,, \qquad
(\overline{f})_{\Omega} = \frac{1}{|\Omega|} \int_{\Omega} f\,,
\end{align}
where $|\Omega|$ denotes the Lebesgue measure of $\Omega$. Finally, we define the Sobolev-Kato space $H^s_{uloc}(\R^2)$: let $\vartheta\in C^\infty_0(\R^2)$ be such that $\supp \vartheta \subset [-1,1]^2$, $\vartheta=1$ on $[-\frac14,\frac14]^2$, and 
\begin{align*}
\sum_{k\in \Z^2} \vartheta_k(x)=1\,,\quad x\in \R^2\,, 
\qquad \vartheta_k(x)=\vartheta(x-k)\,.
\end{align*}
Then, for $s\in[0,\infty)$, we define the space $H^s_{uloc}(\R^2)$ of functions of non-localized $H^s$ energy by
\begin{align*}
H^s_{uloc}(\R^2)
=\Big\{u\in H^s_{loc}(\R^2)~\Big|~ 
\sup_{k\in \Z^2}\|\vartheta_k u\|_{H^s(\R^2)}<\infty \Big\}
\end{align*}
and the space $L^2_{uloc}(\R^2)$ by $L^2_{uloc}(\R^2)=H^0_{uloc}(\R^2)$. We emphasize that $H^s_{uloc}(\R^2)$ is well-defined independently of the choice of $\vartheta$ for any $s\in[0,\infty)$ (see \cite[Lemma 7.1]{ABZ} for the proof) and admits the embedding $W^{1,\infty}(\R^2) \hookrightarrow H^s_{uloc}(\R^2)$ when $s\in[0,1)$.

Note that, since our interest is in the local boundary regularity of \eqref{intro.NS.ep}, the boundary condition is prescribed only on the lower part of $\partial B^\ep_{r,+}(0)$. We work in the framework of weak solutions of \eqref{intro.NS.ep}. A vector function $u^\ep\in H^1(B^\ep_{1,+}(0))^3$ is said to be a weak solution to \eqref{intro.NS.ep} if $u^\ep$ satisfies $\nabla\cdot u^\ep=0$ in the sense of distributions, $u^\ep|_{\Gamma^\ep_1(0)}=0$ in the trace sense, and 
\begin{align}\label{intro.def2}
\int_{B^\ep_{1,+}(0)} \nabla u^\ep \cdot \nabla \varphi
= -\int_{B^\ep_{1,+}(0)} (u^\ep\cdot\nabla u^\ep) \cdot \varphi
\end{align}
for any $\varphi\in C^{\infty}_{0,\sigma}(B^\ep_{1,+}(0))$. Here $C^\infty_{0,\sigma}(\Omega)$ denotes the space of test functions $\{ f \in C^\infty_0(\Omega)^3~|~\nabla\cdot f=0\}$ when $\Omega$ is an open set in $\R^3$. 
For the pressure $p^\ep$, we emphasize that the unique existence in $L^2(B^\ep_{1,+}(0))$ up to an additive constant can be proved in a functional analytic way using the weak formulation \eqref{intro.def2}; see a textbook \cite[Lemma 3.3.1 and Remark 3.3.2, III]{So} for details. 
%
\section{Preliminaries}\label{sec.Preliminaries}
%
\noindent In this section we give preliminary results which will be used in the next section. In Subsection \ref{subsec.half.prob} we prove a well-posedness result for the Stokes problem in the half-space with nonhomogeneous Dirichlet boundary data in $H^{\frac12}_{uloc}(\R^2)^3$. Subsection \ref{subsec.DN} is devoted to the definition and basic properties of the Dirichlet-to-Neumann operator on $H^{\frac12}_{uloc}(\R^2)^3$ associated with the half-space problem. Throughout this section, we use the Fourier transform and its inverse transform respectively defined by 
\begin{align*}
&\mathcal{F}[f](\xi)=\hat{f}(\xi)
=\int_{\R^2} f(x) e^{-ix\cdot\xi}\dd x\,,\quad \xi\in\R^2\,, \\
&\mathcal{F}^{-1}[f](x)
=\frac{1}{(2\pi)^2} \int_{\R^2} f(\xi) e^{ix\cdot\xi}\dd \xi\,,\quad x\in\R^2,
\end{align*}
for $f\in\mathcal S(\R^2)$. 
We also use their extensions on the space of tempered distributions $\mathcal{S}'(\R^2)$.
\subsection{Analysis of the half-space problem}\label{subsec.half.prob}
%
We consider the Stokes equations in the half-space $\R^3_+=\{y=(y',y_3)^\top\in\R^3~|~y_3>0\}$ with a non-localized boundary data $u_0\in H^{\frac12}_{uloc}(\R^2)^3$
\begin{equation}\tag{SH}\label{S.half}
\left\{
\begin{array}{ll}
-\Delta u+\nabla p=0&\mbox{in}\ \R^3_+ \\
\nabla\cdot u=0&\mbox{in}\ \R^3_+ \\
u=u_0&\mbox{on}\ \partial\R^3_+\,.
\end{array}
\right.
\end{equation}
The well-posedness of the problem \eqref{S.half} is stated as follows.
%
\begin{proposition}\label{prop.WP.S.half}
Let $u_0\in H^{\frac12}_{uloc}(\R^2)^3$. Then there exists a unique weak solution $(u,p)\in H^1_{loc}(\overline{\R^3_+})^3\times L^2_{loc}(\overline{\R^3_+})$ to \eqref{S.half} satisfying
\begin{align}\label{est1.prop.WP.S.half}
\sup_{\eta\in\Z^2}\int_{\eta+(0,1)^2}
\int_0^\infty 
|\nabla u(y',y_3)|^2 \dd y_3 \dd y'
\le
C \|u_0\|_{H^{\frac12}_{uloc}(\R^2)}^2\,,
\end{align}
where $C$ is a numerical constant.
\end{proposition}
%
\begin{remark}\label{rem.prop.WP.S.half}
The pressure $p$ can be chosen to satisfy
\begin{align*}
\sup_{\eta\in\Z^2}\int_{\eta+(0,1)^2}
\int_0^\infty 
|p(y',y_3)|^2 \dd y_3 \dd y'
\le
C \|u_0\|_{H^{\frac12}_{uloc}(\R^2)}^2\,.
\end{align*}
\end{remark}
%
\begin{proof} We follow the proof of \cite[Proposition 6]{GM} for the two-dimensional Stokes equations.

\noindent {\bf (Existence)} We give only the outline here since this part is parallel to \cite[Proposition 6]{GM}. Let $u_0\in H^{\frac12}_{uloc}(\R^2)^3$. Then a solution to \eqref{S.half} can be constructed by using the Poisson kernel $(U,P)$ as
\begin{align}\label{rep1.proof.prop.WP.S.half}
\begin{split}
u(y',y_3) &= \int_{\R^2} U(y'-\tilde{y}',y_3) u_0(\tilde{y}') \dd \tilde{y}'\,, \\
p(y',y_3) &= \int_{\R^2} \nabla P(y'-\tilde{y}',y_3) \cdot u_0(\tilde{y}') \dd \tilde{y}'\,,
\end{split}
\end{align}
where the kernels $U=U(y)$ and $P=P(y)$ are respectively defined by
\begin{align}\label{def1.proof.prop.WP.S.half}
\begin{split}
&U(y) = \frac{3y_3}{2\pi(|y'|^2+y_3^2)^\frac52} 
\begin{pmatrix}
y_1^2  & y_1 y_2 & y_1 y_3 \\
y_1 y_2  & y_2^2 & y_2 y_3 \\
y_1 y_3  & y_2 y_3 & y_3^2 \\
\end{pmatrix}\,, \qquad \\
&P(y) = -\frac{y_3}{\pi(|y'|^2+y_3^2)^\frac32}\,.
\end{split}
\end{align}
We easily check that $u$ and $p$ belong to $C^\infty(\R^3_+)$ by the derivative estimates of $(U,P)$
\begin{align*}
|\nabla^m U(y)| 
\le 
\frac{C_m y_3^{\delta_{0m}}}{(|y'|^2+y_3^2)^\frac{m+\delta_{0m}+2}{2}}\,, \qquad
|\nabla^m \nabla P(y)| 
\le
\frac{C_m}{(|y'|^2+y_3^2)^\frac{m+3}{2}}
\end{align*}
for $m\in\N\cup\{0\}$, which can be verified by direct computation. Here $\delta_{0m}$ denotes the Kronecker delta. Moreover, we can prove the following estimates for $a\in(0,\infty)$
\begin{align*}
&\sup_{\eta\in\Z^2}\int_{\eta+(0,1)^2}
\int_a^\infty 
\Big(|\nabla u(y',y_3)|^2 + |p(y',y_3)|^2 \Big) \dd y_3 \dd y' \\
& \quad
\le \frac{C}{a^5} \|u_0\|_{L^2_{uloc}(\R^{2})}^2\,, \\
&\sup_{\eta\in\Z^2}\int_{\eta+(0,1)^2}
\int_0^a
\Big(|\nabla u(y',y_3)|^2 + |p(y',y_3)|^2 \Big) \dd y_3 \dd y' \\
& \quad
\le
C \max\{1,a\} \|u_0\|_{H^{\frac12}_{uloc}(\R^{2})}^2\,,
\end{align*}
which lead to \eqref{est1.prop.WP.S.half}. We can also check that $(u,p)$ solves \eqref{S.half} and $u=u_0$ on $\partial\R^3_+$ in the trace sense.

\noindent {\bf (Uniqueness)} Suppose that $u_0=0$ in \eqref{S.half}. Then we aim at proving $u=0$ in the class 
\begin{align}\label{est0.proof.prop.WP.S.half}
\sup_{\eta\in\Z^2}\int_{\eta+(0,1)^2}
\int_0^\infty 
|\nabla u(y',y_3)|^2 \dd y_3 \dd y'
<\infty\,.
\end{align}
By the regularity theory of the Stokes equations and by the no-slip condition on $\partial\R^3_+$, we have
\begin{align}\label{est1.proof.prop.WP.S.half}
\begin{split}
& \sup_{\eta\in\Z^2}\int_{\eta+(0,1)^2}
\int_0^\infty 
\Big(|\nabla^m \nabla u(y',y_3)|^2 + |\nabla^m \nabla p(y',y_3)|^2 \Big) \dd y_3 \dd y' \\
\le
& 
C \sup_{\eta\in\Z^2}\int_{\eta+(0,1)^2}
\int_0^\infty 
|\nabla u(y',y_3)|^2 \dd y_3 \dd y'
\end{split}
\end{align}
for $m\in\N\cup\{0\}$ with a constant $C$ depending on $m$; see the proof of \cite[Proposition 6]{GM} for the two-dimensional case. Thus, for all fixed $y_3\in(0,\infty)$, we see that $u(\cdot,y_3)$ and $p(\cdot,y_3)$ belong to the space of tempered distributions $\mathcal{S}'(\R^2)$. Hence we can take the (partial) Fourier transform of \eqref{S.half} with $u_0=0$ in $y'$. By letting $\xi\in\R^2$ be the dual variable of $y'$, we have the equations
\begin{equation}\label{eq1.proof.prop.WP.S.half}
\left\{
\begin{array}{ll}
(|\xi|^2-\partial_3^2) \hat{u}'(\xi,y_3) + i\xi \hat{p}(\xi,y_3) = 0\,, &\xi\in \R^2 \\
(|\xi|^2-\partial_3^2) \hat{u}_3(\xi,y_3) + \partial_3 \hat{p}(\xi,y_3) = 0\,, &\xi\in \R^2 \\
i\xi \cdot \hat{u}'(\xi,y_3) + \partial_3 \hat{u}_3(\xi,y_3) = 0\,,&\xi\in \R^2 \\
\hat{u}(\xi,0) = 0\,. &\xi\in \R^2\,.
\end{array}
\right.
\end{equation}
By eliminating the pressure $\hat{p}(\xi,y_3)$ and using the divergence-free condition, we find
\begin{align}\label{eq2.proof.prop.WP.S.half}
(|\xi|^2-\partial_3^2)^2 \hat{u}_3(\xi,y_3) = 0 \ \ \mbox{in} \ \ \mathcal{S}'(\R^2)\,.
\end{align}
To avoid the singularity at $\xi=0$, we introduce a function $\varphi\in C^\infty_0(\R^2)$ satisfying $\varphi(\xi)=0$ in a neighborhood of $\xi=0$. Since $\varphi(\xi) \hat{u}_3(\xi,y_3)$ satisfies the equation \eqref{eq2.proof.prop.WP.S.half} replacing $\hat{u}_3(\xi,y_3)$ by $\varphi(\xi) \hat{u}_3(\xi,y_3)$, there exist compactly supported $A_i\in\mathcal{S}'(\R^2)$ and $B_i\in\mathcal{S}'(\R^2)$, $i\in\{1,2\}$, such that
\begin{align*}
\varphi(\xi) \hat{u}_3(\xi,y_3)
= (A_1 + y_3 A_2) e^{-|\xi| y_3} + (B_1 + y_3 B_2) e^{|\xi| y_3} \ \ \mbox{in} \ \ \mathcal{S}'(\R^2)\,.
\end{align*}
The integrability in the $y_3$ variable in \eqref{est1.proof.prop.WP.S.half} leads to $B_1 = B_2 = 0$, while the boundary conditions $\hat{u}_3(\xi,0)=\partial_3 \hat{u}_3(\xi,0)=0$ imply $A_1=A_2=0$. Hence we have $\varphi(\xi) \hat{u}_3(\xi,y_3)=0$ in $\mathcal{S}'(\R^2)$ for any cut-off function $\varphi\in C^\infty_0(\R^2)$ vanishing near the origin, which yields that $\hat{u}_3(\xi,y_3)$ is supported at $\xi=0$. Thus $u_3(y',y_3)$ is a polynomial in $y'$ with coefficients depending on $y_3$, and therefore, because of \eqref{est1.proof.prop.WP.S.half} for $m=0$, we see that $u_3(y',y_3)$ is in fact independent of $y'$:
\begin{align}\label{eq3.proof.prop.WP.S.half}
u_3(y',y_3)=u_3(y_3)\,.
\end{align}
On the other hand, since $\varphi(\xi) \hat{u}_3(\xi,y_3)=0$ in $\mathcal{S}'(\R^2)$, from the equations \eqref{eq1.proof.prop.WP.S.half}$_1$ and \eqref{eq1.proof.prop.WP.S.half}$_3$ we have $\varphi(\xi) |\xi|^2 \hat{p}(\xi,y_3)=0$ in $\mathcal{S}'(\R^2)$ for any cut-off function $\varphi\in C^\infty_0(\R^2)$ vanishing near $\xi=0$. Thus, by a similar reasoning as for $u_3$, we conclude that there exists a function $f=f(y_3)$ such that
\begin{align}\label{eq4.proof.prop.WP.S.half}
\nabla p(y',y_3)=(0,0,f(y_3))^\top\,.
\end{align}
Then, going back to the original equations \eqref{S.half} with $u_0=0$, we see that $u'(y',y_3)$ solves the Laplace equation with the Dirichlet boundary condition
\begin{equation*}
\left\{
\begin{array}{ll}
-\Delta u'=0&\mbox{in}\ \R^3_+ \\
u'=0&\mbox{on}\ \partial \R^3_+\,.
\end{array}
\right.
\end{equation*}
The Liouville theorem in the class \eqref{est0.proof.prop.WP.S.half} implies that $u'(y',y_3)$ is a constant vector field, and hence, $u'=0$ by the boundary condition. Hence the proof will be complete if we prove $u_3=0$. From $\Delta p=0$ following from \eqref{S.half}, the equality \eqref{eq4.proof.prop.WP.S.half} leads to $\nabla p(y',y_3)=(0,0,a)^\top$ with some $a\in\R$. After inserting \eqref{eq3.proof.prop.WP.S.half} and $\partial_3 p=a$ to \eqref{S.half}, we have $u_3(y_3)=ay_3^2 + by_3 + c $ with some $(b,c)\in\R^2$, which implies $u_3=\partial_3 p=0$ from \eqref{est1.proof.prop.WP.S.half} and the boundary condition. This completes the proof.
\end{proof}
%
\subsection{The Dirichlet-to-Neumann operator}\label{subsec.DN}
In this subsection we recall the definition and basic properties of the Dirichlet-to-Neumann operator ${\rm DN}$ on $H^{\frac12}_{uloc}(\R^2)^3$ associated with the Stokes equations \eqref{S.half}. We follow the procedure in \cite[Subsection 2.2]{GM} treating the two-dimensional problem; see also \cite{ABZ} studying the water-waves equations, \cite{DP} and \cite{KP2} for related studies. Before going into the details, we give a useful lemma for estimating elements in $H^{\frac12}_{uloc}(\R^2)$.
%
\begin{lemma}\label{lem.est.uloc}
Let $u_0\in H^{\frac12}_{uloc}(\R^2)$ and $\chi\in C^\infty_0(\R^2)$ with $\supp\chi\subset(-R,R)^2$ for some $R\in(0,\infty)$. Then we have $\chi u_0\in H^{\frac12}(\R^2)$ and
\begin{align*}
\|\chi u_0\|_{H^{\frac12}(\R^2)}
\le CR \|u_0\|_{H^{\frac12}_{uloc}(\R^2)}\,,
\end{align*}
where the constant $C$ depends only on $\|\chi\|_{W^{1,\infty}(\R^2)}$.
\end{lemma}
%

We refer to the proof of \cite[Lemma 2.26]{DP}.

%
\noindent Let ${\rm M}={\rm M}(\xi)$ and $\widetilde{{\rm M}}=\widetilde{{\rm M}}(\xi)$ be $3\times3$ matrices defined by
\begin{align}\label{def1.subsec.DN}
\begin{split}
&{\rm M}(\xi)
= \begin{pmatrix}
|\xi|+\xi_1^2 |\xi|^{-1} & \xi_1\xi_2 |\xi|^{-1} & i\xi_1 \\
\xi_1\xi_2 |\xi|^{-1} & |\xi|+\xi_2^2 |\xi|^{-1} & i\xi_2 \\
-i\xi_1 & -i\xi_2 & 2|\xi| 
\end{pmatrix}\,, \\
&\widetilde{{\rm M}}(\xi)
= \begin{pmatrix}
0 & 0 & i\xi_1 \\
0 & 0 & i\xi_2 \\
-i\xi_1 & -i\xi_2 & 0
\end{pmatrix}\,.
\end{split}
\end{align}
Here ${\rm M}$ is the symbol of the Dirichlet-to-Neumann operator of \eqref{S.half} on $H^{\frac12}(\R^2)^3$, while $\widetilde{{\rm M}}$ is the singular part of ${\rm M}$ because it is the Fourier transform of a derivative of a Dirac mass. Morever, let $K=K(y')$ be a $3\times3$ matrix defined by
\begin{align*}
K(y')=\mathcal{F}^{-1}[{\rm M}-\widetilde{{\rm M}}](y')\,,
\end{align*}
which must be understood in distributional sense: it is the inverse transform of a tempered distribution ${\rm M}-\widetilde{{\rm M}}$. From the theory of distributions, we see that $K$ is a function on $\R^2\setminus\{0\}$ satisfying
\begin{align}\label{est1.subsec.DN}
|K(y')| \le C|y'|^{-3}\,.
\end{align}
Then the operator ${\rm DN}$ on $H^{\frac12}_{uloc}(\R^2)^3$ of \eqref{S.half} is defined in the following manner. Fix $u_0\in H^{\frac12}_{uloc}(\mathbb R^2)^3$ and $R\in(1,\infty)$, and let $\chi\in C^\infty_0(\R^2)$ be a cut-off function such that
\begin{align*}
& \chi\in[0,1]\,, \quad
\supp\chi\subset (-R-2,R+2)^2\,, \quad \\
& \chi=1 \ \, {\rm in} \ \, [-R-1,R+1]^2\,, \quad 
\|\chi\|_{W^{1,\infty}(\R^2)} \le 2\,.
\end{align*}
Then we define ${\rm DN}(u_0)$ as a functional on the set of test functions $\varphi$ supported in $(-R,R)^2$
\begin{align}\label{def2.subsec.DN}
\begin{split}
\langle {\rm DN}(u_0), \varphi \rangle_{\mathcal{D}', \mathcal{D}}
& = \langle \mathcal{F}^{-1}[{\rm M} \widehat{\chi u_0}], \varphi \rangle_{H^{-\frac12}, H^{\frac12}} \\
&\quad
+ \int_{\R^2} \big(K\ast (1-\chi)u_0\big)(y') \cdot\varphi(y') \dd y'\,,
\end{split}
\end{align}
where $\langle\cdot, \cdot\rangle_{\mathcal{D}', \mathcal{D}}$ and 
$\langle\cdot, \cdot\rangle_{H^{-\frac12}, H^{\frac12}}$ respectively denote the duality product between $\mathcal{D}'(\R^2)^3$ and $\mathcal{D}(\R^2)^3=C^\infty_0(\R^2)^3$, and $H^{-\frac12}(\R^2)^3$ and $H^{\frac12}(\R^2)^3$. Moreover, $\ast$ denotes the usual convolution product. Let us emphasize that the singular part of the Dirichlet-to-Neumann operator $\widetilde{{\rm M}}$ does not appear in \eqref{def2.subsec.DN} because $\supp\varphi\cap\supp(1-\chi)=\emptyset$. Thanks to the properties of $\chi$ and $\varphi$, the second term in the right-hand side of \eqref{def2.subsec.DN} converges, and consequently, the operator ${\rm DN}$ gives an extension of the ``standard" Dirichlet-to-Neumann operator on $H^{\frac12}(\R^2)^3$. One can also check that ${\rm DN}$ is well-defined independently of the choice of $\chi$ in a similar manner as in \cite[Lemma 7]{GM}. We summarize the basic facts of ${\rm DN}$ as follows.
%
\begin{lemma}\label{lem.DN}
{\rm (i)} For $u_0\in H^{\frac12}_{uloc}(\R^2)^3$ and $\varphi\in C^\infty_0(\R^2)^3$ with $\supp\varphi\subset (-R,R)^2$, we have
\begin{align}\label{est1.lem.DN}
|\langle {\rm DN}(u_0), \varphi \rangle_{\mathcal{D}', \mathcal{D}}|
\le C R \|u_0\|_{H^{\frac12}_{uloc}(\R^2)} \|\varphi\|_{H^{\frac12}(\R^2)}\,,
\end{align}
where the constant $C$ is independent of $R\in(1,\infty)$. \\
\noindent {\rm (ii)} Let $u_0\in H^{\frac12}_{uloc}(\R^2)^3$ and let $\{u_{0,n}\}_{n=1}^\infty \subset H^{\frac12}(\R^2)^3$ with $\sup_{n\in\N} \|u_{0,n}\|_{L^2(\R^2)} <\infty$ satisfy
\begin{align}\label{assump1.lem.DN}
u_{0,n} \ \rightharpoonup \ u_0 \quad {\rm in} \ \ H^{\frac12}((-k,k)^2)
\end{align}
for all $k\in\N$. Then for $\varphi\in C^\infty_0(\R^2)^3$ with $\supp\varphi\subset (-R,R)^2$, we have
\begin{align}\label{conv1.lem.DN}
\lim_{n\to\infty} 
\langle {\rm DN}(u_{0,n}), \varphi \rangle_{\mathcal{D}', \mathcal{D}}
= \langle {\rm DN}(u_{0}), \varphi \rangle_{\mathcal{D}', \mathcal{D}}\,.
\end{align}
\noindent {\rm (iii)} For $u_0\in H^{\frac12}_{uloc}(\R^2)^3$ and $\varphi\in C^\infty_0(\overline{\R^3_+})^3$ with $\nabla \cdot \varphi=0$ and $\supp\varphi\subset B_{R}(0)$, we have
\begin{align}\label{rep2.lem.std.DN}
\langle {\rm DN}(u_0), \varphi|_{y_3=0} \rangle_{\mathcal{D}', \mathcal{D}}
= \int_{\R^3_+} \nabla u \cdot \nabla \varphi\,,
\end{align}
where $u\in H^1_{loc}(\overline{\R^3_+})^3$ is the weak solution to \eqref{S.half} with $u=u_0$ on $\partial \R^3_+$ provided by Proposition \ref{prop.WP.S.half}. In particular, if $u_0$ is nonzero and compactly supported, then we have
\begin{align}\label{est2.lem.DN}
\langle {\rm DN}(u_0), u_0 \rangle_{\mathcal{D}', \mathcal{D}} > 0\,.
\end{align}
\end{lemma}
%
\begin{proof}
(i) The first term in the right-hand side of \eqref{def2.subsec.DN} is estimated as
\begin{align}\label{est1.proof.lem.DN}
\begin{split}
|\langle \mathcal{F}^{-1}[{\rm M} \widehat{\chi u_0}], \varphi \rangle_{H^{-\frac12}, H^{\frac12}}|
& \le C \|\chi u_0\|_{H^{\frac12}(\R^2)} \|\varphi\|_{H^{\frac12}(\R^2)} \\
& \le C R \|u_0\|_{H^{\frac12}_{uloc}(\R^2)} \|\varphi\|_{H^{\frac12}(\R^2)}\,,
\end{split}
\end{align}
where Lemma \ref{lem.est.uloc} is applied in the second line. From the definition we have
\begin{align*}
& \int_{\R^2} \big(K\ast (1-\chi)u_0\big)(y') \cdot\varphi(y') \dd y' \\
&= 
\int_{\R^2} \int_{\R^2}
K(y'-\tilde{y}') (1-\chi(\tilde{y}'))u_0(\tilde{y}') \cdot\varphi(y') \dd \tilde{y}' \dd y'\,.
\end{align*}
By using \eqref{est1.subsec.DN} and the properties of $\chi$ and $\varphi$, 
we estimate the second term in the right-hand side of \eqref{def2.subsec.DN} as
\begin{align}\label{est2.proof.lem.DN}
\begin{split}
&\bigg|\int_{\R^2} \big(K\ast (1-\chi)u_0\big)(y') \cdot\varphi(y') \dd y'\bigg| \\
&\le 
C\int_{\R^2} \int_{\R^2}
\frac{1-\chi(\tilde{y}')}{|y'-\tilde{y}'|^3}
|u_0(\tilde{y}')| |\varphi(y')| \dd \tilde{y}' \dd y' \\
&\le 
C \int_{\R^2}
\bigg(\int_{\R^2} \frac{1-\chi(\tilde{y}')}{|y'-\tilde{y}'|^3} \dd \tilde{y}' \bigg)^{\frac12}
\bigg(\int_{\R^2} \frac{1-\chi(\tilde{y}')}{|y'-\tilde{y}'|^3} 
|u_0(\tilde{y}')|^2 \dd \tilde{y}' \bigg)^{\frac12}
|\varphi(y')| \dd y'\\
&\le
C \|u_0\|_{L^2_{uloc}(\R^2)}
\int_{\R^2} |\varphi(y')| \dd y' \\
&\le 
CR \|u_0\|_{L^2_{uloc}(\R^2)} \|\varphi\|_{L^2(\R^2)}\,,
\end{split}
\end{align}
which with \eqref{est1.proof.lem.DN} implies the desired estimate \eqref{est1.lem.DN}. \\
\noindent (ii) From the assumption \eqref{assump1.lem.DN} we see that
\begin{align}\label{conv1.proof.lem.DN}
\lim_{n\to\infty} \langle \mathcal{F}^{-1}[{\rm M} \widehat{\chi u_{0,n}}], \varphi \rangle_{H^{-\frac12}, H^{\frac12}}
= \langle \mathcal{F}^{-1}[{\rm M} \widehat{\chi u_0}], \varphi \rangle_{H^{-\frac12}, H^{\frac12}}\,.
\end{align}
Fix $k\in\N$ arbitrarily. Then again from \eqref{assump1.lem.DN} we have 
\begin{align}\label{conv2.proof.lem.DN}
\lim_{n\to\infty} \int_{\R^2} \int_{|\tilde{y}'|\le k}
K(y'-\tilde{y}') (1-\chi(\tilde{y}'))
\big(u_{0,n}(\tilde{y}') - u_0(\tilde{y}')\big)\cdot\varphi(y') \dd \tilde{y}' \dd y' 
= 0\,.
\end{align}
On the other hand, if we choose $k\in\N$ to satisfy $k>\max\{R+2,2R\}$, then $\chi(\tilde{y}')=0$ and $2|y'| < |\tilde{y}'|$ as long as $|\tilde{y}'| > k$ and $|y'|\le R$. Thus, in a similar way as in \eqref{est2.proof.lem.DN}, we see that
\begin{align}\label{est3.proof.lem.DN}
\begin{split}
&\bigg|\int_{\R^2} \int_{|\tilde{y}'|> k}
K(y'-\tilde{y}') (1-\chi(\tilde{y}'))
\big(u_{0,n}(\tilde{y}') - u_0(\tilde{y}')\big)\cdot\varphi(y') \dd \tilde{y}' \dd y'\bigg| \\
&\le 
C \big(\sup_{n\in\N}\|u_{0,n}\|_{L^2(\R^2)}+ \|u_0\|_{L^2_{uloc}(\R^2)}\big)
\int_{\R^2}
\bigg(\int_{|\tilde{y}'|> k} \frac{\dd \tilde{y}'}{|\tilde{y}'|^3} \bigg)^{\frac12}
|\varphi(y')| \dd y'\\
&\le
\frac{C}{k^\frac12} \big(\sup_{n\in\N}\|u_{0,n}\|_{L^2(\R^2)}+ \|u_0\|_{L^2_{uloc}(\R^2)}\big)\,.
\end{split}
\end{align}
Then from \eqref{conv2.proof.lem.DN} and \eqref{est3.proof.lem.DN} we have
\begin{align*}
\lim_{n\to\infty} \int_{\R^2} \int_{\R^2}
K(y'-\tilde{y}') (1-\chi(\tilde{y}'))
\big(u_{0,n}(\tilde{y}') - u_0(\tilde{y}')\big)\cdot\varphi(y') \dd \tilde{y}' \dd y' 
= 0\,,
\end{align*}
which with \eqref{conv1.proof.lem.DN} implies the assertion \eqref{conv1.lem.DN}. \\
\noindent (iii) Since the both sides of \eqref{rep2.lem.std.DN} are continuous with respect to $u_0\in H^{\frac12}_{uloc}(\R^2)^3$, it suffices to prove it for smooth $u_0$ with all derivatives bounded. Let $V$ and $W$ denote the solutions to \eqref{S.half} with the boundary data $\chi u_0$ and $(1-\chi) u_0$. Then we have $u=V+W$ and one can check that
\begin{align*}
\begin{split}
&\int_{\R^3_+} \nabla V \cdot \nabla \varphi 
=  \langle \mathcal{F}^{-1}[{\rm M} \widehat{\chi u_0}], \varphi \rangle_{H^{-\frac12}, H^{\frac12}}\,, \\
&\int_{\R^3_+} \nabla W \cdot \nabla \varphi 
= \int_{\R^2} \big(K\ast (1-\chi)u_0\big)(y') \cdot\varphi(y') \dd y'\,.
\end{split}
\end{align*}
Thus we obtain \eqref{rep2.lem.std.DN}. If $u_0$ has a compact support, by choosing $R\in(0,\infty)$ sufficiently large, we see that $\chi u_0=u_0$ holds and $(1-\chi)u_0$ is identically zero. Then from the definition of ${\rm DN}$ we have
\begin{align*}
\langle {\rm DN}(u_0), u_0 \rangle_{\mathcal{D}', \mathcal{D}}
= \langle \mathcal{F}^{-1}[{\rm M} \widehat{u_0}], u_0 \rangle_{H^{-\frac12}, H^{\frac12}}\,.
\end{align*}
By the definition of ${\rm M}$ in \eqref{def1.subsec.DN}, one can easily check that the right-hand side is positive as long as $u_0$ is nonzero. Hence we conclude \eqref{est2.lem.DN}. The proof of Lemma \ref{lem.DN} is complete.
\end{proof}
%
\section{The boundary layer corrector}\label{sec.BL}
%
\noindent In this section we study the boundary layer equations in Lipschitz half-spaces. In Subsection \ref{subsec.general} we state the problem and introduce its equivalent formulation imposed in an infinite channel but invol\-ving the Dirichlet-to-Neumann operator ${\rm DN}$ of Subsection \ref{subsec.DN}. This formulation allows us to apply the Poincar\'e inequality in the vertical direction when estimating the local energy. In Subsection \ref{subsec.proof.BL} we prove the unique solvability of the equivalent equations in a similar way to \cite[Proposition 10]{GM} and \cite[Section 5]{KP2} via local energy estimates. The asymptotic behavior of the boundary layer corrector is investigated in the case when the boundary is periodic in Subsection \ref{subsec.per.BL}. We also collect norm estimates for the boundary layer correctors in Subsection \ref{subsec.useful.BL}, which will be useful in Sections \ref{sec.linear} and \ref{sec.nonlinear}.
%
\subsection{General case}\label{subsec.general}
We consider the boundary layer equations for $j\in\{1,2\}$
\begin{equation}\tag{BL$^{(j)}$}\label{BLj}
\left\{
\begin{array}{ll}
-\Delta v+\nabla q=0\,, &y\in\Omega^{\rm bl} \\
\nabla\cdot v=0\,, &y\in\Omega^{\rm bl} \\
v(y',\gamma(y'))=-\gamma(y') {\bf e}_j\,,
\end{array}
\right.
\end{equation}
where $\gamma\in W^{1,\infty}(\R^2)$ and $\Omega^{\rm bl}$ denotes the Lipschitz half-space $\Omega^{\rm bl}=\{y\in\R^3~|~\gamma(y')<y_3<\infty\}$. The unique existence of weak solutions to \eqref{BLj} is stated as follows.
%
\begin{proposition}\label{prop.BL}
Fix $j\in\{1,2\}$ and let $\gamma\in W^{1,\infty}(\R^2)$. Then there exists a unique weak solution $(v,q)=(v^{(j)},q^{(j)})\in H^1_{loc}(\overline{\Omega^{\rm bl}})^3\times L^2_{loc}(\overline{\Omega^{\rm bl}})$ to \eqref{BLj} satisfying
\begin{align}\label{est1.prop.BL}
\sup_{\eta\in\Z^2}\int_{\eta+(0,1)^2}
\int_{\gamma(y')}^\infty 
|\nabla v^{(j)}(y',y_3)|^2 \dd y_3 \dd y'
\le 
C\,,
\end{align}
where the constant $C$ depends only on $\|\gamma\|_{W^{1,\infty}(\R^2)}$. 
\end{proposition}
%
\noindent We postpone the proof of Proposition \ref{prop.BL} to the next subsection. The basic idea of the proof is to decompose the domain $\Omega^{\rm bl}$ into $\overline{\R^3_+}\cup(\Omega^{\rm bl}\setminus\overline{\R^3_+})$ and to derive an equivalent equations to \eqref{BLj} on the infinite channel $\Omega^{\rm bl}\setminus\overline{\R^3_+}$. In the following lemma we introduce the new equations involving the Dirichlet-to-Neumann operator of Subsection \ref{subsec.DN}. Let us set $\Omega^{{\rm bl},-}=\Omega^{\rm bl}\setminus\overline{\R^3_+}$.
%
\begin{lemma}\label{lem.equiv.BL}
Fix $j\in\{1,2\}$ and let $v\in H^1_{loc}(\overline{\Omega^{\rm bl}})^3$ be a weak solution to \eqref{BLj} such that 
\begin{align}\label{est1.lem.equiv.BL}
\sup_{\eta\in\Z^2}\int_{\eta+(0,1)^2}
\int_{\gamma(y')}^\infty 
|\nabla v(y',y_3)|^2 \dd y_3 \dd y'
< \infty\,.
\end{align}
Then the restriction $v^{-}=v|_{\Omega^{{\rm bl},-}}$ satisfies
\begin{equation}\tag{BL$^{(j),-}$}\label{BLj-}
\left\{
\begin{array}{ll}
-\Delta v^{-}+\nabla q^{-} = 0\,, &y\in\Omega^{{\rm bl},-} \\
\nabla\cdot v^{-} = 0\,, &y\in\Omega^{{\rm bl},-} \\
v^{-}(y',\gamma(y')) = -\gamma(y') {\bf e}_j \\
(-\partial_3 v^{-}+q^{-}{\bf e}_3)|_{y_3=0} = {\rm DN}(v^{-}|_{y_3=0})
\end{array}
\right.
\end{equation}
and $v|_{\R^3_+}$ is given by $v(y) = \int_{\R^2} U(y'-\tilde{y}',y_3) v(\tilde{y}',0) \dd \tilde{y}$ using the Poisson kernel $U$ in \eqref{def1.proof.prop.WP.S.half} in the proof of Proposition \ref{prop.WP.S.half}. Conversely, let $v^{-}\in H^1_{loc}(\overline{\Omega^{{\rm bl},-}})^3$ be a weak solution to \eqref{BLj-} such that 
\begin{align}\label{est2.lem.equiv.BL}
\sup_{\eta\in\Z^2}\int_{\eta+(0,1)^2}
\int_{\gamma(y')}^0 
|\nabla v^{-}(y',y_3)|^2 \dd y_3 \dd y'
< \infty\,.
\end{align}
Then the extension $v$ of $v^{-}$ to $\Omega^{{\rm bl}}$ defined by
\begin{align*}
v(y) = 
\left\{
\begin{array}{ll}
v^{-}(y)\,, &y\in \Omega^{{\rm bl},-} \\
\displaystyle \int_{\R^2} U(y'-\tilde{y}',y_3) v^{-}(\tilde{y}',0) \dd \tilde{y}'\,, 
&y \in\R^3_+
\end{array}
\right.
\end{align*}
satisfies both \eqref{BLj} and \eqref{est1.lem.equiv.BL}.
\end{lemma}
%
\begin{remark}\label{rem.lem.equiv.BL}
\noindent We call $v^{-}$ a weak solution to \eqref{BLj-} if $v^{-}$ satisfies \eqref{BLj-}$_3$ in the trace sense and 
\begin{align}\label{eq1.rem.lem.equiv.BL}
\begin{split}
\int_{\Omega^{{\rm bl},-}} \nabla v^{-} \cdot \nabla \varphi
= -\langle{\rm DN}(v^{-}|_{y_3=0}), \varphi|_{y_3=0}\rangle_{\mathcal{D}', \mathcal{D}}
\end{split}
\end{align}
holds for all $\varphi\in C^\infty_0(\overline{\Omega^{{\rm bl},-}})$ with $\nabla\cdot \varphi=0$ and $\varphi(y',\gamma(y'))=0$.
\end{remark}
%
\begin{proofx}{Lemma \ref{lem.equiv.BL}}
We omit the details since the statement easily follows from Proposition \ref{prop.WP.S.half} and Lemma \ref{lem.DN} (iii). This completes the proof.
\end{proofx}
%
\subsection{Proof of Proposition \ref{prop.BL}}\label{subsec.proof.BL}
We prove Proposition \ref{prop.BL} by using Lemma \ref{lem.equiv.BL}. The goal is to show the unique existence of $v^{-}$ solving \eqref{BLj-} in the class \eqref{est2.lem.equiv.BL}. In order to have a homogeneous Dirichlet boundary condition, we introduce new unknowns $w=v^{-}+y_3{\bf e}_j$ and $s=q^{-}$, and consider 
\begin{equation}\tag{$\widetilde{{\rm BL}}^{(j),-}$}\label{eq1.subsec.proof.BL}
\left\{
\begin{array}{ll}
-\Delta w+\nabla s = 0\,, &y\in\Omega^{{\rm bl},-} \\
\nabla\cdot w = 0\,, &y\in\Omega^{{\rm bl},-} \\
w(y',\gamma(y')) = 0 \\
(-\partial_3 w+s{\bf e}_3)|_{y_3=0} = {\rm DN}(w|_{y_3=0})-{\bf e}_j\,.
\end{array}
\right.
\end{equation}
The weak formulation of \eqref{eq1.subsec.proof.BL} is defined in a similar way to  \eqref{eq1.rem.lem.equiv.BL}. Before going into a detailed analysis, let us collect notations used in this subsection. For $r\in(0,\infty)$ and $y'_0\in\R^2$, we set
\begin{align*}
&\Omega^{{\rm bl},-}_{r,y'_0}
= \{y\in \R^3~|~y'\in y'_{0} + (-r, r)^2\,, \ \gamma(y') < y_3 < 0\}\,, \\
&\Sigma_{r,y'_0}
= \{(y',0)^\top\in \R^3~|~y'\in y'_{0} + (-r, r)^2\}\,, 
\end{align*}
and $\Omega^{{\rm bl},-}_{k}=\Omega^{{\rm bl},-}_{k,0}$ and $\Sigma_{k}=\Sigma_{k,0}$ when $k\in\N$. For even $m\in\N$ with $m=2\widetilde{m}$ and $k\in\N$ with $k\ge\widetilde{m}$, let $\mathcal C_{k,m}$ and $\mathcal C_m$ be sets of tiles $T$ in $\R^2$ of area $m^2$ respectively defined by
\begin{align*}
&\mathcal C_{k,m}
=\{T=\eta+(-\widetilde{m}, \widetilde{m})^2
~|~\eta\in\Z^2 \ \ {\rm and} \ \ T\subset\R^2\setminus\Sigma_{k+m-1}\}\,, \\
& \mathcal C_m 
= \{T=\eta+(-\widetilde{m}, \widetilde{m})^2~|~\eta\in\Z^2\}\,.
\end{align*}
Finally, by letting $w=w_{r,y'_0}$ be a weak solution to
\begin{equation}\tag{$\widetilde{{\rm BL}}^{(j),-}_{\,r,y_0}$}\label{BLj-ry0}
\left\{
\begin{array}{ll}
-\Delta w+\nabla s = 0\,, &y\in\Omega^{{\rm bl},-}_{r,y'_0} \\
\nabla\cdot w = 0\,, &y\in\Omega^{{\rm bl},-}_{r,y'_0} \\
w = 0\,, &y\in\partial\Omega^{{\rm bl},-}_{r,y'_0}\setminus\Sigma_{r,y'_0} \\
(-\partial_3 w+s{\bf e}_3)|_{y_3=0} = {\rm DN}(w|_{y_3=0})-{\bf e}_j\,,
\end{array}
\right.
\end{equation}
we extend $w_{r,y'_0}$ by zero to $\Omega^{{\rm bl},-}$ and denote by $E_{k}[w_{r,y'_0}]$ and $E_{T}[w_{r,y'_0}]$ its energy respectively on $\Omega^{{\rm bl},-}_{k}$ for $k\in\N$ and on $\{y\in\Omega^{{\rm bl},-}~|~y'\in T\,, \ \gamma(y') < y_3 < 0\}$ for $T\in\mathcal C_m$:
\begin{align*}
E_{k}[w_{r,y'_0}] = \int_{\Omega^{{\rm bl},-}_{k}} |\nabla w_{r,y'_0}|^2\,,\qquad
E_{T}[w_{r,y'_0}] = \int_{T} |\nabla w_{r,y'_0}|^2\,.
\end{align*} 
Before giving a detailed proof, let us sketch the argument of Proposition \ref{prop.BL}. We approximate a solution $w$ of \eqref{eq1.subsec.proof.BL} by a sequence of functions $\{w_n\}\subset H^1(\Omega^{{\rm bl},-})^3$ where each $w_n$ satisfies $({\rm BL}^{(j),-}_{\,n,0})$ in $\Omega^{{\rm bl},-}_{n}$ and $w_n=0$ in $\Omega^{{\rm bl},-}\setminus\Omega^{{\rm bl},-}_{n}$. We note that, for any $r\in(0,\infty)$ and $y'_0\in\R^2$, the well-posedness of \eqref{BLj-ry0} follows from an energy estimate and from the Lax-Milgram lemma thanks to the positivity of ${\rm DN}$; see \eqref{est2.lem.DN} in Lemma \ref{lem.DN}. Hence we aim at getting a uniform estimate for $w_n$ of the type of the estimate \eqref{est2.lem.equiv.BL} for $v^{-}$.  This uniform bound will be proved using a local energy estimate, the Saint-Venant estimate \eqref{SVest.proof.prop.BL} below, for the equations \eqref{BLj-ry0}. The reader is referred to the pioneering work by Lady\v{z}enskaja and Solonnikov \cite{LS} considering such a locally uniform bound for the Navier-Stokes flow in a channel under the no-slip condition. In contrast, since our problem \eqref{BLj-ry0} involves a non-local operator ${\rm DN}$ as a boundary condition, an additional term $\sup_{T\in\mathcal C_{k,m}} E_{T}[w_{r,y'_0}]$ arises in \eqref{SVest.proof.prop.BL}, and hence, a careful analysis is needed when proving the uniform bound of $\{w_n\}$. Then the existence of a solution $w$ to \eqref{eq1.subsec.proof.BL} is verified by taking the weak limit of $\{w_n\}$.  Uniqueness follows from a variant of the Saint-Venant estimate \eqref{SVest.proof.prop.BL}.
%
\begin{proofx}{Proposition \ref{prop.BL}}
We basically follow the procedure and the notations in \cite[Section 5]{KP2}. We also use the argument in the proof of \cite[Proposition 10]{GM} when estimating the pressure.

\noindent {\bf (Existence)} Firstly we prove the following Saint-Venant estimate to the equations \eqref{BLj-ry0}: fix $r\in(0,\infty)$ and $y'_0\in\R^2$. Let $w_{r,y'_0}\in H^1(\Omega^{{\rm bl},-}_{r,y'_0})^3$ be a solution to \eqref{BLj-ry0}. Then for even $m=2\widetilde{m}\in\N$ with $\widetilde{m}>1$ and $k\in\N$ with $k\ge\widetilde{m}$, the zero-extension of $w_{r,y'_0}$ to $\Omega^{{\rm bl},-}$ satisfies
\begin{align}\tag{SV}\label{SVest.proof.prop.BL}
E_{k}[w_{r,y'_0}]
\le C^\ast \Big(k^2 + E_{k+m}[w_{r,y'_0}] - E_{k}[w_{r,y'_0}]
+ \frac{k^4}{m^6} \sup_{T\in\mathcal C_{k,m}} E_{T}[w_{r,y'_0}] \Big)\,,
\end{align}
where $C^\ast$ is independent of $r$, $y'_0$, $m$, and $k$, and depends on $\|\gamma\|_{W^{1,\infty}(\R^2)}$. In the proof of \eqref{SVest.proof.prop.BL}, we denote $(w_{r,y'_0},s_{r,y'_0})$ by $(w,s)$ and $\langle \cdot, \cdot \rangle_{\mathcal{D}', \mathcal{D}}$ by $\langle \cdot, \cdot \rangle$ in order to simplify notation. Let us take a cut-off function $\chi_k\in C^\infty_0(\R^2)$ satisfying
\begin{align*}
&\chi_k\in[0,1]\,, \quad
\supp\chi_k\subset (-k-1,k+1)^2\,, \\
&\chi_k=1 \ \, {\rm in} \ \, [-k,k]^2\,, \quad 
\|\chi_k\|_{W^{1,\infty}(\R^2)} \le 2\,.
\end{align*}
Then we test $w \chi_k^2$ against \eqref{BLj-ry0} to see that
\begin{align}\label{eq1.proof.prop.BL}
\begin{split}
\int_{\Omega^{{\rm bl},-}} |\nabla w|^2 \chi_k^2
&=
-2 \int_{\Omega^{{\rm bl},-}} \nabla w \cdot (w \nabla \chi_k) \chi_k 
+ \langle{\bf e}_j, w|_{y_3=0} \chi_k^2 \rangle \\
& \quad
- \langle{\rm DN}(w|_{y_3=0}), w|_{y_3=0} \chi_k^2 \rangle
+ 2 \int_{\Omega^{{\rm bl},-}} s (w\cdot\nabla \chi_k) \chi_k\,.
\end{split}
\end{align}
We estimate each term in the right-hand side of \eqref{eq1.proof.prop.BL}. For the first term we have
\begin{align}\label{est1.proof.prop.BL}
\begin{split}
\bigg|2 \int_{\Omega^{{\rm bl},-}} \nabla w \cdot (w \nabla \chi_k) \chi_k\bigg|
&\le
2 \bigg(\int_{\Omega^{{\rm bl},-}} |\nabla w|^2 \chi_k^2\bigg)^{\frac12}
\bigg(\int_{\Omega^{{\rm bl},-}} |w|^2 |\nabla \chi_k|^2\bigg)^{\frac12} \\
&\le
C \bigg(\int_{\Omega^{{\rm bl},-}} |\nabla w|^2 \chi_k^2\bigg)^{\frac12}
\big(E_{k+1}[w]-E_k[w]\big)^{\frac12}\,,
\end{split}
\end{align}
while for the second term we see that from Lemma \ref{lem.DN} (i), 
\begin{align}\label{est2.proof.prop.BL}
\begin{split}
|\langle{\bf e}_j, w|_{y_3=0} \chi_k^2 \rangle|
&\le 
Ck\|w|_{y_3=0} \chi_k^2\|_{H^\frac12(\R^2)} \\
&\le
Ck\|\nabla(w \chi_k^2)\|_{L^2(\Omega^{{\rm bl},-})} \\
&\le
Ck\big(E_{k+1}[w]-E_k[w]\big)^\frac12
+ Ck\bigg(\int_{\Omega^{-}} |\nabla w|^2 \chi_k^2 \bigg)^\frac12\,.
\end{split}
\end{align}
Here the trace theorem and the Poincar\'e inequality have been applied in the second line. Next we estimate the third term in the right-hand side of \eqref{eq1.proof.prop.BL} involving the nonlocal operator ${\rm DN}$. We split it into
\begin{align}\label{eq2.proof.prop.BL}
\begin{split}
-\langle{\rm DN}(w|_{y_3=0}), w|_{y_3=0} \chi_k^2 \rangle
&=
-\langle{\rm DN}(w|_{y_3=0} \chi_k^2), w|_{y_3=0} \chi_k^2 \rangle \\
&\quad
- \langle{\rm DN}(w|_{y_3=0}(\chi_{k+m-1}^2-\chi_k^2), w|_{y_3=0} \chi_k^2 \rangle \\
&\quad
- \langle{\rm DN}(w|_{y_3=0}(1-\chi_{k+m-1}^2)), w|_{y_3=0} \chi_k^2 \rangle\,.
\end{split}
\end{align}
Then from \eqref{est2.lem.DN} in Lemma \ref{lem.DN} we have 
\begin{align}\label{est3.proof.prop.BL}
\begin{split}
-\langle{\rm DN}(w|_{y_3=0}), w|_{y_3=0} \chi_k^2 \rangle
& \le
-\langle{\rm DN}(w|_{y_3=0}), w|_{y_3=0} \chi_k^2 \rangle \\
& \quad
+ \langle{\rm DN}(w|_{y_3=0} \chi_k^2), w|_{y_3=0} \chi_k^2 \rangle\,,
\end{split}
\end{align}
and from the definition in \eqref{def2.subsec.DN}, by taking $R\in(0,\infty)$ sufficiently large depending on $k$, we also have
\begin{align}\label{est4.proof.prop.BL}
\begin{split}
&|\langle{\rm DN}(w|_{y_3=0}(\chi_{k+m-1}^2-\chi_k^2), w|_{y_3=0} \chi_k^2 \rangle| \\
&\le 
C \|w|_{y_3=0}(\chi_{k+m-1}^2-\chi_k^2)\|_{H^{\frac12}(\R^2)}
\|w|_{y_3=0} \chi_k^2\|_{H^{\frac12}(\R^2)} \\
&\le
C \big(E_{k+m}[w]-E_k[w]\big)^{\frac12}
\bigg(\big(E_{k+1}[w]-E_k[w]\big)^\frac12
+ \bigg(\int_{\Omega^{{\rm bl},-}} |\nabla w|^2 \chi_k^2 \bigg)^\frac12 \bigg)\,,
\end{split}
\end{align}
where the trace theorem and the Poincar\'e inequality are used again to obtain the last line. In estimating the third term in the right-hand side of \eqref{eq2.proof.prop.BL}, the following estimate is useful: for even $m=2\widetilde{m}\in\N$ with $\widetilde{m}>1$ and $k\in\N$ with $k\ge\widetilde{m}$, one has
\begin{align}\label{int.est.proof.prop.BL}
\int_{\Sigma_{k+1}} 
\bigg(
\int_{\R^2} \frac{1-\chi^2_{k+m-1}(\tilde{y}')}{|y'-\tilde{y}'|^3}
|w(y',0)|\dd \tilde{y}' \bigg)^2 \dd y'
\le C\frac{k^4}{m^6} \sup_{T\in\mathcal C_{k,m}}E_T[w]\,,
\end{align}
where $C$ is a numerical constant. For the proof of \eqref{int.est.proof.prop.BL} we refer to \cite[Lemma 14]{KP2}. Then, by the definition in \eqref{def2.subsec.DN}, since the supports of $1-\chi^2_{k+m-1}$ and $\chi_k^2$ are disjoint and $m>4$, we obtain
\begin{align}\label{est5.proof.prop.BL}
\begin{split}
&|\langle{\rm DN}(w|_{y_3=0}(1-\chi_{k+m-1}^2)), w|_{y_3=0} \chi_k^2\rangle|\\
&\le
C\int_{\R^2}\int_{\R^2}
\frac{1-\chi_{k+m-1}^2(\tilde{y}')}{|y'-\tilde{y}'|^3}
|w(\tilde{y}',0)| |w(y',0) \chi_k^2(y')| \dd \tilde{y}'\dd y' \\
&\le
C\bigg(\int_{\Sigma_{k+1}} \bigg(\int_{\R^2}
\frac{1-\chi^2_{k+m-1}(\tilde{y}')}{|y'-\tilde{y}'|^3} 
|w(y',0)|\dd \tilde{y}' \bigg)^2 \dd y' \bigg)^{\frac12} \\
& \qquad
\times \bigg(\int_{\Sigma_{k+1}} |w(y',0)|^2 \chi_k^2 \dd y' \bigg)^{\frac12} \\
&\le
C\frac{k^2}{m^3}
\Big(\sup_{T\in \mathcal C_{k,m}} E_T[w] \Big)^\frac12
\bigg(\big(E_{k+1}[w]-E_k[w]\big)^\frac12
+ \bigg(\int_{\Omega^{{\rm bl},-}} |\nabla w|^2 \chi_k^2 \bigg)^\frac12 \bigg)\,.
\end{split}
\end{align}
Thus, by applying \eqref{est3.proof.prop.BL}, \eqref{est4.proof.prop.BL}, and \eqref{est5.proof.prop.BL} to \eqref{eq2.proof.prop.BL}, we have
\begin{align}\label{est6.proof.prop.BL}
\begin{split}
&-\langle{\rm DN}(w|_{y_3=0}), w|_{y_3=0} \chi_k^2 \rangle \\
&\le
C\bigg(
\big(E_{k+m}[w]-E_k[w]\big)^\frac12
+ \frac{k^2}{m^3} \Big(\sup_{T\in \mathcal C_{k,m}} E_T[w] \Big)^\frac12
\bigg) \\
&\qquad
\times \bigg(\big(E_{k+1}[w]-E_k[w]\big)^\frac12
+ \bigg(\int_{\Omega^{{\rm bl},-}} |\nabla w|^2 \chi_k^2 \bigg)^{\frac12} \bigg)\,.
\end{split}
\end{align}
Now we estimate the last term in the right-hand side of \eqref{eq1.proof.prop.BL}. The fundamental theorem of calculus leads to
\begin{align*}
\begin{split}
s(y) = s(y',0) - \int_{y_3}^{0} \partial_3 s(y',t)\dd t\,.
\end{split}
\end{align*}
From $s(y',0)=\partial_3 w_3(y',0) + {\rm DN}(w|_{y_3=0})\cdot {\bf e}_3$ and the equation $\partial_3 s=\Delta w_3$, one finds
\begin{align*}
\begin{split}
s(y) = {\rm DN}(w|_{y_3=0})\cdot {\bf e}_3 
- \sum_{j=1}^{2} \partial_j \bigg(\int_{y_3}^{0} \partial_j w_3(y',t)\dd t\bigg)
+ \partial_3 w_3(y)\,.
\end{split}
\end{align*}
Hence, by setting $W(y')=\int_{\gamma(y')}^0 w(y',y_3)\dd y_3$ and using integration by parts, we see that
\begin{align}\label{eq3.proof.prop.BL}
\begin{split}
& \int_{\Omega^{{\rm bl},-}} s (w\cdot\nabla \chi_k) \chi_k \\
&=
\langle {\rm DN}(w|_{y_3=0})\cdot {\bf e}_3, (W\cdot\nabla \chi_k) \chi_k \rangle \\
&\quad
+ \sum_{j=1}^{2} \int_{\Omega^{{\rm bl},-}}
\bigg(\int_{y_3}^{0} \partial_j w_3(y',t)\dd t \bigg)
\partial_j \big((w\cdot\nabla \chi_k) \chi_k\big) \dd y \\
&\quad
+ \int_{\Omega^{{\rm bl},-}}
\partial_3 w_3 (w\cdot\nabla \chi_k) \chi_k\,.
\end{split}
\end{align}
For the first term in the right-hand side of \eqref{eq3.proof.prop.BL}, in similar ways as in \eqref{eq2.proof.prop.BL}, \eqref{est4.proof.prop.BL}, and \eqref{est5.proof.prop.BL}, we have
\begin{align}\label{est8.proof.prop.BL}
\begin{split}
& |\langle {\rm DN}(w|_{y_3=0})\cdot {\bf e}_3, (W\cdot\nabla \chi_k) \chi_k \rangle| \\
&\le
|\langle {\rm DN}(w|_{y_3=0} \chi_{k}^2)\cdot {\bf e}_3, (W\cdot\nabla \chi_k) \chi_k \rangle| \\
&\quad
+ |\langle {\rm DN}(w|_{y_3=0} (\chi_{k+m-1}^2 - \chi_{k}^2))\cdot {\bf e}_3, (W\cdot\nabla \chi_k) \chi_k \rangle| \\
&\quad
+ |\langle {\rm DN}(w|_{y_3=0} (1-\chi_{k+m-1}^2))\cdot {\bf e}_3, 
(W\cdot\nabla \chi_k) \chi_k \rangle| \\
&\le
C \bigg(\big(E_{k+1}[w]-E_k[w]\big)^\frac12
+ \bigg(\int_{\Omega^{{\rm bl},-}} |\nabla w|^2 \chi_k^2 \bigg)^\frac12 \bigg)
\big(E_{k+1}[w]-E_k[w]\big)^\frac12 \\
&\quad
+ C\big(E_{k+m}[w]-E_k[w]\big)^\frac12
\big(E_{k+1}[w]-E_k[w]\big)^\frac12 \\
&\quad
+ C\frac{k^2}{m^3}
\Big(\sup_{T\in \mathcal C_{k,m}} E_T[w] \Big)^\frac12
\big(E_{k+1}[w]-E_k[w]\big)^\frac12\,,
\end{split}
\end{align}
while for the second and the last terms we have
\begin{align}\label{est9.proof.prop.BL}
\begin{split}
&\sum_{j=1}^{2} \bigg|\int_{\Omega^{{\rm bl},-}}
\bigg(\int_{y_3}^{0} \partial_j w_3(y',t)\dd t \bigg)
\partial_j \big((w\cdot\nabla \chi_k) \chi_k\big) \dd y\bigg| \\
& \quad
+ \bigg|\int_{\Omega^{{\rm bl},-}}
\partial_3 w_3 (w\cdot\nabla \chi_k) \chi_k\bigg| \\
&\le
C\big(E_{k+1}[w]-E_k[w]\big)\,.
\end{split}
\end{align}
Therefore, by applying \eqref{est8.proof.prop.BL} and \eqref{est9.proof.prop.BL} to \eqref{eq3.proof.prop.BL}, we obtain
\begin{equation}\label{est10.proof.prop.BL}
\begin{aligned}
&\bigg|2 \int_{\Omega^{{\rm bl},-}} s (w\cdot\nabla \chi_k) \chi_k \bigg| \\
&\le
C \big(E_{k+1}[w]-E_k[w]\big)^\frac12
\bigg(\big(E_{k+1}[w]-E_k[w]\big)^\frac12
+ \bigg(\int_{\Omega^{{\rm bl},-}} |\nabla w|^2 \chi_k^2 \bigg)^\frac12 \bigg) \\
&\quad
+ C\big(E_{k+m}[w]-E_k[w]\big)^\frac12
\big(E_{k+1}[w]-E_k[w]\big)^\frac12 \\
&\quad
+ C\frac{k^2}{m^3}
\Big(\sup_{T\in \mathcal C_{k,m}} E_T[w] \Big)^\frac12
\big(E_{k+1}[w]-E_k[w]\big)^\frac12\,.
\end{aligned}
\end{equation}
Hence, by using \eqref{est1.proof.prop.BL}, \eqref{est2.proof.prop.BL}, \eqref{est6.proof.prop.BL}, \eqref{est10.proof.prop.BL}, and $E_{k+1}[w]\le E_{k+m}[w]$, we estimate the right-hand side of \eqref{eq1.proof.prop.BL} as
\begin{align*}
&\int_{\Omega^{{\rm bl},-}} |\nabla w|^2 \chi_k^2 \\
&\le 
Ck\big(E_{k+m}[w]-E_k[w]\big)^\frac12 \\
& \quad
+ C\Big(k + \big(E_{k+m}[w]-E_k[w]\big)^{\frac12}\Big)
\bigg(\int_{\Omega^{{\rm bl},-}} |\nabla w|^2 \chi_k^2\bigg)^{\frac12} \\
&\quad
+ C\bigg(
\big(E_{k+m}[w]-E_k[w]\big)^\frac12
+ \frac{k^2}{m^3} \Big(\sup_{T\in \mathcal C_{k,m}} E_T[w] \Big)^\frac12
\bigg) \\
&\qquad\quad
\times \bigg(\big(E_{k+m}[w]-E_k[w]\big)^\frac12
+ \bigg(\int_{\Omega^{{\rm bl},-}} |\nabla w|^2 \chi_k^2 \bigg)^{\frac12} \bigg)\,,
\end{align*}
where $C$ is independent of $r$, $y'_0$, $m$, and $k$. Then the Young inequality leads to the assertion \eqref{SVest.proof.prop.BL}.

Applying the estimate \eqref{SVest.proof.prop.BL}, we prove the existence of a weak solution $w$ to \eqref{eq1.subsec.proof.BL} in the class
\begin{align}\label{goal1.proof.prop.BL}
\sup_{\eta\in\Z^2} \int_{\eta+(0,1)^2}
\int_{\gamma(y')}^0 
|\nabla w(y',y_3)|^2 \dd y_3 \dd y'
< +\infty\,.
\end{align}
Let $\{w_n\}_{n=N}^\infty$ be a sequence of solutions to $({\rm BL}^{(j),-}_{\,n,0})$ with a fixed $N\in\N$. Then our goal is to prove an estimate of the type \eqref{goal1.proof.prop.BL} for $\{w_n\}_{n=N}^\infty$ uniformly bounded in $n$. More precisely, by setting
\begin{align}\label{def1.proof.prop.BL}
\begin{split}
&A=(C^\ast+1) \sum_{p=0}^\infty\Big(\frac{C^*}{C^*+1}\Big)^{p+1} (2p+1)^2\,,\\ &B=\sum_{p=0}^\infty \Big(\frac{C^*}{C^*+1}\Big)^{p+1} (2p+1)^4\,,
\end{split}
\end{align}
where $C^\ast$ is the numerical constant in \eqref{SVest.proof.prop.BL}, we will show that there exists an even number $m=2\widetilde{m}\in\N$ with $\widetilde{m}>1$ such that for all $l\in\N$ and $n=lm$ we have 
\begin{align}
\label{goal2.proof.prop.BL}
\sup_{\eta\in\Z^2}
\int_{\eta+(0,1)^2}
\int_{\gamma(y')}^0 |\nabla w_n(y',y_3)|^2 \dd y_3 \dd y'
\le \frac{Am^2}{2}\,.
\end{align}
Here and in the rest of the proof we do not distinguish $w_n$ on $\Omega^{{\rm bl},-}_{n}$ and its zero-extension to $\Omega^{{\rm bl},-}$. Although the uniform estimate \eqref{goal2.proof.prop.BL} can be proved in the same way as in \cite[Subsection 2.2]{KP2}, we give a brief description of it for the sake of completeness. Firstly we choose $\widetilde{m}\in\N$ sufficiently large so that 
\begin{equation}\label{def2.proof.prop.BL}
\widetilde{m} > \max\Big\{1\,, \sqrt{\frac{B}{32}}\Big\}\,.
\end{equation}
Fix $l\in\N$. Then, since $w_n$ is supported in $\Omega^{{\rm bl},-}_{n}$, by the choice of $n=lm$ there exists $T^\ast\in \mathcal C_{m}$ such that $T^\ast\subset(-n,n)^2$ and $E_{T^*}[w_n]=\sup_{T\in\mathcal C_{m}} E_T[w_n]$. Hence we aim at showing that
\begin{align}\label{goal3.proof.prop.BL}
E_{T^*}[w_n]
\le \frac{Am^2}{2}\,,
\end{align}
which immediately leads to \eqref{goal2.proof.prop.BL} by the definition. On the other hand, the existence of $T^\ast$ implies that there is a point $\eta^\ast\in\R^2$ such that $T^\ast=\eta^\ast+(-\widetilde{m},\widetilde{m})^2$ and $\eta^\ast\in(-n,n)^2$. We define the translation $w^*_n(y)=w_n(y'+\eta^*,y_3)$. Then $E_{\widetilde{m}}[w^*_n]=E_{T^*}[w_n]$ and $w^*_n$ is a solution to $({\rm BL}^{(j),-}_{\,n,-\eta^*})$ replacing and the boundary $\gamma$ by $\gamma^*(y')=\gamma(y'+\xi^*)$, and therefore, $w^*_n$ satisfies
\begin{align}\label{SVest1.proof.prop.BL}
E_{k}[w^*_n]
\le C^\ast \Big(k^2 + E_{k+m}[w^*_n] - E_{k}[w^*_n]
+ \frac{k^4}{m^6} \sup_{T\in\mathcal C_{k,m}} E_{T}[w^*_n] \Big)
\end{align}
for all $k\ge \widetilde{m}$ with the same constant  $C^\ast$ as in \eqref{SVest.proof.prop.BL}. In the following we prove \eqref{goal3.proof.prop.BL} replacing $E_{T^*}[w_n]$ by $E_{\widetilde{m}}[w^*_n]$ by downward induction in \eqref{SVest1.proof.prop.BL} for $k=2n+\widetilde{m}, 2n+\widetilde{m}-m,\ldots,\widetilde{m}$. When $k=2n+\widetilde{m}$, since $w^*_n$ is supported in $\Omega^{-}_{2n}$ by its definition, we have $E_{k+m}[w^*_n]=E_{k}[w^*_n]$ and $E_T[w^*_n]=0$ for any $T\in\mathcal C_{2n+\widetilde{m},m}$. Thus the estimate \eqref{SVest1.proof.prop.BL} implies that
\begin{equation*}
E_{2n+\widetilde{m}}[w^*_n]
\le C^* (2n+\widetilde{m})^2\,.
\end{equation*}
When $k=2n+\widetilde{m}-m$, since $E_T[w^*_n]=0$ for any $T\in\mathcal C_{2n+\widetilde{m}-m,m}$ again by $\widetilde{m}>1$, we have
\begin{align*}
E_{2n+\widetilde{m}-m}[w^*_n]
& \le
\frac{C^*}{C^*+1} 
\Big((2n+\widetilde{m}-m)^2 + E_{2n+\widetilde{m}}[w^*_n]\Big) \\
& \le
\Big(\frac{C^*}{C^*+1}\Big) (2n+\widetilde{m}-m)^2 
+ \Big(\frac{C^*}{C^*+1}\Big) C^* (2n+\widetilde{m})^2\,.
\end{align*}
When $k=2n+\widetilde{m}-\widetilde{l}m$ with $\widetilde{l}\in\{2,\cdots,2l\}$, from $\sup_{T\in\mathcal C_{k,m}} E_{T}[w^*_n]\le E_{\widetilde{m}}[w^*_n]$ we have 
\begin{align*}
E_{2n+\widetilde{m}-\widetilde{l}m}[w^*_n] 
&\le
(C^*+1)
\sum_{p=0}^{\widetilde{l}}
\Big(\frac{C^*}{C^*+1}\Big)^{p+1} \big(2n+\widetilde{m}-(\widetilde{l}-p)m\big)^2 \\
& \quad
+ \frac{E_{\widetilde{m}}[w^*_n]}{m^2} \sum_{p=0}^{\widetilde{l}-2}
\Big(\frac{C^*}{C^*+1}\Big)^{p+1} 
\frac{\big(2n+\widetilde{m}-(\widetilde{l}-p)m\big)^4}{m^4}\,,
\end{align*}
which implies, when $\widetilde{l}=2l$, by $2n=2lm$ and $m=2\widetilde{m}$ that
\begin{align*}
E_{\widetilde{m}}[w^*_n]
&\le
\widetilde{m}^2 (C^*+1)
\sum_{p=0}^{2l}
\Big(\frac{C^*}{C^*+1}\Big)^{p+1} (2p+1)^2 \\
& \quad
+ \frac{E_{\widetilde{m}}[w^*_n]}{64\widetilde{m}^2} \sum_{p=0}^{2l-2}
\Big(\frac{C^*}{C^*+1}\Big)^{p+1} (2p+1)^4 \\
&\le
A \widetilde{m}^2 + \frac{B}{64\widetilde{m}^2} 
E_{\widetilde{m}}[w^*_n]\,.
\end{align*}
Hence we obtain \eqref{goal3.proof.prop.BL} from the choice of $\widetilde{m}$ in \eqref{def2.proof.prop.BL}, and therefore, the desired estimate \eqref{goal2.proof.prop.BL}.

Now we can prove the existence of a weak solution $w$ to \eqref{eq1.subsec.proof.BL} in the class \eqref{goal1.proof.prop.BL}. After taking a sequence $\{w_n\}_{n=m}^\infty$ of weak solutions to $({\rm BL}^{(j),-}_{\,n,0})$, we extract a subsequence which converges to a function $w$ satisfying both \eqref{goal1.proof.prop.BL} and the weak formulation of \eqref{eq1.subsec.proof.BL} thanks to \eqref{goal2.proof.prop.BL}. We note that the convergence of the operator ${\rm DN}$ on the subsequence can be inferred from Lemma \ref{lem.DN} (ii). Hence we have proved the existence of a weak solution of \eqref{BLj} 
in the class \eqref{est1.prop.BL} from Lemma \ref{lem.equiv.BL}.

\noindent {\bf (Uniqueness)} Suppose that $w\in H^{1}_{loc}(\Omega^{{\rm bl},-})$ satisfies both
\begin{align}\label{est11.proof.prop.BL}
\sup_{\eta\in\Z^2} \int_{\eta+(0,1)^2}
\int_{\gamma(y')}^0 
|\nabla w(y',y_3)|^2 \dd y_3 \dd y'
\le C_0\,,
\end{align}
with some bound $C_0\in(0,\infty)$, and
\begin{equation}\label{eq5.proof.prop.BL}
\left\{
\begin{array}{ll}
-\Delta w+\nabla s = 0\,, &y\in\Omega^{{\rm bl},-} \\
\nabla\cdot w = 0\,, &y\in\Omega^{{\rm bl},-} \\
w(y',\gamma(y')) = 0 \\
(-\partial_3 w+r{\bf e}_3)|_{y_3=0} = {\rm DN}(w|_{y_3=0})\,.
\end{array}
\right.
\end{equation}
Then we aim at proving $w=0$. By following the existence part and by defining $E_{k}[w]$ and $E_{T}[w]$ in an obvious way, we see that for even $m=2\widetilde{m}\in\N$ with $\widetilde{m}>1$ and $k\in\N$ with $k\ge\widetilde{m}$, 
\begin{align}\label{SVest2.proof.prop.BL}
E_k[w]
\le C^\ast \Big(E_{k+m}[w] - E_{k}[w]
+ \frac{k^4}{m^6} \sup_{T\in\mathcal C_{k,m}} E_{T}[w] \Big)
\end{align}
holds with the same constant $C^\ast$ as in \eqref{SVest.proof.prop.BL}. In particular, we note that the estimate \eqref{int.est.proof.prop.BL} is valid even if the function $w$ in \eqref{int.est.proof.prop.BL} does not vanish outside a bounded domain. Since $\sup_{T\in\mathcal C_{m}} E_T[w]$ is finite, for any $\ep\in(0,1)$ there exists $T^*_\ep=\eta^*_\varepsilon+(-\widetilde{m},\widetilde{m})^2\in\mathcal C_{m}$ with some $\eta^\ast_\ep\in\Z^2$ such that 
\begin{align}\label{est12.proof.prop.BL}
\sup_{T\in\mathcal C_{m}}E_T[w]-\ep
\le E_{T^\ast_\ep}[w]
\le \sup_{T\in\mathcal C_{m}}E_T[w]\,.
\end{align}
As was done in the existence part of the proof, by defining $w^*(y)=w(y'+\eta^*_\ep,y_3)$, we have 
$\sup_{T\in\mathcal C_{m}}E_T[w^\ast]=\sup_{T\in\mathcal C_{m}}E_T[w]$, 
$E_{\widetilde{m}}[w^*]=E_{T^\ast_\ep}[w]$, and 
\begin{align}\label{SVest3.proof.prop.BL}
E_k[w^*]
\le 
C^* \Big(E_{k+m}[w^*] - E_k[w^*] 
+ \frac{k^4}{m^6} \sup_{T\in\mathcal C_{k,m}} E_T[w^*] \Big)
\end{align}
with the same $C^\ast$ as in \eqref{SVest.proof.prop.BL} again. Fix $\widetilde{m}\in\N$ satisfying \eqref{def2.proof.prop.BL}, let $l\in\N$, and set $m=2\widetilde{m}$ and $n=lm$. Then we derive an estimate of $E_{\widetilde{m}}[w^*]$ by downward induction in \eqref{SVest3.proof.prop.BL} for $k=n+\widetilde{m}, n+\widetilde{m}-m,\ldots,\widetilde{m}$. When $k=n+\widetilde{m}$ we have $E_{n+\widetilde{m}}[w^\ast] \le C_0 (n+\widetilde{m})^2$ from the assumption \eqref{est11.proof.prop.BL} without using \eqref{SVest3.proof.prop.BL}. When $k=n+\widetilde{m}-\widetilde{l}m$ with $\widetilde{l}\in\{1,\cdots,l\}$, we have 
\begin{align*}
E_{n+\widetilde{m}-\widetilde{l}m}[w^\ast] 
&\le
\Big(\frac{C^*}{C^*+1}\Big)^{\widetilde{l}} E_{n+\widetilde{m}}[w^\ast] \\
& \quad
+ \frac{1}{m^2} \sup_{T\in\mathcal C_{k,m}} E_T[w^*]
\sum_{p=0}^{\widetilde{l}-1}
\Big(\frac{C^*}{C^*+1}\Big)^{p+1} 
\frac{\big(n+\widetilde{m}-(\widetilde{l}-p)m\big)^4}{m^4}\,,
\end{align*}
and hence, when $\widetilde{l}=l$, we have by $n=lm$, $m=2\widetilde{m}$, and \eqref{est12.proof.prop.BL}, 
\begin{align}\label{est13.proof.prop.BL}
E_{\widetilde{m}}[w^\ast]
&\le
C_0 \widetilde{m}^2 \Big(\frac{C^*}{C^*+1}\Big)^{l} (2l+1)^2
+ \frac{E_{\widetilde{m}}[w^*] + \ep}{64\widetilde{m}^2}
\sum_{p=0}^{l-1}
\Big(\frac{C^*}{C^*+1}\Big)^{p+1} (2p+1)^4\,.
\end{align}
Then, from the definition of $B$ in \eqref{def1.proof.prop.BL} and the choice of $\widetilde{m}$ in \eqref{def2.proof.prop.BL}, we see that
\begin{align*}
E_{\widetilde{m}}[w^\ast]
\le 
\frac{B\ep}{32\widetilde{m}^2}\,,
\end{align*}
by taking the limit $l\to\infty$ in \eqref{est13.proof.prop.BL}. Then we have $E_{\widetilde{m}}[w^\ast]=0$ and hence $\sup_{T\in\mathcal C_{m}}E_T[w]=0$, which finally implies $w=0$. Therefore we have proved the uniqueness of weak solutions to \eqref{BLj} in the class \eqref{est1.prop.BL} from Lemma \ref{lem.equiv.BL}. The proof of Proposition \ref{prop.BL} is complete.
\end{proofx}
%
\subsection{Periodic case}\label{subsec.per.BL}
%
\noindent We investigate the asymptotic behavior of the boundary layer corrector at spatial infinity when the boundary is periodic. Similar arguments can be found for instance in Achdou, Pironneau, and Valentin \cite{APV}, J\"{a}ger and Mikeli\'{c} \cite{JM2}, and Amirat, Bodart, De Maio, and Gaudiello \cite{ABDG}.
%
\begin{proposition}\label{prop.per.BL}
Fix $j\in\{1,2\}$ and let $\gamma\in W^{1,\infty}(\R^2)$ be $2\pi$-periodic in each variable. Then the weak solution $(v^{(j)}, q^{(j)})$ to \eqref{BLj} provided by Proposition \ref{prop.BL} satisfies the following properties. \\
\noindent {\rm (i)} There exists a constant vector field $\alpha^{(j)}=(\alpha^{(j)}_1, \alpha^{(j)}_2, 0)^\top\in\R^3$ such that 
\begin{align}\label{est1.prop.per.BL}
|v^{(j)}(y)-\alpha^{(j)}| + y_3 |q^{(j)}(y)|
\le C\|v^{(j)}(\cdot,0)\|_{L^2((0,2\pi)^2)}\,e^{-\frac{y_3}{2}}\,, \quad y_3>1\,,
\end{align}
where $C$ is a numerical constant. \\
\noindent {\rm (ii)} The $2\times2$ matrix ${\overline M}\in\R^{2\times2}$ defined by ${\overline M}=(\alpha^{(j)}_i)_{1\le i,j\le 2}$ is symmetric and positive definite.
\end{proposition}
%
\begin{proof}
(i) We give the proof only when $j=1$ since the case $j=2$ can be treated in a similar manner. The proof is based on the Fourier series expansion of $(v^{(1)}, q^{(1)})$ in $(y_1, y_2)$. For a $2\pi$-periodic function $f(y)$ in $y_1$ and $y_2$, we define $\hat{f}_k=\hat{f}_k(y_3)$ for $k=(k_1,k_2)\in\Z^2$ by
\begin{align*}
\hat{f}_k(y_3) 
= \frac{1}{(2\pi)^2} \int_{(0,2\pi)^2} f(y',y_3) e^{-ik\cdot y'} \dd y'\,,\quad y_3\ge0\,.
\end{align*}
Let us write $(v,q)=(v^{(1)},q^{(1)})$ and $\alpha^{(1)}=\alpha$ for simplicity. Since $(v, q)|_{y_3>0}$ solves the equations \eqref{S.half} in Subsection \ref{subsec.half.prob} with the boundary data $b(y')=v(y',0)$, in a similar way as in the uniqueness step in the proof of Proposition \ref{prop.WP.S.half}, we obtain the Fourier series representation of $(v, q)|_{y_3>0}$:
\begin{equation}\label{rep1.proof.prop.per.BL}
\begin{split}
v(y) & = 
\hat{b}_0
+ \sum_{k\in\Z^2\setminus\{(0,0)\}}
\bigg( \hat{b}_k e^{-|k|y_3} 
+ \begin{pmatrix} -ik \\ |k| \end{pmatrix}
(\hat{b}_{k,3} - i \frac{k}{|k|}\cdot \hat{b}'_k) y_3 e^{-|k|y_3}
\bigg) e^{ik\cdot y'}\,, \\
q(y) & = \sum_{k\in\Z^2\setminus\{(0,0)\}} 
\bigg( 2 |k| (\hat{b}_{k,3} - i \frac{k}{|k|}\cdot \hat{b}'_k) e^{-|k|y_3} \bigg) e^{ik\cdot y'}\,.
\end{split}
\end{equation}
Then by a direct computation and using $\|b\|_{L^2((0,2\pi)^2)}=\|v(\cdot, 0)\|_{L^2((0,2\pi)^2)}$ we see that
\begin{equation}\label{est1.proof.prop.per.BL}
\begin{split}
& |v(y) - \alpha|
\le C \|v(\cdot, 0)\|_{L^2((0,2\pi)^2)} e^{-\frac{y_3}{2}}\,, \quad y_3>1\,, \\
& |q(y)|
\le C \|v(\cdot, 0)\|_{L^2((0,2\pi)^2)} \frac{e^{-\frac{y_3}{2}}}{y_3}\,, \quad y_3>1\,,
\end{split}
\end{equation}
where we have set $\alpha=\hat{b}_{0}$ and $C$ is a numerical constant. Hence the proof of (i) is complete if we show $\alpha_3=\hat{b}_{0,3}=0$. From $\nabla \cdot (v + y_3 {\bf e}_1)=0$, we have for any $t\in(0,\infty)$,
\begin{align*}
\int_{(0,t)^2} \int_{\gamma(y')}^{0} 
\big( \partial_1 (v_1(y) + y_3) + \partial_2 v_2(y) + \partial_3 v_3(y) \big) \dd y_3 \dd y' = 0\,.
\end{align*}
Then an integration by parts leads to
\begin{align*}
&\int_{(0,t)^2} b_3(y',0) \dd y' \\
& = 
-\int_0^t \bigg(\int_{\gamma(t,y_2)}^0 v_1(t,y_2,y_3) \dd y_3 
- \int_{\gamma(0,y_2)}^0 v_1(0,y_2,y_3) \dd y_3 \bigg) \dd y_2 \\
& \quad
- \frac{t\big(\gamma(0,y_2)^2 - \gamma(t,y_2)^2 \big)}{2} \\
& \quad
- \int_0^t \bigg(\int_{\gamma(y_1,t)}^0 v_2(y_1,t,y_3) \dd y_3 
- \int_{\gamma(y_1,0)}^0 v_2(y_1,0,y_3) \dd y_3 \bigg) \dd y_1\,.
\end{align*}
Thus, by setting $t=2\pi N$ with $N\in\N$, we see that $N^2 \,\hat{b}_{0,3} = O(N)$ and consequently that $\hat{b}_{0,3}=0$ in the limit $N\to\infty$. Hence we obtain \eqref{est1.prop.per.BL} from \eqref{est1.proof.prop.per.BL} and the proof of (i) is complete. \\
\noindent (ii) Firstly we express $\alpha^{(j)}_i$ in terms of $v^{(j)}(y',0)$. For $\alpha^{(j)}_j$ from the equations we have
\begin{align*}
\int_{(0,2\pi)^2} \int_{\gamma(y')}^{0} 
(-\Delta v^{(j)}(y) + \nabla q^{(j)}(y)) \cdot (v^{(j)}(y) + y_3 {\bf e}_j) \dd y_3 \dd y' = 0\,.
\end{align*}
Then by integration by parts we get
\begin{align*}
(2\pi)^2\alpha^{(j)}_j& =\int_{(0,2\pi)^2} v^{(j)}_j(y',0) \dd y'\\
& = \int_{(0,2\pi)^2} \int_{\gamma(y')}^{0} |\nabla (v^{(j)}(y) + y_3 {\bf e}_j)|^2 \dd y_3 \dd y' \\
& \quad
+ \int_{(0,2\pi)^2} \big(-\partial_3 v^{(j)}(y',0) + q^{(j)}(y',0) {\bf e}_3 \big) 
\cdot v^{(j)}(y',0) \dd y'\\
& = \int_{(0,2\pi)^2} \int_{\gamma(y')}^{0} |\nabla (v^{(j)}(y) + y_3 {\bf e}_j)|^2 \dd y_3 \dd y' \\
& \quad
+\langle {\rm DN_{per}} (v^{(j)}(y',0)),v^{(j)}(y',0) \rangle_{H^{-\frac12}, H^{\frac12}}\,.
\end{align*}
Note that the Dirichlet-to-Neumann operator ${\rm DN_{\rm per}}$ can be represented as
\begin{align*}
{\rm DN_{per}} (v^{(j)}(y',0))
= \sum_{k\in\Z^2\setminus\{(0,0)\}}
\bigg( |k| \hat{b}^{(j)}_k 
+ \begin{pmatrix} ik \\ |k| \end{pmatrix}
\left(\hat{b}^{(j)}_{k,3} 
- i \frac{k}{|k|}\cdot \hat{b}^{(j)'}_k\right) \bigg) e^{ik\cdot y'}
\end{align*}
in this periodic setting by using \eqref{rep1.proof.prop.per.BL}. For $\alpha^{(1)}_2$ and $\alpha^{(2)}_1$ from the equations again we have
\begin{align*}
\int_{(0,2\pi)^2} \int_{\gamma(y')}^{0} 
(-\Delta v^{(2)}(y) + \nabla q^{(2)}(y)) \cdot (v^{(1)}(y) + y_3 {\bf e}_1) \dd y_3 \dd y' = 0\,.
\end{align*}
Then by integration by parts we have
\begin{equation*}
\begin{split}
(2\pi)^2 \alpha^{(1)}_2 & = \int_{(0,2\pi)^2} v^{(1)}_2(y',0) \dd y' \\
& = 
\int_{(0,2\pi)^2} \int_{\gamma(y')}^{0} 
\nabla (v^{(2)}(y) + y_3 {\bf e}_2) \cdot \nabla (v^{(1)}(y) + y_3 {\bf e}_1) \dd y_3 \dd y' \\
& \quad
+ \langle {\rm DN_{per}}(v^{(2)}(y',0)),v^{(1)}(y',0) \rangle_{H^{-\frac12}, H^{\frac12}}\\
&= (2\pi)^2 \alpha^{(2)}_1\,,
\end{split}
\end{equation*}
where we have used the relation
\begin{align*}
\langle {\rm DN_{per}} (v^{(2)}(y',0)),v^{(1)}(y',0) \rangle_{H^{-\frac12}, H^{\frac12}}
= \langle {\rm DN_{per}} (v^{(1)}(y',0)),v^{(2)}(y',0) \rangle_{H^{-\frac12}, H^{\frac12}}\,,
\end{align*}
which can be verified by a direct computation. 
Let us take $X=(X_1,X_2)^\top\in\R^2$. Then, by setting $u=X_1 v^{(1)}+ X_2 v^{(2)}$ and $a(y')=u(y',0)$, we calculate $X^\top \overline MX$, ${\overline M}=(\alpha^{(j)}_i)_{1\le i,j\le 2}$, as 
\begin{equation}\label{est2.proof.prop.per.BL}
\begin{split}
X^\top \overline MX
&= \frac{1}{(2\pi)^2} \int_{(0,2\pi)^2} \int_{\gamma(y')}^{0} 
|\nabla (u(y) + y_3 (X,0)^\top)|^2 \dd y_3 \dd y' \\
& \quad
+ \frac{1}{(2\pi)^2}\langle {\rm DN_{per}}(u(y',0)),u(y',0)\rangle_{H^{-\frac12}, H^{\frac12}}\,,
\end{split}
\end{equation}
which implies $X^\top \overline MX\ge0$ by non-negativity of the Dirichlet-to-Neumann operator. Next let us assume that $X^\top\overline MX=0$ for some $X\in\R^2\setminus\{0\}$. Then from \eqref{est2.proof.prop.per.BL} we see that $u=-y_3 (X,0)^\top+C$ with some $C\in\R^3$ on $\Omega^{{\rm bl},-}$, and on the other hand, that $u=\hat{a}_{0}$ on $\R^3_+$ from the representation \eqref{rep1.proof.prop.per.BL} since $\hat{a}_k=0$ for all $k\in\Z^2\setminus\{(0,0)\}$. This contradicts the fact that $u$ is smooth away from the boundary due to the nontriviality of $X$. The proof of Proposition \ref{prop.per.BL} is complete.
\end{proof}
%
\subsection{Some useful estimates}\label{subsec.useful.BL}
%
\noindent In this subsection we prove an easy lemma useful in estimating $v^{(j)}$.
%
\begin{lemma}\label{lem.est.boundarylayer}
Fix $j\in\{1,2\}$ and let $\ep\in(0,1)$ and $r\in[\ep,1]$. Then we have
\begin{align}\label{est1.lem.est.boundarylayer}
& \int_{B^\ep_{r,+}(0)}
\big|(\nabla_y v^{(j)})(\frac{x}{\ep}) \big|^2 \dd x
\le
C \ep r^2\,,
\end{align}
and for $m\in\{0,1,2\}$,
\begin{align}\label{est2.lem.est.boundarylayer}
\int_{B^\ep_{r,+}(0)}
\big|v^{(j)}(\frac{x}{\ep}) \big|^{2+m} \dd x
\le
\frac{Cr^{4-\frac{m}2}}{\ep^{1+\frac{m}2}}\,,
\end{align}
where the constant $C$ is independent of $\ep$ and $r$ and depends on $\|\gamma\|_{W^{1,\infty}(\R^2)}$.
\end{lemma}
%
\begin{proof}
By the change of variables $y=x/\ep$, we see that from \eqref{est1.prop.BL} in Proposition \ref{prop.BL}, 
\begin{align}\label{est1.proof.lem.est.boundarylayer}
\begin{split}
\int_{B^\ep_{r,+}(0)}
\big|(\nabla_y v^{(j)})(\frac{x}{\ep}) \big|^2 \dd x
& = 
\int_{(-r,r)^2} \int_{\ep\gamma(\frac{x'}{\ep})}^{\ep\gamma(\frac{x'}{\ep})+r}
\big|(\nabla_y v^{(j)})(\frac{x}{\ep}) \big|^2 \dd x_3 \dd x' \\
& = 
\ep^3 \int_{(-\frac{r}{\ep}, \frac{r}{\ep})^2} 
\int_{\gamma(y')}^{\gamma(y')+\frac{r}{\ep}}
|(\nabla_y v^{(j)})(y)|^2 \dd y_3 \dd y' \\
& \le 
C \ep^3 (\frac{r}{\ep})^2\,,
\end{split}
\end{align}
which implies the estimate \eqref{est1.lem.est.boundarylayer}. Next we take a cut-off function $\varphi=\varphi(t)\in C^\infty_0(\R)$ such that $\supp\varphi\subset(-2,1)$ and $\varphi(t)=1$ on $[-1,\frac12]$. Then we have
\begin{align}\label{est2.proof.lem.est.boundarylayer}
\begin{split}
&\int_{B^\ep_{r,+}(0)}
\big|v^{(j)}(\frac{x}{\ep}) \big|^2 \dd x \\
& =
\ep^3 \int_{(-\frac{r}{\ep}, \frac{r}{\ep})^2} 
\int_{\gamma(y')}^{\gamma(y')+\frac{r}{\ep}}
|v^{(j)}(y)|^2 \dd y_3 \dd y' \\
& \le
C \ep^3 \int_{(-\frac{r}{\ep}, \frac{r}{\ep})^2} 
\int_{\gamma(y')}^{\gamma(y')+\frac{r}{\ep}}
|v^{(j)}(y) + \gamma(y') {\bf e}_j \varphi(y_3) |^2 \dd y_3 \dd y' \\
& \quad
+ C \ep^3 \int_{(-\frac{r}{\ep}, \frac{r}{\ep})^2} 
\int_{\gamma(y')}^{\gamma(y')+\frac{r}{\ep}}
|\gamma(y') {\bf e}_j \varphi(y_3) |^2 \dd y_3 \dd y'\,.
\end{split}
\end{align}
Since $v^{(j)}(y) + \gamma(y') {\bf e}_j \varphi(y_3)$ vanishes on the boundary $y_3=\gamma(y')$, we have
\begin{align}\label{est3.proof.lem.est.boundarylayer}
\begin{split}
|v^{(j)}(y) + \gamma(y') {\bf e}_j \varphi(y_3) | 
& = \Big| \int_{\gamma(y')}^{y_3} 
\frac{\dd}{\dd t} \big[v^{(j)}(y',t) + \gamma(y') {\bf e}_j \varphi(t) \big] \dd t \Big| \\
& \le 
\int_{\gamma(y')}^{y_3} |(\partial_3 v^{(j)})(y',t)| \dd t
+ \int_{\gamma(y')}^{ 1} |\gamma(y') {\bf e}_j \frac{\dd \varphi}{\dd t}(t)| \dd t \\
& \le 
\big(y_3-\gamma(y')\big)^\frac12 
\Big( \int_{\gamma(y')}^{\infty} |\nabla v(y',t)|^2 \dd t \Big)^\frac12 
+ C\,,
\end{split}
\end{align}
where the H\"older inequality is applied in the last line and the constant $C$ depends on $\|\gamma\|_{W^{1,\infty}(\R^2)}$ 
and $\|\frac{\dd \varphi}{\dd t}\|_{L^{\infty}(\R)}$. 
Thus after inserting  \eqref{est3.proof.lem.est.boundarylayer} to \eqref{est2.proof.lem.est.boundarylayer}, by a similar computation as in \eqref{est1.proof.lem.est.boundarylayer}, we see that
\begin{align*}
\int_{B^\ep_{r,+}(0)}
\big|v^{(j)}(\frac{x}{\ep}) \big|^2 \dd x
\le 
C\Big( \ep^3 (\ep^{-1}r)^4 + \ep^3 (\ep^{-1}r)^3 \Big)\leq C\ep^{-1}{r^4}\,.
\end{align*}
Thus we obtain \eqref{est2.lem.est.boundarylayer} for $m=0$. On the other hand, by the Sobolev inequality and \eqref{est1.lem.est.boundarylayer} we find
\begin{align*}
&\bigg(\int_{B^\ep_{r,+}(0)} \big|v^{(j)}(\frac{x}{\ep}) \big|^6 \dd x \bigg)^\frac16 \\
&\le
\bigg(\int_{B^\ep_{r,+}(0)} \big| v^{(j)}(\frac{x}{\ep}) 
+ \gamma(\frac{x'}{\ep}) {\bf e}_j \varphi(\frac{x_3}{\ep}) \big|^6 \dd x \bigg)^\frac16 \\
& \quad
+ \bigg(\int_{B^\ep_{r,+}(0)} \big| \gamma(\frac{x'}{\ep}) {\bf e}_j \varphi(\frac{x_3}{\ep}) \big|^6 \dd x \bigg)^\frac16 \\
& \le
C\bigg(\int_{B^\ep_{r,+}(0)} 
\big|\nabla_x \big( v^{(j)}(\frac{x}{\ep}) 
+ \gamma(\frac{x'}{\ep}) {\bf e}_j \varphi(\frac{x_3}{\ep}) \big) \big|^2 \dd x \bigg)^\frac12
+ C (\ep r^2)^\frac16 \\
& \leq C(\ep^{-\frac12}r+\ep^\frac16r^\frac13)\le
C \ep^{-\frac12} r^\frac13\,.
\end{align*}
Then, by the H\"older inequality and \eqref{est2.lem.est.boundarylayer} with $m=0$, we obtain \eqref{est2.lem.est.boundarylayer} for $m=1$ from
\begin{align*}
\int_{B^\ep_{r,+}(0)}
\big|v^{(j)}(\frac{x}{\ep}) \big|^3 \dd x
& \le
\bigg(\int_{B^\ep_{r,+}(0)} \big|v^{(j)}(\frac{x}{\ep}) \big|^6 \dd x \bigg)^\frac14
\bigg(\int_{B^\ep_{r,+}(0)} \big|v^{(j)}(\frac{x}{\ep}) \big|^2 \dd x \bigg)^\frac34 \\
& \le
C (\ep^{-\frac12}r^\frac13)^\frac32
(\ep^{-1}r^4)^\frac34\,.
\end{align*}
We can prove \eqref{est2.lem.est.boundarylayer} for $m=2$ in a similar way by using \eqref{est2.lem.est.boundarylayer} with $m=1$. The proof is complete.
\end{proof}
%
\section{Regularity for the Stokes equations}\label{sec.linear}
%
\noindent In this subsection we consider the Stokes equations
\begin{equation}\tag{S$^\ep$}\label{S.ep}
\left\{
\begin{array}{ll}
-\Delta u^\ep+\nabla p^\ep=0&\mbox{in}\ B^\ep_{1,+}(0)\\
\nabla\cdot u^\ep=0&\mbox{in}\ B^\ep_{1,+}(0)\\
u^\ep=0&\mbox{on}\ \Gamma^\ep_1(0)
\end{array}
\right.
\end{equation}
in order to demonstrate how the compactness and iteration arguments work in a simpler setting. We note that a weak formulation for \eqref{S.ep} can be defined in a similar manner as \eqref{intro.def2} for \eqref{intro.NS.ep} in the introduction. Our goal in this section is to prove the following linear version of Theorem \ref{theo.lip.nonlinear}.
%
\begin{theorem}[linear estimate]\label{theo.lip}
There exists a constant $\ep^{(3)} \in(0,1)$ depending on $\|\gamma\|_{W^{1,\infty}(\R^2)}$ such that the following statement holds. For all $\ep\in(0,\ep^{(3)}]$ and $r\in[\ep/\ep^{(3)}, 1]$, any weak solution $u^\ep=(u^\ep_1(x), u^\ep_2(x), u^\ep_3(x))^\top \in H^1(B^\ep_{1,+}(0))^3$ to \eqref{S.ep} satisfies
\begin{align}\label{est.theo.lip}
\bigg(\dashint_{B^\ep_{r,+}(0)} |u^\ep|^2 \bigg)^\frac12 
\le
C^{(3)} r \bigg(\dashint_{B^\ep_{1,+}(0)} |u^\ep|^2 \bigg)^\frac12\,,
\end{align}
where the constant $C^{(3)}$ is independent of $\ep$ and $r$, and depends on $\|\gamma\|_{W^{1,\infty}(\R^2)}$.
\end{theorem}
%
\noindent 
We prove the compactness and iteration lemmas in Subsection \ref{subsec.cpt.itr} which are essential tools for our argument. We prove Theorem \ref{theo.lip} in Subsection \ref{subsec.linear.proof} using the estimates in Subsections \ref{subsec.per.BL} and \ref{subsec.useful.BL}.
%
\subsection{Compactness and iteration lemmas}\label{subsec.cpt.itr}
%
The compactness lemma is stated as follows. An important ingredient in the proof is that one does not need a higher order expansion for the component $u^\ep_3(x)$ thanks to the divergence-free and the no-slip conditions in the $\ep$-zero limit equations. Along the way, we overcome the difficulty coming from the vectoriality of \eqref{S.ep}. Let $v^{(j)}=v^{(j)}(y)$ be the weak solution to \eqref{BLj} for $j\in\{1,2\}$ provided by Proposition \ref{prop.BL}.
%
\begin{lemma}\label{lem.cpt}
For $\mu\in(0,1)$, there exist constants $\theta\in(0,\frac18)$ and $\ep_\mu\in(0,1)$ depending on $\|\gamma\|_{W^{1,\infty}(\R^2)}$ and $\mu$ such that the following statement holds. For $\ep\in(0,\ep_\mu]$, any weak solution $u^\ep=(u^\ep_1(x), u^\ep_2(x), u^\ep_3(x))^\top \in H^1(B^\ep_{1,+}(0))^3$ to \eqref{S.ep} with
\begin{align}\label{est1.lem.cpt}
\dashint_{B^\ep_{1,+}(0)} |u^\ep|^2 
\le
1
\end{align}
satisfies
\begin{align}\label{est2.lem.cpt}
\dashint_{B^\ep_{\theta,+}(0)}
\big|u^\ep(x)
- \sum_{j=1}^{2}
(\overline{\partial_{3} u^\ep_j})_{B^\ep_{\theta,+}(0)} 
\big( x_3 {\bf e}_j + \ep v^{(j)}(\frac{x}{\ep})\big) \big|^2 \dd x
\le
\theta^{2+2\mu}\,.
\end{align}
\end{lemma}
%
\begin{proof}
For given $\mu\in(0,1)$, we choose $\theta\in(0,\frac18)$ in the statement as follows. Let $(u^0, p^0) \in H^1(B_{\frac12,+}(0))^3\times L^2(B_{\frac12,+}(0))$ be a weak solution to the $\ep$-zero limit equations
\begin{equation}\label{eq1.proof.lem.cpt}
\left\{
\begin{array}{ll}
-\Delta u^0+\nabla p^0=0&\mbox{in}\ B_{\frac12,+}(0)\\
\nabla\cdot u^0=0&\mbox{in}\ B_{\frac12,+}(0)\\
u^0=0&\mbox{on}\ \Gamma_\frac12(0)
\end{array}
\right.
\end{equation}
with
\begin{align}\label{est1.proof.lem.cpt}
\int_{B_{\frac12,+}(0)}
|u^0|^2 \le 4\,.
\end{align}
By the regularity theory to \eqref{eq1.proof.lem.cpt} in Appendix \ref{appendix.Regularity} combined with \eqref{est1.proof.lem.cpt}, we see that $u^0\in C^2(\overline{B_{\frac38,+}(0)})^3$. From the no-slip condition in \eqref{eq1.proof.lem.cpt}, we calculate the tangential component $u^0_j$ of $u$ with $j\in\{1,2\}$ as
\begin{align*}
& u^0_j(x) - (\overline{\partial_3 u^0_j})_{B_{\theta,+}(0)}\,x_3 \\
& = \frac{x_3}{|B_{\theta,+}(0)|}
\int_0^1 \int_{B_{\theta,+}(0)} 
\big(\partial_3 u^0_j(x', t x_3) - \partial_3 u^0_j(z) \big) \dd z \dd t\,,
\end{align*}
where $\theta\in(0,\frac14)$ is arbitrary. Thus we see that
\begin{align}\label{est2.proof.lem.cpt}
\dashint_{B_{\theta,+}(0)}
|u^0_j(x) - (\overline{\partial_3 u^0_j})_{B_{\theta,+}(0)}\,x_3|^2 \dd x
\le 
C \theta^4
\end{align}
with a constant $C$ independent of $\theta$. For the normal component $u^0_3$ of $u$, by the divergence-free and no-slip conditions in \eqref{eq1.proof.lem.cpt}, we have
\begin{align*}
u^0_3(x) 
= -x_3 \int_0^1  
\sum_{j=1}^{2} \partial_j u^0_j(x', t x_3) \dd t\,.
\end{align*}
Since $\partial_j u^0_j=0$ on $\Gamma_\frac12(0)$ holds for $j\in\{1,2\}$, we also have
\begin{align*}
u^0_3(x) 
& = -x_3^2 \int_0^1 \int_0^1
\sum_{j=1}^{2} t\,\partial_3 \partial_j u^0_j(x', s t x_3) \dd s \dd t\,.
\end{align*}
Thus there exists a constant $C$ independent of $\theta$ such that for any $\theta\in(0,\frac14)$, 
\begin{align}\label{est3.proof.lem.cpt}
\dashint_{B_{\theta,+}(0)}
|u^0_3|^2
\le 
C \theta^4\,.
\end{align}
Then we choose $\theta\in(0,\frac18)$ in \eqref{est2.proof.lem.cpt} and \eqref{est3.proof.lem.cpt} sufficiently small depending on $\mu$ so that 
\begin{align}\label{est4.proof.lem.cpt}
\begin{split}
& \dashint_{B_{\theta,+}(0)}
\big|u^0(x)
- \sum_{j=1}^{2}
(\overline{\partial_{3} u^0_j})_{B_{\theta,+}(0)} x_3 {\bf e}_j \big|^2 \dd x \\
& =
\,\dashint_{B_{\theta,+}(0)}
|u^0_j(x) - (\overline{\partial_3 u^0_j})_{B_{\theta,+}(0)}\,x_3 |^2 \dd x
\,+\, \dashint_{B_{\theta,+}(0)}
|u^0_3|^2 
< \frac{\theta^{2+2\mu}}{8}\,.
\end{split}
\end{align}
The rest of the proof is by contradiction. Assume that there exist sequences $\{\ep_k\}_{k=1}^{\infty}$ in $(0,1)$ with $\lim_{k\to\infty}\ep_k=0$ and $\{u^{\ep_k}\}_{k=1}^{\infty}$ in $H^1(B^{\ep_k}_{1,+}(0))^3$ with
\begin{align}\label{est5.proof.lem.cpt}
\dashint_{B^{\ep_k}_{1,+}(0)} |u^{\ep_k}|^2 
\le 1
\end{align}
satisfying both
\begin{equation*}
\left\{
\begin{array}{ll}
-\Delta u^{\ep_k}+\nabla p^{\ep_k}=0&\mbox{in}\ B^{\ep_k}_{1,+}(0) \\
\nabla\cdot u^{\ep_k}=0&\mbox{in}\ B^{\ep_k}_{1,+}(0) \\
u^{\ep_k}=0&\mbox{on}\ \Gamma^{\ep_k}_{1}(0)
\end{array}
\right.
\end{equation*}
and
\begin{align}\label{est5'.proof.lem.cpt}
\dashint_{B^{\ep_k}_{\theta,+}(0)}
\big|u^{\ep_k}(x) - \sum_{j=1}^{2} (\overline{\partial_{3} u^{\ep_k}_j})_{B^{\ep_k}_{\theta,+}(0)} 
\big( x_3 {\bf e}_j + {\ep_k} v^{(j)}(\frac{x}{\ep_k})\big) \big|^2 \dd x
> \theta^{2+2\mu}\,.
\end{align}
Since $\ep_k\gamma(x'/\ep_k)\to 0$ uniformly in $x'\in\R^2$, the boundary $\Gamma^{\ep_k}_{1}(0)$ is included in the set $(-1,1)^2\times(-\frac12,0)$ when $k$ is sufficiently large. We extend $u^{\ep_k}$ by zero below the boundary, which is denoted again by $u^{\ep_k}$, and we see that $u^{\ep_k}\in H^1(B_1(0))^3$ for all $k\in\N$. Then, by the Caccioppoli inequality in Lemma \ref{appendix.lem.Caccioppoli.ineq.} with $\rho=\frac12$ and $r=1$ in Appendix \ref{appendix.Caccioppoli}, we have from \eqref{est5.proof.lem.cpt},
\begin{align*}
\dashint_{B^{\ep_k}_{\frac12,+}(0)} |\nabla u^{\ep_k}|^2 
\le C
\end{align*}
with $C$ independent of $\ep_k$. Hence, up to a subsequence of $\{u^{\ep_k}\}_{k=1}^{\infty}$, which is denoted by $\{u^{\ep_k}\}_{k=1}^{\infty}$ again, there exists $u^0\in H^1(B_{\frac12}(0))^3$ such that in the limit $k\to\infty$,
\begin{align*}
u^{\ep_k} \ \to \ u^0 \ \ {\rm in} \ \ L^2(B_{\frac12}(0))^3\,, \qquad
\nabla u^{\ep_k} \ \rightharpoonup \ \nabla u^0 \ \ {\rm in} \ \ L^2(B_{\frac12}(0))^{3\times3}\,,
\end{align*}
and \eqref{est1.proof.lem.cpt} holds by the assumption \eqref{est5.proof.lem.cpt}. Moreover, we have for any $\varphi\in C^\infty_0((-\frac12,\frac12)^2\times(-\frac12,0))^{3}$,  
\begin{align*}
\int_{(-\frac12,\frac12)^2\times(-\frac12,0)} u^0 \cdot \varphi
&= \lim_{k\to\infty} 
\int_{(-\frac12,\frac12)^2\times(-\frac12,0)} u^{\ep_k} \cdot \varphi \\
&=0
\end{align*}
and for any $\varphi\in C^\infty_{0,\sigma}(B_{\frac12}(0))^{3}$,
\begin{align*}
\int_{B_{\frac12,+}(0)} \nabla u^0 \cdot \nabla \varphi
&= \lim_{k\to\infty} 
\int_{B^{\ep_k}_{1,+}(0)} \nabla u^{\ep_k} \cdot \nabla \varphi \\
&=0\,.
\end{align*}
We see that $u^0=0$ on $(-\frac12,\frac12)^2\times(-\frac12,0)$ and hence that $u^0=0$ on $\Gamma_\frac12(0)$ from $u^0\in H^1(B_{\frac12}(0))$. Thus $u^0$ is a weak solution to \eqref{eq1.proof.lem.cpt} satisfying \eqref{est1.proof.lem.cpt}. Then, from $B^{\ep_k}_{\theta,+}(0)=\big(B^{\ep_k}_{\theta,+}(0)\setminus B_{\theta,+}(0)\big) \cup \big(B^{\ep_k}_{\theta,+}(0)\cap B_{\theta,+}(0)\big)$ and $|B^{\ep_k}_{\theta,+}(0)|=|B_{\theta,+}(0)|=4\theta^3$, by the triangle inequality we have
\begin{align*}
& \dashint_{B^{\ep_k}_{\theta,+}(0)}
\big|u^{\ep_k}(x) - \sum_{j=1}^{2} (\overline{\partial_{3} u^{\ep_k}_j})_{B^{\ep_k}_{\theta,+}(0)} 
\big( x_3 {\bf e}_j + {\ep_k} v^{(j)}(\frac{x}{\ep_k})\big) \big|^2 \dd x \\
& \le
\frac{1}{4\theta^3}
\int_{B^{\ep_k}_{\theta,+}(0)\setminus B_{\theta,+}(0)}
\big|u^{\ep_k}(x) - \sum_{j=1}^{2} (\overline{\partial_{3} u^{\ep_k}_j})_{B^{\ep_k}_{\theta,+}(0)} 
\big( x_3 {\bf e}_j + {\ep_k} v^{(j)}(\frac{x}{\ep_k})\big) \big|^2 \dd x \\
& \quad
+ \frac{2}{\theta^3}
\bigg(
\int_{B^{\ep_k}_{\theta,+}(0)\cap B_{\theta,+}(0)}
\big|u^{\ep_k} - u^0 \big|^2 \\
& \qquad\qquad\qquad
+ \sum_{j=1}^{2} \big|(\overline{\partial_{3} u^{\ep_k}_j})_{B^{\ep_k}_{\theta,+}(0)} x_3{\bf e}_j
- (\overline{\partial_3 u^0_j})_{B_{\theta,+}(0)} x_3{\bf e}_j \big|^2 \dd x \\
& \qquad\qquad\qquad
+ \ep_k^2 \sum_{j=1}^{2} 
\big|(\overline{\partial_{3} u^{\ep_k}_j})_{B^{\ep_k}_{\theta,+}(0)}\big|^2
\int_{B^{\ep_k}_{\theta,+}(0)\cap B_{\theta,+}(0)}
\big|v^{(j)}(\frac{x}{\ep_k}) \big|^2 \dd x
\bigg) \\
& \quad 
+ 8\,\dashint_{B_{\theta,+}(0)}
\big|u^0(x) - \sum_{j=1}^{2} (\overline{\partial_{3} u^0_j})_{B_{\theta,+}(0)} x_3 {\bf e}_j \big|^2 \dd x\,.
\end{align*}
Since $u^{\ep_k} \to u^0$ in $L^2(B_{\frac12}(0))^3$ and $\{\nabla u^{\ep_k}\}_{k=1}^{\infty}$ is uniformly bounded in $L^2(B_{\frac12}(0))^{3\times3}$, from the assumption \eqref{est5'.proof.lem.cpt} we see that
\begin{align*}
\theta^{2+2\mu}
& \le 
\varlimsup_{k\to\infty} 
\dashint_{B^{\ep_k}_{\theta,+}(0)}
\big|u^{\ep_k}(x) - \sum_{j=1}^{2} (\overline{\partial_{3} u^{\ep_k}_j})_{B^{\ep_k}_{\theta,+}(0)} 
\big( x_3 {\bf e}_j + {\ep_k} v^{(j)}(\frac{x}{\ep_k})\big) \big|^2 \dd x \\
& \le 
8\,\dashint_{B_{\theta,+}(0)}
\big|u^0(x) - \sum_{j=1}^{2} (\overline{\partial_{3} u^0_j})_{B_{\theta,+}(0)} x_3 {\bf e}_j \big|^2 \dd x\,,
\end{align*}
where \eqref{est2.lem.est.boundarylayer} with $m=0$ in Lemma \ref{lem.est.boundarylayer} is applied to obtain the second line. Hence the choice of $\theta$ in \eqref{est4.proof.lem.cpt} contradicts  \eqref{est5'.proof.lem.cpt}. This completes the proof of Lemma \ref{lem.cpt}.
\end{proof}
%
\noindent The iteration lemma to \eqref{S.ep} is stated as follows. Let $K_0$ be the constant of the Caccioppoli inequality in Lemma \ref{appendix.lem.Caccioppoli.ineq.} in Appendix \ref{appendix.Caccioppoli}.
%
\begin{lemma}\label{lem.itr}
Fix $\mu\in(0,1)$ and let $\theta\in(0,\frac18)$ and $\ep_{\mu}\in(0,1)$ be the constants in Lemma \ref{lem.cpt}. Then for $k\in\N$ and $\ep\in(0, \theta^{k-1}\ep_\mu]$, any weak solution $u^\ep=(u^\ep_1(x), u^\ep_2(x), u^\ep_3(x))^\top \in H^1(B^\ep_{1,+}(0))^3$ to \eqref{S.ep} with
\begin{align}\label{est1.lem.itr}
\dashint_{B^\ep_{1,+}(0)} |u^\ep|^2 
\le
1
\end{align}
satisfies
\begin{align}\label{est2.lem.itr}
\dashint_{B^\ep_{\theta^k,+}(0)}
\big|u^\ep(x) 
- \sum_{j=1}^{2}
a^\ep_{k,j}
\big( x_3 {\bf e}_j + \ep v^{(j)}(\frac{x}{\ep})\big) \big|^2 \dd x
\le
\theta^{(2+2\mu)k}\,.
\end{align}
Here the number $a^\ep_{k,j}\in\R$, $j\in\{1,2\}$, is estimated as
\begin{align}\label{est3.lem.itr}
\sum_{j=1}^{2} |a^\ep_{k,j}|
\le 
2K_0^\frac12 \theta^{-\frac32} (1-\theta)^{-1}
\sum_{l=1}^{k} \theta^{\mu(l-1)}\,.
\end{align}
\end{lemma}
%
\begin{proof}
The proof is done by induction on $k\in\N$. The case $k=1$ is valid since it is exactly \eqref{est2.lem.cpt} in Lemma \ref{lem.cpt} putting $a^\ep_{1,j}=(\overline{\partial_{3} u^\ep_j})_{B^\ep_{\theta,+}(0)}$, $j\in\{1,2\}$. Indeed, by the H\"older inequality we have
\begin{align*}
\sum_{j=1}^{2} |a^\ep_{1,j}|
& \le 
2|B^\ep_{\theta,+}(0)|^{-\frac12}
\|\nabla u^\ep\|_{L^2(B^\ep_{\theta,+}(0))} \\
& \le 
K_0^\frac12 \theta^{-\frac32} (1-\theta)^{-1}
\|u^\ep\|_{L^2(B^\ep_{1,+}(0))}\,,
\end{align*}
where we have applied the Caccioppoli inequality to \eqref{S.ep} with $\rho=\theta$ and $r=1$ in Lemma \ref{appendix.lem.Caccioppoli.ineq.} in Appendix \ref{appendix.Caccioppoli}. Thus by \eqref{est1.lem.itr} we have \eqref{est3.lem.itr} for $k=1$. Next let us assume that \eqref{est2.lem.itr} and \eqref{est3.lem.itr} hold at rank $k\in\N$ and let $\ep\in(0, \theta^{k}\ep_\mu]$. Then we define new functions $U^{\ep/\theta^k}=(U^{\ep/\theta^k}_1(y), U^{\ep/\theta^k}_2(y), U^{\ep/\theta^k}_3(y))^\top$ and $P^{\ep/\theta^k}=P^{\ep/\theta^k}(y)$ on $B^{\ep/\theta^k}_{1,+}(0)$ by
\begin{align*}
U^{\ep/\theta^k}(y)
& = \frac{1}{\theta^{(1+\mu)k}}
\Big( u^\ep(\theta^k y) 
- \sum_{j=1}^{2}
\theta^k a^\ep_{k,j} \big(y_3 {\bf e}_j 
+ \frac{\ep}{\theta^k} v^{(j)}(\frac{\theta^k y}{\ep})\big) \Big)\,, \\
P^{\ep/\theta^k}(y) 
& = \frac{1}{\theta^{\mu k}} \Big( p^\ep(\theta^k y) 
- \sum_{j=1}^{2} a^\ep_{k,j} q^{(j)}(\frac{\theta^k y}{\ep}) \Big)\,.
\end{align*}
We see that $(U^{\ep/\theta^k}, P^{\ep/\theta^k})$ is a weak solution to
\begin{equation}\label{eq1.proof.lem.itr}
\left\{
\begin{array}{ll}
-\Delta_y U^{\ep/\theta^k} +\nabla_y P^{\ep/\theta^k}=0&\mbox{in}\ B^{\ep/\theta^k}_{1,+}(0) \\
\nabla_y\cdot U^{\ep/\theta^k}=0&\mbox{in}\ B^{\ep/\theta^k}_{1,+}(0) \\
U^{\ep/\theta^k}=0&\mbox{on}\ \Gamma^{\ep/\theta^k}_{1}(0)\,.
\end{array}
\right.
\end{equation}
From the recurrence hypothesis \eqref{est2.lem.itr} at rank $k$, we have
\begin{align}\label{est1.proof.lem.itr}
\dashint_{B^{\ep/\theta^k}_{1,+}(0)} |U^{\ep/\theta^k}|^2 \le 1
\end{align}
by a change of variables. Now, since $\ep/\theta^k \in(0, \ep_\mu]$, we can apply Lemma \ref{lem.cpt} to see that
\begin{align*}
\dashint_{B^{\ep/\theta^k}_{\theta,+}(0)}
\big|U^{\ep/\theta^k}(y) 
- \sum_{j=1}^{2}
(\overline{\partial_{y_3} U^{\ep/\theta^k}_j})_{B^{\ep/\theta^k}_{\theta,+}(0)} 
\big( y_3 {\bf e}_j + \frac{\ep}{\theta^k} v^{(j)}(\frac{\theta^k y}{\ep})\big) \big|^2 \dd y
\le
\theta^{2+2\mu}\,.
\end{align*}
A change of variables leads to
\begin{align}\label{goal.proof.lem.itr}
\dashint_{B^\ep_{\theta^{k+1},+}(0)}
\big|u^\ep(x) 
- \sum_{j=1}^{2}
a^\ep_{k+1,j} \big( x_3 {\bf e}_j + \ep v^{(j)}(\frac{x}{\ep})\big) \big|^2 \dd x
\le
\theta^{(2+2\mu)(k+1)}\,,
\end{align}
where the number $a^\ep_{k+1,j}\in\R$, $j\in\{1,2\}$, is defined as
\begin{align}\label{def1.proof.lem.itr}
a^\ep_{k+1,j} = a^\ep_{k,j} 
+  \theta^{\mu k} (\overline{\partial_{y_3} U^{\ep/\theta^k}_j})_{B^{\ep/\theta^k}_{\theta,+}(0)}\,.
\end{align}
The Caccioppoli inequality to \eqref{eq1.proof.lem.itr} with $\rho=\theta$ and $r=1$ combined with \eqref{est1.proof.lem.itr} 
leads to
\begin{align*}
\begin{split}
\|\nabla_y U^{\ep/\theta^k}\|_{L^2(B^{\ep/\theta^k}_{\theta,+}(0))}
& \le 
K_0^\frac12 (1-\theta)^{-1} \|U^{\ep/\theta^k}\|_{L^2(B^{\ep/\theta^k}_{1,+}(0))} \\
& \le 
2K_0^\frac12 (1-\theta)^{-1}\,.
\end{split}
\end{align*}
Therefore, from the assumption \eqref{est3.lem.itr} for $k$ and \eqref{def1.proof.lem.itr}, by the H\"older inequality we obtain
\begin{align*}
\sum_{j=1}^{2} |a^\ep_{k+1,j}| 
& \le
\sum_{j=1}^{2} |a^\ep_{k,j}| 
+\theta^{\mu k} \sum_{j=1}^{2} 
\big|(\overline{\partial_{y_3} U^{\ep/\theta^k}_j})_{B^{\ep/\theta^k}_{\theta,+}(0)}\big| \\
& \le 
2K_0^\frac12 \theta^{-\frac32} (1-\theta)^{-1}
\sum_{l=1}^{k+1} \theta^{\mu(l-1)}\,,
\end{align*}
which with \eqref{goal.proof.lem.itr} proves the assertions \eqref{est2.lem.itr} and \eqref{est3.lem.itr} for $k+1$. This completes the proof.
\end{proof}
%
\subsection{Proof of Theorem \ref{theo.lip}}\label{subsec.linear.proof}
%
We prove Theorem \ref{theo.lip} by applying Lemma \ref{lem.itr}. Fix $\mu\in(0,1)$ and let $\theta\in(0,\frac18)$ and $\ep_{\mu}\in(0,1)$ be the constants in Lemma \ref{lem.cpt}.
%
\begin{proofx}{Theorem \ref{theo.lip}} Since the equations \eqref{S.ep} are linear, it suffices to prove the estimate
\begin{align}\label{est1.proof.theo.lip}
\bigg(\dashint_{B^\ep_{r,+}(0)} |u^\ep|^2 \bigg)^\frac12 
& \le
C r\,.
\end{align}
Set $\ep^{(3)}=\ep_{\mu}$ and let $\ep\in(0,\ep^{(3)}]$. Firstly we note that if $r\in(\theta, 1]$, then
\begin{align*}
\bigg(\dashint_{B^\ep_{r,+}(0)} |u^\ep|^2 \bigg)^\frac12 
\le
\theta^{-\frac52} r
\end{align*}
holds. Thus we focus on the case $r\in[\ep/\ep^{(3)}, \theta]$. For any given $r\in[\ep/\ep^{(3)},\theta]$, there exists $k\in\N$ 
 with $k\ge2$ 
such that $r\in(\theta^k, \theta^{k-1}]$ holds. From $\ep\in(0,\theta^{k-1}\ep^{(3)}]$ we apply Lemma \ref{lem.itr} to see that
\begin{align}\label{est2.proof.theo.lip}
\begin{split}
& \bigg(\dashint_{B^\ep_{r,+}(0)} |u^\ep|^2 \bigg)^\frac12 
\le
\bigg(\theta^{-3} \dashint_{B^\ep_{\theta^{k-1},+}(0)} |u^\ep|^2 \bigg)^\frac12 \\
& \le
\theta^{-\frac32}
\bigg(
\dashint_{B^\ep_{\theta^{k-1},+}(0)}
\big|u^\ep(x) 
- \sum_{j=1}^{2}
a^\ep_{k-1,j}
\big( x_3 {\bf e}_j + \ep v^{(j)}(\frac{x}{\ep})\big) \big|^2 \dd x
\bigg)^\frac12 \\
& \quad
+ \theta^{-\frac32}
\Big(\sum_{j=1}^{2} |a^\ep_{k-1,j}|\Big)
\bigg(\sum_{j=1}^{2} \dashint_{B^\ep_{\theta^{k-1},+}(0)}
\big| x_3 {\bf e}_j + \ep v^{(j)}(\frac{x}{\ep})\big) \big|^2 \dd x \bigg)^\frac12 \\
& \le
\theta^{(1+\mu)(k-1)-\frac32} \\
& \quad
+ C \theta^{-3} 
(1-\theta)^{-1} (1-\theta^\mu)^{-1}
\bigg(\sum_{j=1}^{2} \dashint_{B^\ep_{\theta^{k-1},+}(0)}
\big| x_3 {\bf e}_j + \ep v^{(j)}(\frac{x}{\ep})\big) \big|^2 \dd x \bigg)^\frac12\,,
\end{split}
\end{align}
where $C$ depends only on $\|\gamma\|_{W^{1,\infty}(\R^2)}$. From \eqref{est2.lem.est.boundarylayer} with $m=0$ in Lemma \ref{lem.est.boundarylayer} one has
\begin{align*}
\bigg(\sum_{j=1}^{2} \dashint_{B^\ep_{\theta^{k-1},+}(0)}
\big| x_3 {\bf e}_j + \ep v^{(j)}(\frac{x}{\ep})\big) \big|^2 \dd x \bigg)^\frac12
&\le 
C (\theta^{k-1} + \ep^\frac12 \theta^{\frac{k-1}{2}})\,.
\end{align*}
Therefore, by $\theta^{k-1}\in(0,\theta^{-1} r)$ and $\ep\in(0,\theta^{k-1} \ep^{(3)}]$, we have from \eqref{est2.proof.theo.lip}, 
\begin{align*}
\begin{split}
\bigg(\dashint_{B^\ep_{r,+}(0)} |u^\ep|^2 \bigg)^\frac12 
& \le
\theta^{-\frac52-\mu} r^{1+\mu}
+ C \theta^{-3} 
(1-\theta)^{-1} (1-\theta^\mu)^{-1}
(\theta^{k-1} + \ep^\frac12 \theta^{\frac{k-1}{2}}) \\
& \le
\Big( \theta^{-\frac52-\mu} r^{\mu} + C\theta^{-4} 
(1-\theta)^{-1} (1-\theta^\mu)^{-1}
(1 + (\ep^{(3)})^\frac12) \Big) r\,.
\end{split}
\end{align*}
Hence we obtain the desired estimate \eqref{est1.proof.theo.lip} by letting $\mu=\frac12$ for instance.
This completes the proof of Theorem \ref{theo.lip}.
\end{proofx}
%
\section{Proof of the main results}\label{sec.nonlinear}
%
\noindent We prove Theorem \ref{theo.lip.nonlinear} and Theorem \ref{theo.lip.nonlinear.periodic} in this section. As is done in Section \ref{sec.linear}, we first work out the compactness and iteration lemmas in Subsection \ref{subsec.nonlinear.cpt.itr}. Contrary to the linear case, we need to carry out a careful analysis of the iteration argument due to the nonlinearity. Indeed, since we do not assume any smallness condition on solutions of \eqref{intro.NS.ep}, a naive iterated application of the Caccioppoli inequality leads to a blow-up of the derivative estimate in the nonlinear case. We overcome this difficulty a priori by taking the free parameter $\theta$ appearing in the compactness lemma sufficiently small depending on the bound $M$ of the solution to \eqref{intro.NS.ep}. Eventually, the proof of Theorem \ref{theo.lip.nonlinear} and Theorem \ref{theo.lip.nonlinear.periodic} is given in Subsection \ref{subsec.nonlinear.proof}.
%
\subsection{Nonlinear compactness and iteration lemmas}\label{subsec.nonlinear.cpt.itr}
%
We give the proof of the compactness and iteration lemmas to the nonlinear equations. We consider the modified Navier-Stokes equations:
\begin{equation}\tag{MNS$^\ep$}\label{modified.NS.ep}
\left\{
\begin{aligned}
&-\Delta U^{\ep} +\nabla P^{\ep}
= -\nabla \cdot 
(U^{\ep} \otimes b^{\ep} + b^\ep \otimes U^{\ep}) \\
& \qquad\qquad\qquad\qquad
-\lambda^\ep U^{\ep} \cdot\nabla U^{\ep}
+ \nabla \cdot F^\ep \ \ \mbox{in}\ B^{\ep}_{1,+}(0) \\
&\nabla\cdot U^{\ep}=0 \ \ \mbox{in}\ B^{\ep}_{1,+}(0) \\
&U^{\ep}=0 \ \ \mbox{on}\ \Gamma^{\ep}_1(0)\,,
\end{aligned}\right.
\end{equation}
where $b^\ep=b^\ep(x)$ is defined as
\begin{align}
b^\ep(x) = \sum_{j=1}^{2} C^\ep_j 
\big(x_3 {\bf e}_j + \ep v^{(j)}(\frac{x}{\ep})\big)\,, 
\quad x\in B^\ep_{1,+}(0)\,.
\end{align}
Note that $\nabla\cdot b^\ep=0$ in $B^{\ep}_{1,+}(0)$ and $b^\ep=0$ on $\Gamma^{\ep}_1(0)$. The compactness lemma is stated as follows.
%
\begin{lemma}\label{lem.nonlinear.cpt}
For $M\in(0,\infty)$ and $\mu\in(0,1)$, there exists a constant $\theta_{0}\in(0,\frac18)$ depending on $M$ and $\mu$ such that the following statement holds. For any $\theta\in(0,\theta_{0}]$, there exists $\ep_\mu\in(0,1)$ depending on $\|\gamma\|_{W^{1,\infty}(\R^2)}$, $M$, $\mu$, and $\theta$ such that for $\ep\in(0, \ep_\mu]$, $(\lambda^\ep, C^\ep_1, C^\ep_2)\in[-1,1]^3$, and $F^{\ep}\in L^2(B^\ep_{1,+}(0))^{3\times3}$ with
\begin{align}\label{est1.lem.nonlinear.cpt}
\|F^{\ep}\|_{L^2(B^\ep_{1,+}(0))}
\le M \ep_{\mu}\,,
\end{align}
any weak solution $U^\ep=(U^\ep_1(x), U^\ep_2(x), U^\ep_3(x))^\top \in H^1(B^\ep_{1,+}(0))^3$ to \eqref{modified.NS.ep} with
\begin{align}\label{est2.lem.nonlinear.cpt}
\dashint_{B^\ep_{1,+}(0)} |U^\ep|^2 
\le M^2
\end{align}
satisfies
\begin{align}\label{est3.lem.nonlinear.cpt}
\dashint_{B^\ep_{\theta,+}(0)}
\big|U^\ep(x)
- \sum_{j=1}^{2}
(\overline{\partial_{3} U^\ep_j})_{B^\ep_{\theta,+}(0)} 
\big( x_3 {\bf e}_j + \ep v^{(j)}(\frac{x}{\ep})\big) \big|^2 \dd x
\le
M^2 \theta^{2+2\mu}\,.
\end{align}
\end{lemma}
%
\begin{proof}
By setting
\begin{align*}
V^\ep = \frac{U^\ep}{M}\,, \qquad
Q^\ep = \frac{P^\ep}{M}\,, \qquad
G^\ep = \frac{F^\ep}{M}\,,
\end{align*}
we see that $V^\ep$ and $G^\ep$ satisfy 
\begin{align*}
\dashint_{B^\ep_{1,+}(0)} |V^\ep|^2 \le 1\,, \qquad
\|G^\ep\|_{L^2(B^\ep_{1,+}(0))} \le \ep_\mu\,,
\end{align*}
and that $(V^\ep, Q^\ep)$ solves the equations
\begin{equation}\label{eq0.proof.lem.nonlinear.cpt}
\left\{
\begin{aligned}
&-\Delta V^\ep +\nabla Q^\ep
= -\nabla \cdot 
(V^\ep \otimes b^\ep + b^\ep \otimes V^\ep) \\
& \qquad\qquad\qquad\qquad
- M \lambda^\ep V^\ep \cdot\nabla V^\ep
+ \nabla \cdot G^\ep \ \ \mbox{in}\ B^{\ep}_{1,+}(0) \\
&\nabla\cdot V^\ep=0 \ \ \mbox{in}\ B^{\ep}_{1,+}(0) \\
&V^\ep=0 \ \ \mbox{on}\ \Gamma^{\ep}_1(0)\,.
\end{aligned}\right.
\end{equation}
In the following we consider the rescaled equations \eqref{eq0.proof.lem.nonlinear.cpt}.   Hence our goal is to obtain
\begin{align}\label{goal.proof.lem.nonlinear.cpt}
\dashint_{B^\ep_{\theta,+}(0)}
\big|V^\ep(x)
- \sum_{j=1}^{2}
(\overline{\partial_{3} V^\ep_j})_{B^\ep_{\theta,+}(0)} 
\big( x_3 {\bf e}_j + \ep v^{(j)}(\frac{x}{\ep})\big) \big|^2 \dd x
\le
\theta^{2+2\mu}\,.
\end{align}
For given $M\in (0,\infty)$ and $\mu\in(0,1)$, we choose $\theta_0\in(0,\frac18)$ in the statement as follows. Let $(V^0, Q^0) \in H^1(B_{\frac12,+}(0))^3\times L^2(B_{\frac12,+}(0))$ be a weak solution to the $\ep$-zero limit equations
\begin{equation}\label{eq1.proof.lem.nonlinear.cpt}
\left\{
\begin{aligned}
&-\Delta V^0 +\nabla Q^0
= -\nabla \cdot 
\Big(V^0 \otimes \big(\sum_{j=1}^{2} C^0_j x_3 {\bf e}_j \big)
+ \big(\sum_{j=1}^{2} C^0_j x_3 {\bf e}_j \big) \otimes V^0\Big) \\
& \qquad\qquad\qquad\qquad
- M \lambda^0 V^0 \cdot\nabla V^0 \ \ \mbox{in}\ B_{\frac12,+}(0) \\
&\nabla\cdot V^0=0 \ \ \mbox{in}\ B_{\frac12,+}(0) \\
&V^0=0 \ \ \mbox{on}\ \Gamma_{\frac12}(0)
\end{aligned}\right.
\end{equation}
with
\begin{align}\label{est1.proof.lem.nonlinear.cpt}
\int_{B_{\frac12,+}(0)}
|V^0|^2 \le 4\,.
\end{align}
By the regularity theory to \eqref{eq1.proof.lem.nonlinear.cpt} in Appendix \ref{appendix.Regularity} using \eqref{est1.proof.lem.nonlinear.cpt}, we see that $V^0\in C^2(\overline{B_{\frac38,+}(0)})^3$ and 
\begin{align*}
\|V^0\|_{C^2(\overline{B_{\frac38,+}(0)})} \le K
\end{align*}
with a constant $K$ depending on $M$ but independent of $(\lambda^0,C^0_1, C^0_2)\in[-1,1]^3$. Then, in the same way as in the proof of Lemma \ref{lem.cpt}, we choose $\theta_0\in(0,\frac18)$ sufficiently small so that for any $\theta\in(0,\theta_0]$
\begin{align}\label{est2.proof.lem.nonlinear.cpt}
\dashint_{B_{\theta,+}(0)}
\big|V^0(x)
- \sum_{j=1}^{2}
(\overline{\partial_{3} V^0_j})_{B_{\theta,+}(0)} x_3 {\bf e}_j \big|^2 \dd x
< \frac{\theta^{2+2\mu}}{8}
\end{align}
holds. We emphasize that $\theta_0$ depends only on $M$ and $\mu$. The rest of the proof is done by contradiction. Assume that there exist $\theta\in(0,\theta_0]$ and sequences $\{\ep_k\}_{k=1}^{\infty}\subset(0,1)$ with $\lim_{k\to\infty}\ep_k=0$,
$\{(\lambda^{\ep_k},C^{\ep_k}_1, C^{\ep_k}_2)\}_{k=1}^{\infty}\subset[-1,1]^3$, 
and $\{G^{\ep_k}\}_{k=1}^{\infty}\subset L^2(B^{\ep_k}_{1,+}(0))^{3\times3}$ with
\begin{align*}
\|G^{\ep_k}\|_{L^2(B^{\ep_k}_{1,+}(0))} \le \ep_{k}\,.
\end{align*}
Moreover, we assume that there exists $\{V^{\ep_k}\}_{k=1}^{\infty}$ in $H^1(B^{\ep_k}_{1,+}(0))^3$ with
\begin{align}\label{est3.proof.lem.nonlinear.cpt}
\dashint_{B^{\ep_k}_{1,+}(0)} |V^{\ep_k}|^2 
\le 1
\end{align}
satisfying both
\begin{equation*}
\left\{
\begin{aligned}
&-\Delta V^{\ep_k} +\nabla Q^{\ep_k}
= -\nabla \cdot 
(V^{\ep_k} \otimes b^{\ep_k} + b^{\ep_k} \otimes V^{\ep_k}) \\
& \qquad\qquad\qquad\qquad\quad 
-M \lambda^{\ep_k} V^{\ep_k} \cdot\nabla V^{\ep_k}
+ \nabla \cdot G^{\ep_k} \ \ \mbox{in}\ B^{\ep_k}_{1,+}(0) \\
&\nabla\cdot V^{\ep_k}=0 \ \ \mbox{in}\ B^{\ep_k}_{1,+}(0) \\
&V^{\ep_k}=0 \ \ \mbox{on}\ \Gamma^{\ep_k}_1(0)
\end{aligned}\right.
\end{equation*}
and
\begin{align}\label{est4.proof.lem.nonlinear.cpt}
\dashint_{B^{\ep_k}_{\theta,+}(0)}
\big|V^{\ep_k}(x) - \sum_{j=1}^{2} (\overline{\partial_{3} V^{\ep_k}_j})_{B^{\ep_k}_{\theta,+}(0)} 
\big( x_3 {\bf e}_j + {\ep_k} v^{(j)}(\frac{x}{\ep_k})\big) \big|^2 \dd x
> \theta^{2+2\mu}\,.
\end{align}
We extend $V^{\ep_k}$, $v^{(j)}(\cdot/\ep_k)$, and $G^{\ep_k}$ by zero below the boundary, which are respectively denoted by $V^{\ep_k}$, $v^{(j)}(\cdot/\ep_k)$, and $G^{\ep_k}$ again, and see that $V^{\ep_k}\in H^1(B_1(0))^3$ and $G^{\ep_k}\in L^2(B_1(0))^{3\times3}$ for all $k\in\N$. By applying the Caccioppoli inequality in Lemma \ref{appendix.lem.Caccioppoli.ineq.} with $\rho=\frac12$ and $r=1$ in Appendix \ref{appendix.Caccioppoli}, we obtain
\begin{align*}
\|\nabla V^{\ep_k}\|_{L^2(B^{\ep_k}_{\frac12,+}(0))} \le C(1+M^3)
\end{align*}
uniformly in $k$ with a constant $C$ independent of $M$. Here we have used \eqref{est1.lem.est.boundarylayer} in Lemma \ref{lem.est.boundarylayer} and \eqref{est3.proof.lem.nonlinear.cpt}. Hence, up to subsequences of $\{V^{\ep_k}\}_{k=1}^{\infty}$, $\{(\lambda^{\ep_k}, C^{\ep_k}_1, C^{\ep_k}_2)\}_{k=1}^{\infty}$, and $\{G^{\ep_k}\}_{k=1}^{\infty}$, which are respectively denoted by $\{V^{\ep_k}\}_{k=1}^{\infty}$, $\{(\lambda^{\ep_k}, C^{\ep_k}_1, C^{\ep_k}_2)\}_{k=1}^{\infty}$, and $\{G^{\ep_k}\}_{k=1}^{\infty}$ again, 
there exist $V^0\in H^1(B_{\frac12}(0))^3$ and $(\lambda^0, C^0_1, C^0_2)\in[-1,1]^3$ such that in the limit $k\to\infty$,
\begin{align*}
& V^{\ep_k} \ \to \ V^0 \ \ {\rm in} \ \ L^2(B_{\frac12}(0))^3\,, \qquad
\nabla V^{\ep_k} \ \rightharpoonup \ \nabla V^0 \ \ {\rm in} \ \ L^2(B_{\frac12}(0))^{3\times3}\,, \\
&(\lambda^{\ep_k}, C^{\ep_k}_1, C^{\ep_k}_2) \to (\lambda^0, C^0_1, C^0_2) \ \ {\rm in} \ \ [-1,1]^3\,, \qquad
G^{\ep_k} \ \to \ 0 \ \ {\rm in} \ \ L^2(B_{\frac12}(0))^{3\times3}\,.
\end{align*}
On the other hand, the assumption \eqref{est3.proof.lem.nonlinear.cpt} implies \eqref{est1.proof.lem.nonlinear.cpt}. Hence, from \eqref{est2.lem.est.boundarylayer} with $m=0$ in Lemma \ref{lem.est.boundarylayer}, by a similar reasoning as in the proof of Lemma \ref{lem.cpt} combined with the convergences
\begin{align*}
&V^{\ep_k}\otimes V^{\ep_k} \ \to \ V^0\otimes V^0 
\ \ {\rm in} \ \ L^1(B_{\frac12}(0))^3\,, \\
&V^{\ep_k} \otimes \big(\sum_{j=1}^{2} C^{\ep_k}_j \ep_k v^{(j)}(\frac{\cdot}{\ep_k})\big)
+ \big(\sum_{j=1}^{2} C^{\ep_k}_j \ep_k v^{(j)}(\frac{\cdot}{\ep_k})\big) \otimes V^{\ep_k} \\
&\qquad\quad\quad \ \ 
\ \to \ 0 \ \ {\rm in} \ \ L^1(B_{\frac12}(0))^{3\times3}\,,
\end{align*}
we see that the limit $V^0$ gives a weak solution to \eqref{eq1.proof.lem.nonlinear.cpt} satisfying \eqref{est1.proof.lem.nonlinear.cpt}. Then, in the same way as in the proof of Lemma \ref{lem.cpt}, we reach a contradiction to \eqref{est4.proof.lem.nonlinear.cpt} from the choice of $\theta\in(0,\theta_0]$ in \eqref{est2.proof.lem.nonlinear.cpt}. Hence we obtain the desired estimate \eqref{goal.proof.lem.nonlinear.cpt} yielding \eqref{est3.lem.nonlinear.cpt}. This completes the proof.
\end{proof}
%
\noindent Next we prove the iteration lemma to the Navier-Stokes equations
\begin{equation}\tag{NS$^\ep$}\label{NS.ep}
\left\{
\begin{array}{ll}
-\Delta u^\ep+\nabla p^\ep
=-u^\ep\cdot\nabla u^\ep&\mbox{in}\ B^\ep_{1,+}(0)\\
\nabla\cdot u^\ep=0&\mbox{in}\ B^\ep_{1,+}(0)\\
u^\ep=0&\mbox{on}\ \Gamma^\ep_1(0)\,.
\end{array}
\right.
\end{equation}
An important step is the a priori choice of the parameter $\theta$ of Lemma \ref{lem.nonlinear.cpt} depending on the bound of the solution. Let $K_0$ be the constant in the Caccioppoli inequality in Lemma \ref{appendix.lem.Caccioppoli.ineq.} in Appendix \ref{appendix.Caccioppoli}.
%
\begin{lemma}\label{lem.nonlinear.itr}
Fix $M\in(0,\infty)$ and $\mu\in(0,1)$, and let $\theta_0\in(0,\frac18)$ be the constant in Lemma \ref{lem.nonlinear.cpt}. Choose $\theta=\theta(M,\mu)\in(0,\theta_0]$ sufficiently small to satisfy the conditions
\begin{align}\label{est7.proof.lem.nonlinear.itr}
\begin{split}
&4K_0^\frac12 (1-\theta^\mu)^{-1} (6 + 2^8M^4)^\frac12 M \theta^{\frac12} \le 1\,, \\
&\big(C_1(1-\theta^\mu)^{-1} (6 + 2^8M^4)^\frac12 M \theta^{\frac12}\big)^4 \\
& + (1-\theta)^{-\frac43} \big(C_1 (1-\theta^\mu)^{-1} (6 + 2^8M^4)^\frac12 M \theta^{\frac12}\big)^\frac43 \\
& \le \frac{(1-\theta)^{-2}}{4}\,, \\
&{\rm and} \quad C_2 (1-\theta^\mu)^{-2} (6 + 2^8M^4) M \theta \le 1\,,
\end{split}
\end{align}
where $C_1$ and $C_2$ are numerical constants appearing respectively in \eqref{est5.proof.lem.nonlinear.itr} and \eqref{est6.proof.lem.nonlinear.itr} in the proof. Moreover, let $\ep_\mu\in(0,1)$ be the corresponding constant for $\theta$ in Lemma \ref{lem.nonlinear.cpt}. Then for $k\in\N$ and 
\begin{align}\label{est1.lem.nonlinear.itr}
\ep\in(0, \theta^{k+2(2+\mu)(1-\delta_{1k})-1} \ep_{\mu}^{2-\delta_{1k}}]\,,
\end{align}
where $\delta_{1k}$ is the Kronecker delta, any weak solution $u^\ep=(u^\ep_1(x), u^\ep_2(x), u^\ep_3(x))^\top \in H^1(B^\ep_{1,+}(0))^3$ to \eqref{NS.ep} with
\begin{align}\label{est2.lem.nonlinear.itr}
\dashint_{B^\ep_{1,+}(0)} |u^\ep|^2 
\le M^2 
\end{align}
satisfies
\begin{align}\label{est3.lem.nonlinear.itr}
\dashint_{B^\ep_{\theta^k,+}(0)}
\big|u^\ep(x) 
- \sum_{j=1}^{2}
a^\ep_{k,j}
\big( x_3 {\bf e}_j + \ep v^{(j)}(\frac{x}{\ep})\big) \big|^2 \dd x
\le
M^2 \theta^{(2+2\mu)k}\,.
\end{align}
Here the number $a^\ep_{k,j}\in\R$, $j\in\{1,2\}$, is estimated as
\begin{align}\label{est4.lem.nonlinear.itr}
\sum_{j=1}^{2} |a^\ep_{k,j}| 
\le 
K_0^\frac12 \theta^{-\frac32} (1-\theta)^{-1} 
\big(6 + 2^6 (1-\theta)^{-2} M^4\big)^\frac12 M
\sum_{l=1}^{k} \theta^{\mu(l-1)}\,.
\end{align}
\end{lemma}
%
\begin{proof} 
The proof is done by induction on $k\in\N$. For $k=1$, from $\ep\in(0, \ep_\mu]$, we can apply Lemma \ref{lem.nonlinear.cpt} to \eqref{NS.ep} by putting in \eqref{modified.NS.ep}, 
\begin{align*}
(U^{\ep}, P^{\ep}) = (u^\ep, p^\ep)\,, \qquad
\lambda^\ep = 1\,, \qquad
C^\ep_j = 0\,, \ \ j\in\{1,2\}\,, \qquad
F^\ep = 0\,.
\end{align*}
Thus, if we set $a^\ep_{1,j}=(\overline{\partial_{3} u^\ep_j})_{B^\ep_{\theta,+}(0)}$, $j\in\{1,2\}$, the assertion \eqref{est3.lem.nonlinear.itr} for the case $k=1$ follows. Moreover, from \eqref{est2.lem.nonlinear.itr}, the Caccioppoli inequality in Lemma \ref{appendix.lem.Caccioppoli.ineq.} with $\rho=\theta$ and $r=1$ leads to 
\begin{align*}
\|\nabla u^{\ep_k}\|_{L^2(B^{\ep_k}_{\theta,+}(0))}^2 
& \le 
K_0 (1-\theta)^{-2} \big(\|u^\ep\|_{L^2(B^\ep_{1,+}(0))}^2
+ (1-\theta)^{-2} \|u^\ep\|_{L^2(B^\ep_{1,+}(0))}^6 \big) \\
& \le 
K_0 (1-\theta)^{-2} \big(4 + 2^6(1-\theta)^{-2}M^4\big) M^2\,.
\end{align*}
Hence we obtain \eqref{est4.lem.nonlinear.itr} for $k=1$ from
\begin{align}\label{est1.proof.lem.nonlinear.itr}
\begin{split}
\sum_{j=1}^{2} |a^\ep_{1,j}|
& \le 
2 |B^\ep_{\theta,+}(0)|^{-\frac12}
\|\nabla u^\ep\|_{L^2(B^\ep_{\theta,+}(0))} \\
& \le 
K_0^\frac12 \theta^{-\frac32} (1-\theta)^{-1} 
\big(4 + 2^6(1-\theta)^{-2}M^4\big)^\frac12 M\,.
\end{split}
\end{align}
Next we assume that \eqref{est3.lem.nonlinear.itr} and \eqref{est4.lem.nonlinear.itr} hold for $k\in\N$ and let $\ep\in(0, \theta^{k+2(2+\mu)} \ep_{\mu}^2]$. We define $U^{\ep/\theta^k}=(U^{\ep/\theta^k}_1(y), U^{\ep/\theta^k}_2(y), U^{\ep/\theta^k}_3(y))^\top$ and $P^{\ep/\theta^k}=P^{\ep/\theta^k}(y)$ on $B^{\ep/\theta^k}_{1,+}(0)$ by
\begin{align*}
U^{\ep/\theta^k}(y)
& = \frac{1}{\theta^{(1+\mu)k}}
\Big( u^\ep(\theta^k y) 
- \sum_{j=1}^{2}
\theta^k a^\ep_{k,j} \big(y_3 {\bf e}_j 
+ \frac{\ep}{\theta^k} v^{(j)}(\frac{\theta^k y}{\ep})\big) \Big)\,, \\
P^{\ep/\theta^k}(y) 
& = \frac{1}{\theta^{\mu k}} \Big( p^\ep(\theta^k y) 
- \sum_{j=1}^{2} a^\ep_{k,j} q^{(j)}(\frac{\theta^k y}{\ep})\big) \Big)\,.
\end{align*}
After a direct computation, we see that $(U^{\ep/\theta^k}, P^{\ep/\theta^k})$ is a weak solution to
\begin{equation}\label{eq1.proof.lem.nonlinear.itr}
\left\{
\begin{aligned}
&-\Delta_y U^{\ep/\theta^k} +\nabla_y P^{\ep/\theta^k} 
= -\nabla_y \cdot 
\big( U^{\ep/\theta^k} \otimes (\theta^{k} b^{\ep/\theta^k})
+ (\theta^{k} b^{\ep/\theta^k}) \otimes U^{\ep/\theta^k} \big) \\
& \qquad\qquad\qquad\qquad\qquad\quad \ 
-\theta^{(2+\mu)k} U^{\ep/\theta^k} \cdot\nabla_y U^{\ep/\theta^k} \\
& \qquad\qquad\qquad\qquad\qquad\quad \ 
+ \nabla_y \cdot F^{\ep/\theta^k} \ \ \mbox{in}\ B^{\ep/\theta^k}_{1,+}(0) \\
&\nabla_y\cdot U^{\ep/\theta^k}=0 \ \ \mbox{in}\ B^{\ep/\theta^k}_{1,+}(0) \\
&U^{\ep/\theta^k}=0 \ \ \mbox{on}\ \Gamma^{\ep/\theta^k}_1(0)\,,
\end{aligned}\right.
\end{equation}
where $b^{\ep/\theta^k}=b^{\ep/\theta^k}(y)$ and $F^{\ep/\theta^k}=F^{\ep/\theta^k}(y)$ are respectively defined on $B^{\ep/\theta^k}_{1,+}(0)$ by
\begin{align*}
&b^{\ep/\theta^k}(y)
= \sum_{j=1}^{2} C^\ep_{j,k} 
\big( y_3 {\bf e}_j + \frac{\ep}{\theta^k} v^{(j)}(\frac{\theta^k y}{\ep})\big)\,, 
\qquad
C^\ep_{j,k}  = \theta^{k} a^{\ep}_{k,j}\,, \\
&F^{\ep/\theta^k}(y)
= -\theta^{-\mu k}
\Big(b^{\ep/\theta^k}(y)\otimes b^{\ep/\theta^k}(y)
- \big(\sum_{j=1}^{2} C^\ep_{j,k} y_3 {\bf e}_j\big)
\otimes \big(\sum_{j=1}^{2} C^\ep_{j,k} y_3 {\bf e}_j\big) \Big)\,. 
\end{align*}
Note that $\nabla_y\cdot b^{\ep/\theta^k}=0$ in $B^{\ep/\theta^k}_{1,+}(0)$ and $b^{\ep/\theta^k}=0$ on $\Gamma^{\ep/\theta^k}_1(0)$. Moreover, we can subtract 
\begin{align*}
\big(\sum_{j=1}^{2} C^\ep_{j,k} y_3 {\bf e}_j\big)
\otimes \big(\sum_{j=1}^{2} C^\ep_{j,k} y_3 {\bf e}_j\big) 
\end{align*}
from $b^{\ep/\theta^k}\otimes b^{\ep/\theta^k}$ beforehand, since it vanishes if we take its divergence. This is indeed a crucial fact in the following proof where we cancel singularities in $\theta^{-1}$ by choosing $\ep$ small with respect to $\theta$ as in \eqref{est1.lem.nonlinear.itr}. From the recurrence hypothesis, \eqref{est3.lem.nonlinear.itr} at rank $k$, we also have
\begin{align}\label{est2.proof.lem.nonlinear.itr}
\dashint_{B^{\ep/\theta^k}_{1,+}(0)} |U^{\ep/\theta^k}|^2 
\le M^2
\end{align}
by a change of variables. Let us estimate $b^{\ep/\theta^k}$ and $F^{\ep/\theta^k}$. From the recurrence hypothesis 
\begin{align}\label{est3.proof.lem.nonlinear.itr}
\sum_{j=1}^{2} |a^\ep_{k,j}|
\le 
4K_0^\frac12 \theta^{-\frac32} (1-\theta^\mu)^{-1} (6 + 2^8M^4)^\frac12 M
\end{align}
holds, where $(1-\theta)^{-1}\le2$ was used. We have uniformly in $k\in\N$, 
\begin{align}\label{est4.proof.lem.nonlinear.itr}
|\theta^{k} C^\ep_{j,k}|
\le 
4K_0^\frac12 (1-\theta^\mu)^{-1} (6 + 2^8M^4)^\frac12 M \theta^{\frac12}\leq 1\,
\end{align}
by \eqref{est7.proof.lem.nonlinear.itr}. Moreover, by \eqref{est1.lem.est.boundarylayer} in Lemma \ref{lem.est.boundarylayer} and $\ep\in(0, \theta^{k+2(2+\mu)} \ep_{\mu}^2]$, we see that 
\begin{align}\label{est5.proof.lem.nonlinear.itr}
\begin{split}
\|\nabla_y (\theta^{k} b^{\ep/\theta^k})\|_{L^2(B^{\ep/\theta^k}_{1,+}(0))} 
& \le
C \big(\sum_{j=1}^{2} |\theta^{k} C^\ep_{j,k}| \big)
(1 + \ep^\frac12 \theta^{-\frac{k}{2}}) \\
& \le
C_1 (1-\theta^\mu)^{-1} (6 + 2^8M^4)^\frac12 M \theta^{\frac12}\,,
\end{split}
\end{align}
where $C_1$ is independent of $k$, $M$, $\theta$, and $\ep$, while the definition of $F^{\ep/\theta^k}$ implies that for $y\in B^{\ep/\theta^k}_{1,+}(0)$,
\begin{align*}
|F^{\ep/\theta^k}(y)|
& \le
C \theta^{(2-\mu)k}
\big(\sum_{j=1}^{2} |a^\ep_{k,j}|\big)^2
\sum_{j=1}^{2} 
\big((\frac{\ep}{\theta^{k}}) \big|v^{(j)}(\frac{\theta^k y}{\ep})\big|
+ (\frac{\ep}{\theta^{k}})^2 \big|v^{(j)}(\frac{\theta^k y}{\ep})\big|^2 \big)\,.
\end{align*}
Thus, from \eqref{est2.lem.est.boundarylayer} with $m=0$ and $m=2$ in Lemma \ref{lem.est.boundarylayer}, we have again by $\ep\in(0, \theta^{k+2(2+\mu)} \ep_{\mu}^2]$,
\begin{align}\label{est6.proof.lem.nonlinear.itr}
\begin{split}
\|F^{\ep/\theta^k}\|_{L^2(B^{\ep/\theta^k}_{1,+}(0))}
& \le
C \theta^{(2-\mu)k-3} (1-\theta^\mu)^{-2} (6 + 2^8M^4) M^2
(\ep^\frac12 \theta^{-\frac{k}{2}} + \ep \theta^{-k}) \\
& \le
C \theta^{-1-\mu} (1-\theta^\mu)^{-2} (6 + 2^8M^4) M^2
(\ep_{\mu} \theta^{2+\mu} + \ep_{\mu}^2 \theta^{2(2+\mu)}) \\
& \le
\big(C_2 (1-\theta^\mu)^{-2} (6 + 2^8M^4) M \theta\big) M \ep_{\mu}\,,
\end{split}
\end{align}
where $C_2$ is independent of $k$, $M$, $\theta$, and $\ep$. Then, from \eqref{est2.proof.lem.nonlinear.itr} combined with \eqref{est4.proof.lem.nonlinear.itr} and \eqref{est6.proof.lem.nonlinear.itr} under \eqref{est7.proof.lem.nonlinear.itr}, since $\ep/\theta^k\in(0, \ep_\mu]$, we can apply Lemma \ref{lem.nonlinear.cpt} to \eqref{eq1.proof.lem.nonlinear.itr} by putting 
\begin{align*}
& (U^{\ep}, P^{\ep}) = (U^{\ep/\theta^k}, P^{\ep/\theta^k})\,, \qquad
\lambda^\ep=\theta^{(2+\mu)k}\,, \\
& C^\ep_j = \theta^{k} C^\ep_{j,k}\,, \ \ j\in\{1,2\}\,, \qquad
F^\ep = F^{\ep/\theta^k}
\end{align*}
in \eqref{modified.NS.ep} and find that
\begin{align*}
&\dashint_{B^{\ep/\theta^k}_{\theta,+}(0)}
\big|U^{\ep/\theta^k}(y) 
- \sum_{j=1}^{2}
(\overline{\partial_{y_3} U^{\ep/\theta^k}_j})_{B^{\ep/\theta^k}_{\theta,+}(0)} 
\big( y_3 {\bf e}_j + \frac{\ep}{\theta^k} v^{(j)}(\frac{\theta^k y}{\ep})\big) \big|^2 \dd y \\
&\le
M^2 \theta^{2+2\mu}\,.
\end{align*}
A change of variables yields that
\begin{align}\label{goal.proof.lem.nonlinear.itr}
\dashint_{B^\ep_{\theta^{k+1},+}(0)}
\big|u^\ep(x) 
- \sum_{j=1}^{2}
a^\ep_{k+1,j} \big( x_3 {\bf e}_j + \ep v^{(j)}(\frac{x}{\ep})\big) \big|^2 \dd x
\le
M^2 \theta^{(2+2\mu)(k+1)}\,,
\end{align}
where the number $a^\ep_{k+1,j}$, $j\in\{1,2\}$, is defined as
\begin{align}\label{def1.proof.lem.nonlinear.itr}
a^\ep_{k+1,j} = a^\ep_{k,j} 
+  \theta^{\mu k}
(\overline{\partial_{y_3} U^{\ep/\theta^k}_j})_{B^{\ep/\theta^k}_{\theta,+}(0)}\,.
\end{align}
Let us estimate $a^\ep_{k+1,j}$. 
By \eqref{est2.proof.lem.nonlinear.itr} and \eqref{est5.proof.lem.nonlinear.itr} under \eqref{est7.proof.lem.nonlinear.itr} we have
\begin{align*}
&\Big(\|\nabla_y (\theta^{k} b^{\ep/\theta^k})\|_{L^2(B^{\ep/\theta^k}_{1,+}(0))}^4 
+ (1-\theta)^{-\frac43} \|\nabla_y (\theta^{k} b^{\ep/\theta^k})\|_{L^2(B^{\ep/\theta^k}_{1,+}(0))}^\frac43
\Big) \\
&\quad
\times \|U^{\ep/\theta^k}\|_{L^2(B^{\ep/\theta^k}_{1,+}(0))}^2 \\
&\le
(1-\theta)^{-2} M^2\,.
\end{align*}
Then \eqref{est6.proof.lem.nonlinear.itr} under \eqref{est7.proof.lem.nonlinear.itr} and the Caccioppoli inequality applied to \eqref{eq1.proof.lem.nonlinear.itr} with $\rho=\theta$ and $r=1$ lead to
\begin{align*}
\begin{split}
& \|\nabla_y U^{\ep/\theta^k}\|_{L^2(B^{\ep/\theta^k}_{\theta,+}(0))}^2 \\
& \le 
K_0 (1-\theta)^{-2} \Big(
\|U^{\ep/\theta^k}\|_{L^2(B^{\ep/\theta^k}_{1,+}(0))}^2
+ (1-\theta)^{-2} \|U^{\ep/\theta^k}\|_{L^2(B^{\ep/\theta^k}_{1,+}(0))}^6 
+ 2M^2\Big) \\
& \le 
K_0 (1-\theta)^{-2} \big(6 + 2^6 (1-\theta)^{-2} M^4\big) M^2\,.
\end{split}
\end{align*}
Therefore from the H\"older inequality we obtain
\begin{align*}
\begin{split}
\sum_{j=1}^{2} 
\big|(\overline{\partial_{y_3} U^{\ep/\theta^k}_j})_{B^{\ep/\theta^k}_{\theta,+}(0)}\big|
& \le 
2 |B^{\ep/\theta^k}_{\theta,+}(0)|^{-\frac12}
\|\nabla_y U^{\ep/\theta^k}\|_{L^2(B^{\ep/\theta^k}_{\theta,+}(0))} \\
& \le 
K_0^\frac12 \theta^{-\frac32} (1-\theta)^{-1} 
\big(6 + 2^6 (1-\theta)^{-2} M^4\big)^\frac12
M\,.
\end{split}
\end{align*}
Thus, by the recurrence hypothesis, \eqref{est4.lem.nonlinear.itr} at rank $k$, and \eqref{def1.proof.lem.nonlinear.itr}, we have
\begin{align*}
\sum_{j=1}^{2} |a^\ep_{k+1,j}| 
& \le
\sum_{j=1}^{2} |a^\ep_{k,j}| 
+ \theta^{\mu k} \sum_{j=1}^{2} 
\big|(\overline{\partial_{y_3} U^{\ep/\theta^k}_j})_{B^{\ep/\theta^k}_{\theta,+}(0)}\big| \\
& \le 
K_0^\frac12 \theta^{-\frac32} (1-\theta)^{-1} 
\big(6 + 2^6 (1-\theta)^{-2} M^4\big)^\frac12 M
\sum_{l=1}^{k+1} \theta^{\mu(l-1)}\,,
\end{align*}
which with \eqref{goal.proof.lem.nonlinear.itr} proves the assertions \eqref{est3.lem.nonlinear.itr} and \eqref{est4.lem.nonlinear.itr} at rank $k+1$. This completes the proof.
\end{proof}
%
\subsection{Proof of Theorems \ref{theo.lip.nonlinear} and \ref{theo.lip.nonlinear.periodic}}\label{subsec.nonlinear.proof}
%
Firstly we prove Theorem \ref{theo.lip.nonlinear} by applying Lemma \ref{lem.nonlinear.itr}. Throughout this subsection, for given $M\in(0,\infty)$ and $\mu\in(0,1)$, let $\theta\in(0,\frac18)$ and $\ep_{\mu}\in(0,1)$ be the corresponding constants in Lemma \ref{lem.nonlinear.itr}. Note that, for any $k\in\N$, we have
\begin{align*}
(0, \theta^{k-1} (\theta^{2(2+\mu)} \ep_\mu^2)]
\subset (0, \theta^{k+2(2+\mu)(1-\delta_{1k})-1} \ep_{\mu}^{2-\delta_{1k}}]\,.
\end{align*}
%
\begin{proofx}{Theorem \ref{theo.lip.nonlinear}} We fix $\mu\in(0,1)$ and set $\ep^{(1)}=\theta^{2(2+\mu)}\ep_\mu^2$. Let $\ep\in(0, \ep^{(1)}]$. 
As in the proof of Theorem \ref{theo.lip} in Subsection \ref{subsec.linear.proof}, we can focus on the case $r\in[\ep/\ep^{(1)}, \theta]$. For 
any given $r\in[\ep/\ep^{(1)},\theta]$, there exists $k\in\N$ 
with $k\ge2$ 
such that $r\in(\theta^k, \theta^{k-1}]$. From the bound \eqref{est1.theo.lip.nonlinear} and $\ep\in(0,\theta^{k-1} \ep^{(1)}]$, one can apply Lemma \ref{lem.nonlinear.itr}. By using an easy estimate of $a^\ep_{k,j}\in\R$, $j\in\{1,2\}$:
\begin{align}\label{est1.proof.theo.lip.nonlinear}
\sum_{j=1}^{2} |a^\ep_{k,j}| 
\le 
C \theta^{-\frac32}  (1-\theta^\mu)^{-1} (1+M^4)^\frac12 M
\end{align}
with a constant $C$ depends only on $\|\gamma\|_{W^{1,\infty}(\R^2)}$, we see that
\begin{align*}
& \bigg(\dashint_{B^\ep_{r,+}(0)} |u^\ep|^2 \bigg)^\frac12 
\le
\bigg(\theta^{-3} \dashint_{B^\ep_{\theta^{k-1},+}(0)} |u^\ep|^2 \bigg)^\frac12 \\
& \le
\theta^{-\frac32}
\bigg(
\dashint_{B^\ep_{\theta^{k-1},+}(0)}
\big|u^\ep(x) 
- \sum_{j=1}^{2}
a^\ep_{k-1,j}
\big( x_3 {\bf e}_j + \ep v^{(j)}(\frac{x}{\ep})\big) \big|^2 \dd x
\bigg)^\frac12 \\
& \quad
+ \theta^{-\frac32}
\Big(\sum_{j=1}^{2} |a^\ep_{k-1,j}|\Big)
\bigg(\sum_{j=1}^{2} \dashint_{B^\ep_{\theta^{k-1},+}(0)}
\big| x_3 {\bf e}_j + \ep v^{(j)}(\frac{x}{\ep})\big) \big|^2 \dd x \bigg)^\frac12 \\
& \le
M \theta^{(1+\mu)(k-1)-\frac32} \\
& \quad
+ C \theta^{-3}  (1-\theta^\mu)^{-1} (1+M^4)^\frac12 M
\bigg(\sum_{j=1}^{2} \dashint_{B^\ep_{\theta^{k-1},+}(0)}
\big| x_3 {\bf e}_j + \ep v^{(j)}(\frac{x}{\ep})\big) \big|^2 \dd x \bigg)^\frac12\,.
\end{align*}
Then, in the same way as in the proof of Theorem \ref{theo.lip}, we have
\begin{align*}
\bigg(\dashint_{B^\ep_{r,+}(0)} |u^\ep|^2\bigg)^\frac12 
\le
\Big(\theta^{-\frac52-\mu} r^\mu
+ C \theta^{-4}  (1-\theta^\mu)^{-1} (1 + (\ep^{(1)})^\frac12)
(1+M^4)^\frac12
\Big) M r\,.
\end{align*}
Hence we obtain the assertion \eqref{est2.theo.lip.nonlinear} by letting $\mu=\frac12$ for instance and 
by defining $C^{(1)}_M$ by
\begin{align*}
C^{(1)}_M = \Big(\theta^{ -3} 
+ C \theta^{-4}  (1-\theta^{\frac12})^{-1} (1 + (\ep^{(1)})^\frac12)
(1+M^4)^\frac12
\Big) M\,.
\end{align*}
Indeed, it is easy to see that $C^{(1)}_M$ increases monotonically in $M$ if one chooses $\theta$ to be the supremum of the numbers $\theta$ satisfying \eqref{est7.proof.lem.nonlinear.itr} with 
$\mu=\frac12$. 
Moreover $C^{(1)}_M$ converges to zero when $M\to0$ from this choice of $\theta$. The proof is complete if we combine the trivial estimate for $r\in(\theta, 1]$. 
\end{proofx}
%
\noindent Next we prove Theorem \ref{theo.lip.nonlinear.periodic}. Let $\alpha^{(j)}\in\R^3$, $j\in\{1,2\}$, be the constant vector in Proposition \ref{prop.per.BL}. 
%
\begin{proofx}{Theorem \ref{theo.lip.nonlinear.periodic}} As in the proof of Theorem \ref{theo.lip.nonlinear}, we set $\ep^{(2)}=\theta^{2(2+\mu)}\ep_\mu^2$ and take $\ep\in(0, \ep^{(2)}]$.

\noindent (i) 
We focus on the case $r\in[\ep/\ep^{(1)}, \theta]$ again as in the proof of Theorem \ref{theo.lip.nonlinear}. Since every $r\in[\ep/\ep^{(2)},\theta]$ satisfies $r\in(\theta^k, \theta^{k-1}]$ with some $k\in\N$ satisfying $k\ge2$ we have
\begin{align*}
\begin{split}
& \bigg( \dashint_{B^\ep_{r,+}(0)}
\big|u^\ep(x) 
- \sum_{j=1}^{2}
 a^\ep_{k-1,j} 
x_3 {\bf e}_j \big|^2 \dd x
\bigg)^\frac12 \\
& \le
\bigg(\theta^{-3}
\dashint_{B^\ep_{\theta^{k-1},+}(0)}
\big|u^\ep(x) 
- \sum_{j=1}^{2}
 a^\ep_{k-1,j} 
x_3 {\bf e}_j \big|^2 \dd x
\bigg)^\frac12 \\
& \le
M \theta^{(1+\mu)(k-1)-{\frac32}} \\
& \quad
+ C \theta^{-3}  (1-\theta^\mu)^{-1} (1+M^4)^\frac12 M \ep
\bigg(
\sum_{j=1}^{2}
\dashint_{B^\ep_{\theta^{k-1},+}(0)}
\big| v^{(j)}(\frac{x}{\ep}) \big|^2 \dd x
\bigg)^\frac12\,,
\end{split}
\end{align*}
where Lemma \ref{lem.nonlinear.itr} has been applied in the third line. The estimate \eqref{est1.proof.theo.lip.nonlinear} for $a^\ep_{k-1,j} \in\R$, $j\in\{1,2\}$, is also used in the same line. Then \eqref{est2.lem.est.boundarylayer} with $m=0$ in Lemma \ref{lem.est.boundarylayer} and $\theta^{k-1}\in(0,\theta^{-1} r)$ lead to
\begin{align*}
\begin{split}
& \bigg( \dashint_{B^\ep_{r,+}(0)}
\big|u^\ep(x) 
- \sum_{j=1}^{2}
a^\ep_{k-1,j} x_3 {\bf e}_j \big|^2 \dd x
\bigg)^\frac12 \\
& \le
\Big(\theta^{-\frac52-\mu} r^{1+\mu} 
+ C \theta^{-\frac72} (1-\theta^\mu)^{-1} (1+M^4)^\frac12 \ep^\frac12 r^{\frac12} 
\Big) M\,.
\end{split}
\end{align*}
Hence we obtain the assertion \eqref{est1.theo.lip.nonlinear.periodic} by defining $c^\ep_{r,j}$ and $C^{(2)}_M$ by
\begin{align}\label{def1.proof.theo.lip.nonlinear.periodic}
c^\ep_{r,j} = a^\ep_{k-1,j} \,, \qquad
C^{(2)}_M = \Big(\theta^{-\frac52-\mu} 
+ C \theta^{-\frac72} (1-\theta^\mu)^{-1} (1+M^4)^\frac12
\Big) M\,,
\end{align}
and by combining the trivial estimate for $r\in(\theta, 1]$.

\noindent (ii) In a similar way as in (i), for $r\in[\ep/\ep^{(2)},\theta]$ with $r\in(\theta^k, \theta^{k-1}]$, we have
\begin{align}\label{est1.proof.theo.lip.nonlinear.periodic}
\begin{split}
& \bigg( \dashint_{B^\ep_{r,+}(0)}
\big|u^\ep(x) 
- \sum_{j=1}^{2}
c^\ep_{r,j}
(x_3 {\bf e}_j + \ep \alpha^{(j)}) \big|^2 \dd x
\bigg)^\frac12 \\
& \le
M \theta^{(1+\mu)(k-1)-\frac32} \\
& \quad
+ C \theta^{-3}(1-\theta^\mu)^{-1}(1+M^4)^\frac12 M \ep
\bigg(
\sum_{j=1}^{2}
\dashint_{B^\ep_{\theta^{k-1},+}(0)}
\big| v^{(j)}(\frac{x}{\ep}) - \alpha^{(j)} \big|^2 \dd x
\bigg)^\frac12\,,
\end{split}
\end{align}
where Lemma \ref{lem.nonlinear.itr} and the estimate \eqref{est1.proof.theo.lip.nonlinear} are applied again. Moreover, the notation $c^\ep_{r,j} =a^\ep_{k-1,j} $ in \eqref{def1.proof.theo.lip.nonlinear.periodic} is used. Then \eqref{est2.lem.est.boundarylayer} with $m=0$ in Lemma \ref{lem.est.boundarylayer} and \eqref{est1.prop.per.BL} in Proposition \ref{prop.per.BL} lead to
\begin{align*}
\begin{split}
&
\bigg( \sum_{j=1}^{2} \dashint_{B^\ep_{\theta^{k-1},+}(0)}
\big| v^{(j)}(\frac{x}{\ep}) - \alpha^{(j)} \big|^2 \dd x
\bigg)^\frac12 \\
& \le
\theta^{-\frac32(k-1)}
\bigg( \sum_{j=1}^{2} \int_{B^\ep_{2\ep,+}(0)}
\big| v^{(j)}(\frac{x}{\ep}) - \alpha^{(j)} \big|^2 \dd x
\bigg)^\frac12 \\
& \quad
+ \theta^{-\frac32(k-1)}
\bigg(
\sum_{j=1}^{2} \int_{(-\theta^{k-1}, \theta^{k-1})^2}
\int_{\ep}^{\theta^{k-1}}
\big| v^{(j)}(\frac{x}{\ep}) - \alpha^{(j)} \big|^2 \dd x_3 \dd x'
\bigg)^\frac12 \\
& \le
C (\ep^\frac32 \theta^{-\frac32(k-1)}
+ \ep^\frac12 \theta^{-\frac12(k-1)})\,.
\end{split}
\end{align*}
Hence, by $\theta^{k-1}\in(0,\theta^{-1} r)$, $r^{-1}\in[\theta^{-(k-1)}, \theta^{-k})$, and $\ep\in(0,\theta^{k-1} \ep^{(2)}]$, from \eqref{est1.proof.theo.lip.nonlinear.periodic} we find
\begin{align*}
\begin{split}
& \bigg(
\dashint_{B^\ep_{r,+}(0)}
\big|u^\ep(x) 
- \sum_{j=1}^{2}
c^\ep_{r,j}
(x_3 {\bf e}_j + \ep \alpha^{(j)}) \big|^2 \dd x
\bigg)^\frac12 \\
& \le
\Big(\theta^{-\frac52-\mu} r^{1+\mu} 
+ C \theta^{-3}  (1-\theta^\mu)^{-1} 
(1+\ep^{(2)})
(1+M^4)^\frac12 \ep^\frac32 r^{-\frac12} 
\Big) M\,.
\end{split}
\end{align*}
The assertion \eqref{est2.theo.lip.nonlinear.periodic} follows by setting
\begin{align*}
\widetilde{C^{(2)}_M} = \Big(\theta^{-\frac52-\mu} 
+ C \theta^{-3} (1-\theta^\mu)^{-1} 
(1+\ep^{(2)})
(1+M^4)^\frac12
\Big) M\,.
\end{align*}
This completes the proof of Theorem \ref{theo.lip.nonlinear.periodic}
 by using the trivial estimate for $r\in(\theta, 1]$. 
\end{proofx}
%
\appendix
%
\section{Regularity theory}\label{appendix.Regularity}
\noindent In this appendix we recall the regularity results for the Stokes equations
\begin{equation}\label{appendix.S}
\left\{
\begin{array}{ll}
-\Delta u+\nabla p = 0&\mbox{in}\ B_{\frac12,+}(0)\\
\nabla\cdot u=0&\mbox{in}\ B_{\frac12,+}(0)\\
u=0&\mbox{on}\ \Gamma_{\frac12}(0)
\end{array}
\right.
\end{equation}
and the modified Navier-Stokes equations
\begin{equation}\label{appendix.NS}
\left\{
\begin{array}{ll}
-\Delta u+\nabla p 
= -\nabla \cdot (u \otimes b + b \otimes u) 
- \lambda u\cdot\nabla u &\mbox{in}\ B_{\frac12,+}(0)\\
\nabla\cdot u=0&\mbox{in}\ B_{\frac12,+}(0)\\
u=0&\mbox{on}\ \Gamma_{\frac12}(0)\,,
\end{array}
\right.
\end{equation}
where $b=b(x)$ is defined as $b(x) = \sum_{j=1}^{2} C_j x_3 {\bf e}_j$.
%
\begin{lemma}\label{appendix.lem.regl}
{\rm (i)} Let $(u, p) \in H^1(B_{\frac12,+}(0))^3\times L^2(B_{\frac12,+}(0))$ be a weak solution to \eqref{appendix.S}. Then for all $r\in(0,\frac{7}{16})$, we have
\begin{align}\label{est1.appendix.lem.regl}
u\in C^\infty(\overline{B_{r,+}(0)})^3\,,\qquad
p\in C^\infty(\overline{B_{r,+}(0)})\,,
\end{align}
and for all $k\in\N\cup\{0\}$, we have
\begin{align}\label{est2.appendix.lem.regl}
\|u\|_{C^k(\overline{B_{r,+}(0)})} \le K_1\|u\|_{L^2(B_{\frac12,+}(0))}\,,
\end{align}
where the constant $K_1$ depends only on 
$k$. \\
\noindent {\rm (ii)} Let $(\lambda,C_1,C_2)\in\R^3$ and let $(u, p) \in H^1(B_{\frac12,+}(0))^3\times L^2(B_{\frac12,+}(0))$ be a weak solution to \eqref{appendix.NS}.
Then for all $r\in(0,\frac{7}{16})$, we have
\begin{align}\label{est3.appendix.lem.regl}
u\in C^\infty(\overline{B_{r,+}(0)})^3\,,\qquad
p\in C^\infty(\overline{B_{r,+}(0)})\,,
\end{align}
and for all $k\in\N\cup\{0\}$, we have
\begin{align}\label{est4.appendix.lem.regl}
\|u\|_{C^k(\overline{B_{r,+}(0)})} \le K_2\,,
\end{align}
where the constant $K_2$ depends nonlinearly on $(\lambda,C_1,C_2)$, $\|u\|_{L^2(B_{\frac12,+}(0))}$, and $k$.
\end{lemma}
%
\begin{proof}
(i) Fix $r\in(0,\frac{7}{16})$ and for $l\in\N$ with $l\ge2$ let $r<\ldots\, <r_l<r_{l-1}<\ldots\, <\frac7{16}$ and $\Omega_l$ be a domain with a smooth boundary such that
\begin{align*}
B_{r,+}(0) \subset \Omega_{l} \subset B_{r_{l},+}(0)
\subset \Omega_{l-1} \subset B_{r_{l-1},+}(0) \subset B_{\frac{7}{16},+}(0)\,, 
\qquad 
\Gamma_r(0) \subset \overline{\Omega_{l}}\cap\Gamma_1(0)\,.
\end{align*}
Then it suffices to show for all $l\in\N$ with $l\ge2$
\begin{align}\label{est1.proof.appendix.lem.regl}
\|u\|_{W^{l,2}(\Omega_{l})} + \|p\|_{W^{l-1,2}(\Omega_{l})}
\le L_l\|u\|_{L^2(B_{\frac12,+})}\,,
\end{align}
where $L_l$ depends only on $l$. Let $\varphi_l\in C_0^\infty(\overline{\Omega_{l-1}})$ be a cut-off function  such that
\begin{align*}
\varphi_l =1 \ \, {\rm in} \ \, \Omega_l\,, \qquad
\supp \varphi_l \subset \Omega_{l-1}\setminus\Gamma_1(0)\,,
\end{align*}
and let $\mathbb{B}[\nabla \varphi_l \cdot u]$ be the Bogovskii corrector in $\Omega_{l-1}$. Note that $\mathbb{B}[\nabla \varphi_l \cdot u]$ satisfies 
\begin{align*}
\supp\mathbb{B}[\nabla \varphi_l \cdot u] \subset \Omega_{l-1}\,, \qquad 
\nabla\cdot\mathbb{B}[\nabla \varphi_l \cdot u] = \nabla \varphi_l \cdot u
\end{align*}
and estimates for all $m\in\N$
\begin{align*}
\|\nabla^{m+1} \mathbb{B} [\nabla \varphi_l \cdot u]\|_{L^2(\Omega_{l-1})} 
\le C 
\|\nabla^{m}(\nabla \varphi_l \cdot u) \|_{L^2(\Omega_{l-1})}\,.
\end{align*}
See \cite{BS} or \cite{Galdibook} for a proof. Then by setting
\begin{align}\label{def1.proof.appendix.lem.regl}
u_l = \varphi_l u - \mathbb{B} [\nabla \varphi_l \cdot u]\,, \qquad 
p_l = \varphi_l p\,,
\end{align}
we easily see that $(u_l, p_l)$ solves
\begin{equation*}
\left\{
\begin{array}{ll}
-\Delta u_l+\nabla p_l = f_l(u,p) &\mbox{in}\ \Omega_{l-1}\\
\nabla\cdot u_l =0&\mbox{in}\ \Omega_{l-1}\\
u_l=0&\mbox{on}\ \partial\Omega_{l-1}\,,
\end{array}
\right.
\end{equation*}
where 
\begin{align*}
f_l(u,p) = -(\Delta \varphi_l) u - 2\nabla u \nabla\varphi_l 
+ \Delta \mathbb{B} [\nabla \varphi_l \cdot u] + p\nabla\varphi_l\,.
\end{align*}
Now we prove \eqref{est1.proof.appendix.lem.regl} by induction on $l\in\N$ with $l\ge2$. We assume that $\int_{B_{\frac{7}{16},+}(0)}p=0$ without loss of generality. For $l=2$, from $\|p\|_{L^2(B_{\frac{7}{16},+}(0))}\le \|\nabla p\|_{W^{-1,2}(B_{\frac{7}{16},+}(0))}$ and $\nabla p=\Delta u$, we have
\begin{align*}
\|f_2(u,p)\|_{L^2(\Omega_{1})} 
& \le 
C(\|\nabla u\|_{L^2(B_{\frac{7}{16},+}(0))} 
+ \|p\|_{L^2(B_{\frac{7}{16},+}(0))}) \\
& \le 
C \|\nabla u\|_{L^2(B_{\frac{7}{16},+}(0))}\,,
\end{align*}
which implies
\begin{align*}
\|u_2\|_{W^{2,2}(\Omega_{1})} + \|p_2\|_{W^{1,2}(\Omega_{1})}
\le \widetilde{L_2}
\end{align*}
by the regularity theory of the Stokes equations; see \cite[Subsection 1.5, 
III]{So} for example. Moreover, one can choose $\widetilde{L_2}$ depending on $\|u\|_{L^2(B_{\frac12,+})}$ rather than on $\|\nabla u\|_{L^2(B_{\frac{7}{16},+}(0))}$ thanks to the Caccioppoli inequality in Lemma \ref{appendix.lem.Caccioppoli.ineq.} in Appendix \ref{appendix.Caccioppoli}. Then the definitions of $\varphi_2$ and $u_2$ imply \eqref{est1.proof.appendix.lem.regl} when $l=2$. Next we assume that \eqref{est1.proof.appendix.lem.regl} holds for $l\in\N$ with $l\ge2$. Since we have
\begin{align*}
\|f_{l+1}(u,p)\|_{W^{l-1,2}(\Omega_{l})} 
& \le 
C(\|u\|_{W^{l,2}(\Omega_{l})}  + \|p\|_{W^{l-1,2}(\Omega_{l})} ) \\
& \le 
C L_l\,,
\end{align*}
by the regularity theory again we see that
\begin{align*}
\|u_{l+1}\|_{W^{l+1,2}(\Omega_{l})} + \|p_{l+1}\|_{W^{l,2}(\Omega_{l})}
\le \widetilde{L_{l+1}}\,,
\end{align*}
which leads to \eqref{est1.proof.appendix.lem.regl} when $l+1$ from the definitions of $\varphi_{l+1}$ and $u_{l+1}$. Hence we have proved \eqref{est1.proof.appendix.lem.regl} for all $l\in\N$ with $l\ge2$, and therefore, 
the assertions \eqref{est1.appendix.lem.regl} and \eqref{est2.appendix.lem.regl} from the Sobolev embedding.

\noindent (ii) Fix $r\in(0,\frac{7}{16})$ and for $l\in\N$ with $l\ge2$ let us take the same domain $\Omega_l$ and cut-off function $\varphi_l$ as in (i). Then we see that $(u_l, p_l)$ defined in \eqref{def1.proof.appendix.lem.regl} now solves
\begin{equation}\label{eq1.proof.appendix.lem.regl}
\left\{
\begin{array}{ll}
-\Delta u_l+\nabla p_l
= - \lambda \varphi_l (u\cdot\nabla u)
+ g_l(u,p) &\mbox{in}\ \Omega_{l-1}\\
\nabla\cdot u_l=0&\mbox{in}\ \Omega_{l-1}\\
u_l=0&\mbox{on}\ \partial\Omega_{l-1}\,,
\end{array}
\right.
\end{equation}
where 
\begin{align*}
g_l(u,p) 
& = -(\Delta \varphi_l) u - 2\nabla u \nabla\varphi_l  
+ \Delta \mathbb{B} [\nabla \varphi_l \cdot u] + p\nabla\varphi_l
-\varphi_l \nabla \cdot (u \otimes b + b \otimes u)\,.
\end{align*}
From the same reasoning as in (i) we aim at proving
\begin{align}\label{est2.proof.appendix.lem.regl}
\|u\|_{W^{l-1,2}(\Omega_{l})} + \|p\|_{W^{l-2,2}(\Omega_{l})}
\le L_l
\end{align}
by induction on $l\in\N$ with $l\ge3$. Here $L_l$ depends nonlinearly on $(\lambda,C_1,C_2)$, $\|u\|_{L^2(B_{\frac12,+}(0))}$, and $l$. We assume that $\int_{B_{\frac{7}{16},+}(0)}p=0$ again. Firstly let us consider the equations \eqref{eq1.proof.appendix.lem.regl} with $l=2$. Then, by following the argument in the proof of \cite[Theorem 3.6.1, III]{So}, we see that
\begin{align*}
\|\nabla u_2\|_{L^{3}(\Omega_{1})} 
\le C(\|\nabla u\|_{L^2(B_{\frac{7}{16},+}(0))}^2 
+ \|\nabla u\|_{L^2(B_{\frac{7}{16},+}(0))})\,,
\end{align*}
and therefore, from the definition of $u_{2}$, that 
\begin{align}\label{est3.proof.appendix.lem.regl}
\|\nabla u\|_{L^{3}(\Omega_{2})} 
\le C(\|\nabla u\|_{L^2(B_{\frac{7}{16},+}(0))}^2 
+ \|\nabla u\|_{L^2(B_{\frac{7}{16},+}(0))})\,.
\end{align}
Now let us us go back to the equations \eqref{eq1.proof.appendix.lem.regl} with $l=3$. Then we have
\begin{align*}
\|g_3(u,p)\|_{L^2(\Omega_{2})} 
\le 
C \|\nabla u\|_{L^2(B_{\frac{7}{16},+}(0))}
\end{align*}
and, from \eqref{est3.proof.appendix.lem.regl} and the Sobolev inequality, 
\begin{align*}
\|\lambda \varphi_3 (u\cdot\nabla u)\|_{L^{2}(\Omega_{2})} 
\le C\|\nabla u\|_{L^2(B_{\frac{7}{16},+}(0))}
(\|\nabla u\|_{L^2(B_{\frac{7}{16},+}(0))}^2 
+ \|\nabla u\|_{L^2(B_{\frac{7}{16},+}(0))})\,.
\end{align*}
Then the regularity theory of the Stokes equations leads to 
\begin{align*}
\|u_3\|_{W^{2,2}(\Omega_{2})} + \|p_3\|_{W^{1,2}(\Omega_{2})}
\le \widetilde{L_3}\,,
\end{align*}
where, thanks to the Caccioppoli inequality, $\widetilde{L_3}$ can be chosen to depend on $\|u\|_{L^2(B_{\frac12,+})}$ rather than $\|\nabla u\|_{L^2(B_{\frac{7}{16},+}(0))}$. Therefore we have \eqref{est2.proof.appendix.lem.regl} when $l=3$. Next we assume that \eqref{est2.proof.appendix.lem.regl} holds for $l\in\N$ with $l\ge3$. Then by the Sobolev inequality we can obtain
\begin{align*}
\|g_{l+1}(u,p)\|_{W^{l-2,2}(\Omega_{l})} 
+ \|\lambda \varphi_{l+1} (u\cdot\nabla u)\|_{W^{l-2,2}(\Omega_{l})}
\le 
C(L_l)\,,
\end{align*}
where the constant $C(L_l)$ depending nonlinearly on $L_l$; see the proof of \cite[Theorem 3.6.1, III]{So} for details. Then by the regularity theory again we see that
\begin{align*}
\|u_{l+1}\|_{W^{l,2}(\Omega_{l})} + \|p_{l+1}\|_{W^{l-1,2}(\Omega_{l})}
\le \widetilde{L_{l+1}}\,,
\end{align*}
and that \eqref{est2.proof.appendix.lem.regl} holds when $l+1$. Hence we have \eqref{est2.proof.appendix.lem.regl} for all $l\in\N$ with $l\ge3$, which implies the assertions \eqref{est3.appendix.lem.regl} and \eqref{est4.appendix.lem.regl} from the Sobolev embedding. The proof of Lemma \ref{appendix.lem.regl} is complete.
\end{proof}
%
\section{Caccioppoli inequality}\label{appendix.Caccioppoli}
\noindent In this appendix we prove the Caccioppoli inequality for the Stokes and Navier-Stokes equations. Firstly we prepare a technical lemma which will be used in the proof of Lemma \ref{appendix.lem.Caccioppoli.ineq.}.
%
\begin{lemma}\label{appendix.lem.itr.}
Let $f$, $g$, $h_1$, $h_2$, and $h_3$ are non-negative and monotone increasing functions in $C^1([0,1])$ such that for some $\delta\in(0,\frac{1}{16})$ and for all $0<\rho<r\le1$,
\begin{align*}
f(\rho) \le \delta f(r)  
+ C_\delta \Big( g(r) 
+ \frac{h_1(r)}{(r-\rho)^\frac43} 
+ \frac{h_2(r)}{(r-\rho)^2} 
+ \frac{h_3(r)}{(r-\rho)^4} \Big)\,,
\end{align*}
where the constant $C_\delta$ depends on $\delta$. Then we have for all $0<\rho<r\le1$, 
\begin{align}\label{est.appendix.lem.itr.}
f(\rho) 
\le 
C_\delta \Big(
\frac{g(r)}{1-\delta} 
+ \frac{2^\frac43}{1-2^\frac43\delta} \frac{h_1(r)}{(r-\rho)^\frac43}
+ \frac{4}{1-4\delta} \frac{h_2(r)}{(r-\rho)^2}
+ \frac{16}{1-16\delta} \frac{h_3(r)}{(r-\rho)^4}
\Big)\,.
\end{align}
\end{lemma}
%
\begin{proof}
Let us take a sequence $\{a_n\}_{n=0}^{\infty}$ in $[0,1]$ such that 
\begin{align*}
a_0 = \rho\,,\qquad a_n-a_{n-1} = \frac{r-\rho}{2^n}\,, \quad n\in\N\,.
\end{align*}
Then we see that $\lim_{n\to\infty} a_n = r$ and have recursively for $n\in\N$,
\begin{align}\label{est1.appendix.proof.lem.itr.}
\begin{split}
f(a_0) 
& \le \delta f(a_1) 
+ C_\delta \Big(g(a_1) 
+ \frac{h_1(a_1)}{(a_1-a_0)^\frac43} 
+ \frac{h_2(a_1)}{(a_1-a_0)^2} 
+ \frac{h_3(a_1)}{(a_1-a_0)^4} \Big) \\
& \le 
\delta \bigg(\delta f(a_2) 
+ C_\delta \Big( g(a_2) 
+ \frac{h_1(a_2)}{(a_2-a_1)^\frac43} 
+ \frac{h_2(a_2)}{(a_2-a_1)^2} 
+ \frac{h_3(a_2)}{(a_2-a_1)^4} \Big) \bigg) \\
& \quad
+ C_\delta \Big(g(a_1) 
+ \frac{h_1(a_1)}{(a_1-a_0)^\frac43} 
+ \frac{h_2(a_1)}{(a_1-a_0)^2}  
+ \frac{h_3(a_1)}{(a_1-a_0)^4} \Big) \\
& \le \ldots \\
& \le \delta^n f(a_n) 
+ C_\delta 
\big( g(a_1) + \delta g(a_2) + \cdots + \delta^{n-1} g(a_n) \big) \\
& \quad
+ C_\delta
\Big( 
\frac{h_1(a_1)}{(a_1-a_0)^\frac43} + \delta \frac{h_1(a_2)}{(a_2-a_1)^\frac43}
+ \cdots + \delta^{n-1} \frac{h_1(a_n)}{(a_n-a_{n-1})^\frac43} 
\Big) \\
& \quad
+ C_\delta
\Big( 
\frac{h_2(a_1)}{(a_1-a_0)^2} + \delta \frac{h_2(a_2)}{(a_2-a_1)^2} 
+ \cdots + \delta^{n-1} \frac{h_2(a_n)}{(a_n-a_{n-1})^2} 
\Big) \\
& \quad
+ C_\delta
\Big( 
\frac{h_3(a_1)}{(a_1-a_0)^4} + \delta \frac{h_3(a_2)}{(a_2-a_1)^4}
+ \cdots + \delta^{n-1} \frac{h_3(a_n)}{(a_n-a_{n-1})^4} 
\Big) \,.
\end{split}
\end{align}
Then we find
\begin{align}\label{est2.appendix.proof.lem.itr.}
\varlimsup_{n \to \infty} \delta^n f(a_n) = 0\,.
\end{align}
On the other hand, since $g(x)$ is monotone increasing on $[0,1]$, we see that
\begin{align}\label{est3.appendix.proof.lem.itr.}
g(a_1) + \delta g(a_2) + \cdots + \delta^{n-1} g(a_n) 
& \le 
\frac{g(r)}{1-\delta}
\end{align}
uniformly in $n\in\N$. By the same reason as above we have
\begin{align}\label{est4.appendix.proof.lem.itr.}
\begin{split}
& \frac{h_1(a_1)}{(a_1-a_0)^\frac43} + \delta \frac{h_1(a_2)}{(a_2-a_1)^\frac43} 
+ \cdots + \delta^{n-1} \frac{h_1(a_n)}{(a_n-a_{n-1})^\frac43} \\
& \le
2^\frac43 \big(1 + 2^\frac43 \delta + \cdots + (2^\frac43 \delta)^{n-1}\big) \frac{h_1(r)}{(r-\rho)^\frac43} \\
& \le
\frac{2^\frac43}{1-2^\frac43 \delta} \frac{h_1(r)}{(r-\rho)^\frac43}\,.
\end{split}
\end{align}
We also have
\begin{align}
& \frac{h_2(a_1)}{(a_1-a_0)^2} + \delta \frac{h_2(a_2)}{(a_2-a_1)^2} 
+ \cdots + \delta^{n-1} \frac{h_2(a_n)}{(a_n-a_{n-1})^2} 
\le \frac{4}{1-4\delta} \frac{h_2(r)}{(r-\rho)^2}\,, \label{est5.appendix.proof.lem.itr.} \\
& \frac{h_3(a_1)}{(a_1-a_0)^4} + \delta \frac{h_3(a_2)}{(a_2-a_1)^4}
+ \cdots + \delta^{n-1} \frac{h_3(a_n)}{(a_n-a_{n-1})^4}
\le \frac{16}{1-16\delta} \frac{h_3(r)}{(r-\rho)^4}\,. \label{est6.appendix.proof.lem.itr.}
\end{align}
Then \eqref{est.appendix.lem.itr.} is proved by inserting \eqref{est3.appendix.proof.lem.itr.}--\eqref{est6.appendix.proof.lem.itr.} to \eqref{est1.appendix.proof.lem.itr.} and using \eqref{est2.appendix.proof.lem.itr.}. The proof is complete.
\end{proof}
%
\noindent 
We establish the Caccioppoli inequality to the modified Navier-Stokes equations
\begin{equation}\label{appendix.modified.NS.ep}
\left\{
\begin{array}{ll}
-\Delta u^\ep+\nabla p^\ep
= -\nabla\cdot(b^\ep\otimes u^\ep + u^\ep\otimes b^\ep)
- \lambda^\ep (u^\ep\cdot\nabla u^\ep)
+ \nabla\cdot F^\ep&\mbox{in}\ B^\ep_{1,+}(0)\\
\nabla\cdot u^\ep=0&\mbox{in}\ B^\ep_{1,+}(0)\\
u^\ep=0&\mbox{on}\ \Gamma^\ep_1(0)\,.
\end{array}
\right.
\end{equation}
Note that the Stokes equations can be obtained by setting $b^\ep=0$ and $\lambda^\ep=0$.
%
\begin{lemma}\label{appendix.lem.Caccioppoli.ineq.}
Let $\ep\in[0,1)$, $b^\ep \in H^1(B^\ep_{1,+}(0))^3$ with $b^\ep=0$ on $\Gamma^\ep_1(0)$, $\lambda^\ep\in[0,1]$, and $F^\ep\in L^2(B^\ep_{1,+}(0))^{3\times3}$, and let $u^\ep\in H^1(B^\ep_{1,+}(0))^3$ be a weak solution to \eqref{appendix.modified.NS.ep}. Then we have for all $0<\rho<r\le1$,
\begin{align}\label{est.appendix.lem.Caccioppoli.ineq.}
\begin{split}
& \|\nabla u^\ep\|_{L^2(B^\ep_{\rho,+}(0))}^2 \\
& \le 
K_0\bigg(
\frac{1}{(r-\rho)^2} \|u^\ep\|_{L^2(B^\ep_{r,+}(0))}^2 \\
& \qquad\quad
+ \Big(
\|\nabla b^\ep\|_{L^2(B^\ep_{r,+}(0))}^4 
+ \frac{\|\nabla b^\ep\|_{L^2(B^\ep_{r,+}(0))}^\frac43}{(r-\rho)^\frac43}
\Big)
\|u^\ep\|_{L^2(B^\ep_{r,+}(0))}^2 \\
& \qquad\quad
+ \frac{(\lambda^\ep)^4}{(r-\rho)^4} \|u^\ep\|_{L^2(B^\ep_{r,+}(0))}^6
+ \|F^\ep\|_{L^2(B^\ep_{r,+}(0))}^2
\bigg)\,,
\end{split}
\end{align}
where the constant $K_0$ depends only on $\|\gamma\|_{W^{1,\infty}(\R^2)}$. In particular it is independent of $\ep$, $b^\ep$, $\lambda^\ep$, $\rho$, and $r$.
\end{lemma}
%
\begin{proof}
\emph{In this proof the norm $\|\cdot\|_{L^p(B^\ep_{r,+}(0))}$ is denoted by $\|\cdot\|_{L^p}$ for simplicity.} Fix $0<\rho<r\le1$. We extend $u^\ep$ and $b^\ep$ by zero below the boundary, which are respectively denoted by $u^\ep$ and $b^\ep$ again, and see that $u^\ep$ and $b^\ep$ are in $H^1(B_1(0))^3$. By taking a cut-off function such that
\begin{align*}
\supp\varphi\subset B_r(0)\,, 
\qquad \varphi(x)=1 \ \, {\rm in} \ \, \overline{B_r(0)}\,,
\qquad \|\nabla \varphi\|_{L^\infty(B_{1}(0))} \le \frac{C}{r-\rho}\,,
\end{align*}
we test $u^\ep \varphi^2$ against the equations \eqref{appendix.modified.NS.ep}. Then we find that
\begin{align}\label{eq1.proof.appendix.lem.Caccioppoli.ineq.}
\begin{split}
\int_{B^\ep_{1,+}(0)}
(\nabla\cdot F^\ep) \cdot u^\ep \varphi^2
& = 
-\int_{B^\ep_{1,+}(0)} \Delta u^\ep \cdot u^\ep \varphi^2
+ \int_{B^\ep_{1,+}(0)} \nabla (p^\ep - [p^\ep]_r ) \cdot u^\ep \varphi^2 \\
& \quad
+ \int_{B^\ep_{1,+}(0)} 
\big(\nabla\cdot(b^\ep\otimes u^\ep + u^\ep\otimes b^\ep)\big)\cdot u^\ep \varphi^2 \\
& \quad
+ \lambda \int_{B^\ep_{1,+}(0)}
(u^\ep\cdot\nabla u^\ep) \cdot u^\ep \varphi^2\,,
\end{split}
\end{align}
where we have set $[p^\ep]_r = (\overline{p^\ep})_{B^\ep_{r,+}(0)}$. Then, by integration by parts and $\nabla\cdot u^\ep=0$, one has
\begin{align}\label{eq2.proof.appendix.lem.Caccioppoli.ineq.}
\begin{split}
& -\int_{B^\ep_{1,+}(0)} \Delta u^\ep \cdot u^\ep \varphi^2
+ \int_{B^\ep_{1,+}(0)} \nabla (p^\ep - [p^\ep]_r) \cdot u^\ep \varphi^2 \\
& = 
\int_{B^\ep_{1,+}(0)} |\nabla u^\ep|^2 \varphi^2
+ 2 \int_{B^\ep_{1,+}(0)} \nabla u^\ep \cdot (u^\ep \nabla \varphi) \varphi 
- 2 \int_{B^\ep_{1,+}(0)} (p^\ep - [p^\ep]_r) (u^\ep\cdot\nabla \varphi) \varphi
\end{split}
\end{align}
and
\begin{align}\label{eq3.proof.appendix.lem.Caccioppoli.ineq.}
\begin{split}
& \int_{B^\ep_{1,+}(0)} 
\big(\nabla\cdot(b^\ep\otimes u^\ep + u^\ep\otimes b^\ep)\big)\cdot u^\ep \varphi^2 \\
& = 
-\int_{B^\ep_{1,+}(0)} 
(b^\ep\otimes u^\ep + u^\ep\otimes b^\ep) \cdot (\nabla u^\ep) \varphi^2 \\
&\quad
-2\int_{B^\ep_{1,+}(0)} 
(b^\ep\otimes u^\ep + u^\ep\otimes b^\ep) \cdot (u^\ep \nabla \varphi) \varphi\,.
\end{split}
\end{align}
The nonlinearity is calculated as
\begin{align}\label{eq4.proof.appendix.lem.Caccioppoli.ineq.}
\begin{split}
\int_{B^\ep_{1,+}(0)}
(u^\ep\cdot\nabla u^\ep) \cdot u^\ep \varphi^2
& = \frac12 \int_{B^\ep_{1,+}(0)} \nabla|u^\ep|^2 \cdot u^\ep \varphi^2 \\
& = -\int_{B^\ep_{1,+}(0)} |u^\ep|^2 (u^\ep\cdot\nabla \varphi) \varphi\,.
\end{split}
\end{align}
By using $u^\ep\cdot\nabla u^\ep=\nabla\cdot(u^\ep\otimes u^\ep)$, the pressure $\|p^\ep - [p^\ep]_r \|_{L^2}$ is estimated as
\begin{align}\label{est1.proof.appendix.lem.Caccioppoli.ineq.}
\begin{split}
\|p^\ep - [p^\ep]_r \|_{L^2}
& \le
C \|\nabla p^\ep\|_{W^{-1,2}(B^\ep_{r,+}(0))} \\
& = C 
\|\nabla\cdot \big(\nabla u^\ep 
- (b^\ep\otimes u^\ep + u^\ep\otimes b^\ep)
- \lambda u^\ep\otimes u^\ep + F^\ep \big)\|_{W^{-1,2}(B^\ep_{r,+}(0))} \\
& \le C 
\big( \|\nabla u^\ep\|_{L^2} 
+ \|b^\ep\otimes u^\ep + u^\ep\otimes b^\ep\|_{L^2} 
+ \lambda \|u^\ep\|_{L^4}^2 + \|F^\ep\|_{L^2} \big)\,.
\end{split}
\end{align}
After inserting \eqref{eq2.proof.appendix.lem.Caccioppoli.ineq.}--\eqref{eq4.proof.appendix.lem.Caccioppoli.ineq.} to \eqref{eq1.proof.appendix.lem.Caccioppoli.ineq.} and using \eqref{est1.proof.appendix.lem.Caccioppoli.ineq.}, we see that from the H\"older inequality,
\begin{align}\label{est2.proof.appendix.lem.Caccioppoli.ineq.}
\begin{split}
\int_{B^\ep_{\rho,+}(0)} |\nabla u^\ep|^2 \varphi^2 
& \le 
\frac{C}{r-\rho} \|u^\ep\|_{L^2} \|\nabla u^\ep\|_{L^2} \\
& \quad
+ C \|b^\ep\otimes u^\ep + u^\ep\otimes b^\ep\|_{L^2}
\Big( \|\nabla u^\ep\|_{L^2} 
+ \frac{\|u^\ep\|_{L^2}}{r-\rho} \Big) \\
& \quad
+ \frac{C\lambda^\ep}{r-\rho}
\|u^\ep\|_{L^4}^2 \|u^\ep\|_{L^2} 
+ C \|F^\ep\|_{L^2}
\Big( \|\nabla u^\ep\|_{L^2} + \frac{\|u^\ep\|_{L^2}}{r-\rho} \Big)\,,
\end{split}
\end{align}
where $C$ depends only on $\|\gamma\|_{W^{1,\infty}(\R^2)}$. By the H\"older and Sobolev inequalities we have
\begin{align}\label{est3.proof.appendix.lem.Caccioppoli.ineq.}
\begin{split}
\|b^\ep\otimes u^\ep + u^\ep\otimes b^\ep\|_{L^2}
& \le \|b^\ep\|_{L^6} \|u^\ep\|_{L^3} \\
& \le \|b^\ep\|_{L^6} \|u^\ep\|_{L^2}^\frac12 \|u^\ep\|_{L^6}^\frac12 \\
& \le C \|\nabla b^\ep\|_{L^2(B^\ep_{1,+}(0))} \|u^\ep\|_{L^2}^\frac12 \|\nabla u^\ep\|_{L^2}^\frac12
\end{split}
\end{align}
and
\begin{align}\label{est4.proof.appendix.lem.Caccioppoli.ineq.}
\begin{split}
\|u^\ep\|_{L^4}^2 
& \le \|u^\ep\|_{L^2}^\frac12 \|u^\ep\|_{L^6}^\frac32 \\
& \le C \|u^\ep\|_{L^2}^\frac12 \|\nabla u^\ep\|_{L^2}^\frac32\,.
\end{split}
\end{align}
Therefore, by inserting \eqref{est3.proof.appendix.lem.Caccioppoli.ineq.} and \eqref{est4.proof.appendix.lem.Caccioppoli.ineq.} into \eqref{est2.proof.appendix.lem.Caccioppoli.ineq.},  we find
\begin{align*}
\begin{split}
\|\nabla u^\ep\|_{L^2(B^\ep_{\rho,+}(0))}^2 
& \le 
\frac{C}{r-\rho} \|\nabla u^\ep\|_{L^2} \|u^\ep\|_{L^2} \\
& \quad
+ C \|\nabla b^\ep\|_{L^2(B^\ep_{1,+}(0))} 
\Big(
\|u^\ep\|_{L^2}^\frac12 \|\nabla u^\ep\|_{L^2}^\frac32
+ \frac{\|u^\ep\|_{L^2}^\frac32 \|\nabla u^\ep\|_{L^2}^\frac12}{r-\rho} 
\Big) \\
& \quad
+ \frac{C\lambda^\ep}{r-\rho}
\|u^\ep\|_{L^2}^\frac32 \|\nabla u^\ep\|_{L^2}^\frac32
+ C \|F^\ep\|_{L^2(B^\ep_{1,+}(0))}
\Big( \|\nabla u^\ep\|_{L^2} + \frac{\|u^\ep\|_{L^2}}{r-\rho} \Big)\,,
\end{split}
\end{align*}
and thus by applying the Young inequality, we have for any $\delta\in(0,\frac{1}{16})$,
\begin{align*}
\begin{split}
& \|\nabla u^\ep\|_{L^2(B^\ep_{\rho,+}(0))}^2 \\
& \le 
\delta \|\nabla u^\ep\|_{L^2}^2
+ C_\delta 
\Big( 
\frac{\|u^\ep\|_{L^2}^2}{(r-\rho)^2} 
+ \|\nabla b^\ep\|_{L^2(B^\ep_{1,+}(0))}^4 \|u^\ep\|_{L^2}^2
+ \frac{\|\nabla b^\ep\|_{L^2(B^\ep_{1,+}(0))}^\frac43}{(r-\rho)^\frac43}
\|u^\ep\|_{L^2}^2 \\
& \qquad\qquad\qquad\qquad \ 
+ \frac{(\lambda^\ep)^4}{(r-\rho)^4} \|u^\ep\|_{L^2}^6
+ \|F^\ep\|_{L^2(B^\ep_{1,+}(0))}^2
\Big)\,,
\end{split}
\end{align*}
where $C_\delta$ depends on $\delta$. The assertion \eqref{est.appendix.lem.Caccioppoli.ineq.} follows from Lemma \ref{appendix.lem.itr.}. This completes the proof.
\end{proof}
%

\subsection*{Acknowledgement}
M. H. is grateful to Universit\'e de Bordeaux for their kind hospitality during his stay as a postdoctoral researcher in the spring semester of 2019.

\subsection*{Funding and conflict of interest.} 
The authors acknow\-ledge financial support from the IDEX of the University of Bordeaux for the BOLIDE project. The second author is also partially supported by the project BORDS grant ANR-16-CE40-0027-01 and by the project SingFlows grant ANR-18-CE40-0027 of the French National Research Agency (ANR). 
The authors declare that they have no conflict of interest. 

\small
\bibliographystyle{abbrv}
\bibliography{nsbumpy}

\end{document}